\documentclass[preprint,3p,times,nopreprintline]{elsarticle}
\pdfoutput=1

\usepackage{framed,multirow}

\usepackage{mathtools}
\usepackage{amsfonts}
\usepackage{amssymb}
\usepackage{amsbsy}
\usepackage{subfig,float,subfloat}
\usepackage{amsthm,bm}
\usepackage[normalem]{ulem}
\usepackage{commath}
\usepackage[labelfont=bf]{caption}
\usepackage{listings}
\usepackage{multirow} 
\usepackage{tikz}
\usepackage{tabulary}
\usepackage{lineno}
\usepackage[linkcolor=blue,colorlinks=true]{hyperref}

\usepackage{natbib}

\newcommand{\Cv}{\mathrm{C_{v}}}
\newcommand{\Cp}{\mathrm{C_{p}}}
\newcommand{\pI}{\mathrm{p_I}}
\newcommand{\uI}{\mathrm{u_I}}
\newcommand{\vI}{\mathrm{v_I}}
\newcommand{\vecu}{\normalfont{\textbf{u}}}

\newcommand{\nc}{\normalfont{\textbf{d}}}
\newcommand{\D}{\normalfont{\textbf{D}}}
\newcommand{\V}{\normalfont{\textbf{v}}}
\newcommand{\n}{\normalfont{\textbf{n}}}

\newcommand{\BNM}{\scalebox{.8}{$\scriptscriptstyle\rm BNM$}}
\newcommand{\lb}{\llbracket}
\newcommand{\rb}{\rrbracket}

\newcommand{\press}{\mathrm{p}}
\newtheorem{remark}{\textbf{Remark}}[section]
\newtheorem{prop}{\textbf{Proposition}}[section]
\newtheorem{mythm}{\textbf{Theorem}}[section]
\newtheorem*{leib*}{\textit{Leibniz's rule}}

\newtheorem{definition}{Definition}[section]

\newcommand{\tD}{\Tilde{\D}}
\renewcommand{\norm}[1]{\left\lVert#1\right\rVert}

\makeatletter
\newsavebox\myboxA
\newsavebox\myboxB
\newlength\mylenA
\newcommand*\xoverline[2][1.0]{%
    \sbox{\myboxA}{$\m@th#2$}%
    \setbox\myboxB\null
    \ht\myboxB=\ht\myboxA%
    \dp\myboxB=\dp\myboxA%
    \wd\myboxB=#1\wd\myboxA
    \sbox\myboxB{$\m@th\overline{\copy\myboxB}$}
    \setlength\mylenA{\the\wd\myboxA}
    \addtolength\mylenA{-\the\wd\myboxB}%
    \ifdim\wd\myboxB<\wd\myboxA%
       \rlap{\hskip 0.5\mylenA\usebox\myboxB}{\footnotesize{\usebox\myboxA}}%
    \else
        \hskip -0.5\mylenA\rlap{\usebox\myboxA}{\footnotesize{\hskip 0.5\mylenA\usebox\myboxB}}%
    \fi}
\makeatother

\newcommand{\equaltext}[1]{\ensuremath{\stackrel{\text{#1}}{=}}}
\usepackage{latexsym}
\usepackage{cleveref}
\crefname{mythm}{Theorem}{Theorems}
\crefformat{mythm}{Theorem~#2#1#3}
\crefname{remark}{Remark}{Remarks}
\crefname{lemma}{Lemma}{Lemmas}
\crefname{prop}{Proposition}{Propositions}
\crefformat{prop}{#2#1#3}
\crefname{table}{Table}{Tables}
\crefrangeformat{table}{Tables~#3#1#4 to~#5#2#6}
\crefname{figure}{Figure}{Figures}
\crefrangeformat{figure}{Figures~#3#1#4 to~#5#2#6}
\crefformat{equation}{#2#1#3}
\crefformat{appendix}{Appendix~#2#1#3}

\usepackage{url}
\usepackage{xcolor}
\definecolor{newcolor}{rgb}{.8,.349,.1}

\title{An entropy stable high-order discontinuous Galerkin spectral element method for the Baer-Nunziato two-phase flow model}

\author[1]{Fr\'ed\'eric Coquel}
\cortext[cor1]{Corresponding author's email address: pratik.rai@onera.fr (Pratik Rai).}
\author[2]{Claude Marmignon}
\author[1,2,]{Pratik Rai,\corref{cor1}} 
\author[2]{Florent Renac}

\address[1]{CMAP, \'Ecole Polytechnique, Route de Saclay, 91128 Palaiseau Cedex, France}
\address[2]{DAAA, ONERA, Universit\'e Paris Saclay F-92322 Ch\^atillon, France}

\providecommand{\keywords}[1]
{
  \small	
  \textbf{Keywords. } #1
}

\providecommand{\AMS}[1]
{
  \small	
  \textbf{AMS subject classifications. } #1
}

\begin{document}

\maketitle
\vspace{-2cm}
\begin{abstract}
In this work we propose a high-order discretization of the Baer-Nunziato two-phase flow model (Baer and Nunziato, Int. J. Multiphase Flow, 12 (1986), pp. 861-889) with closures for interface velocity and pressure adapted to the treatment of discontinuous solutions, and stiffened gas equations of states. We use the discontinuous Galerkin spectral element method (DGSEM), based on collocation of quadrature and interpolation points (Kopriva and Gassner, J. Sci. Comput., 44 (2010), pp. 136-155). The DGSEM uses summation-by-parts (SBP) operators in the numerical quadrature for approximating the integrals over discretization elements (Carpenter et al., SIAM J. Sci. Comput., 36 (2014), pp. B835-B867; Gassner et al., J. Comput. Phys., 327 (2016), pp. 39-66). Here, we build upon the framework provided in (F. Renac, J. Comput. Phys., 382 (2019), pp. 1-36) for nonconservative hyperbolic systems to modify the integration over cell elements using the SBP operators and replace the physical fluxes with entropy conservative fluctuation fluxes from Castro et al. (SIAM J. Numer. Anal., 51 (2013), pp. 1371-1391), while we derive entropy stable fluxes applied at interfaces. This allows to establish a semi-discrete inequality for the cell-averaged physical entropy, while being high-order accurate. The design of the numerical fluxes also formally preserves the kinetic energy at the semi-discrete level. High-order integration in time is performed using strong stability-preserving Runge-Kutta schemes and we propose conditions on the numerical parameters for the positivity of the cell-averaged void fraction and partial densities. The positivity of the cell-averaged solution is extended to nodal values by the use of an a posteriori limiter. The high-order accuracy, nonlinear stability, and robustness of the present scheme are assessed through several numerical experiments in one and two space dimensions.
\end{abstract}

\\

\AMS{65M12, 65M70, 76T10}


%
%
\section{Introduction}
\label{Introduction}
Compressible two-phase flow models find extensive applications in engineering and physics. For instance, in the aerospace industry, they are used to model the flow of a mixture of liquid kerosene and air through the combustion chamber of jet engines, whereas, in the oil and gas industry they are used to model and simulate the extraction of oil through pipelines. Elsewhere, in the nuclear industry these models are used to study and simulate the flow inside a pressurized water reactor. One of the models commonly employed for the study of compressible two-phase flows is the  Baer-Nunziato model \cite{baer1986two}, which was originally proposed to describe the flow of a mixture of energetic granular material embedded in gaseous combustion product. This was later modified and adapted to the study of mixture of gas and liquid in \cite{saurel1999multiphase,coquel2002closure,drew2006theory,gallouet2004numerical}. In general, the model is a two-velocity, two-pressure, two-temperature system that describes two-phase flows in complete disequilibrium with respect to the chemical, mechanical, thermal, and thermodynamic processes. The interaction between the phases are governed by the presence of nonconservative products and zeroth order relaxation source terms. In this work we will neglect the source terms and limit ourselves to the convective part of the model. However, the homogeneous model under consideration is fairly general using closure laws for the interface velocity and pressure \cite{coquel2002closure,gallouet2004numerical} as well as stiffened gas equations of states (EOS) relevant for flows with both liquid and gas phases.

The homogeneous Baer-Nunziato model is a system of first order, nonlinear, nonconservative partial differential equations. The system is hyperbolic and may become weakly hyperbolic and even resonant. Hyperbolic systems may generate discontinuous solutions in finite time even for smooth initial data, however, in the case of nonconservative systems, the definition of the nonconservative product is not unique at discontinuities in the classical sense of distributions and leads to an ambiguity in the value of the product. Following the notion of the Rankine-Hugoniot conditions from conservation laws, the jump conditions for nonconservative systems may be generalized and may be either based on the choice of Lipschitz paths connecting separate states around discontinuities \cite{dal1995definition}, or based on the kinetic relations derived from the physical entropy \cite{berthon2012many}. Furthermore, uniqueness of the solution requires satisfying a nonlinear stability condition, for a given convex entropy function, called the entropy condition.

Numerical schemes that approximate hyperbolic systems should ideally recover admissible solutions by satisfying a discrete entropy condition \cite{lax1960systems,harten1983upstream}. This property of the numerical scheme is known as entropy stability. In the case of conservation laws, Tadmor \cite{tadmor1987numerical} provided the framework for entropy conservative and entropy stable numerical fluxes which allow for either conservation or dissipation  of entropy in space by three-point finite volume schemes. This was extended to nonconservative systems in \cite{pares2006numerical,castro2013entropy} by the use of fluctuation fluxes and the theory of connecting paths \cite{dal1995definition}. However, path-consistent schemes do not always converge to the right admissible solutions as the solutions are dependent on the choice of path which defines the jump relation and hence the viscous profile used to attain entropy stability \cite{abgrall2010comment,castro2008many,CHALONS2017592}. Entropy stable schemes using fluctuation fluxes to discretize nonconservative hyperbolic systems can be found in \cite{hiltebrand2018entropy,castro2017well,renac2019entropy} and we refer to \cite{lefloch2014numerical} for a review.

High-order accuracy of the numerical scheme is another exceedingly desirable quality that one seeks. The path-consistent framework can easily be adapted to high-order schemes by using reconstruction operators \cite{castro2006high}, central schemes \cite{castro2012central}, discontinuous Galerkin (DG) methods \cite{rhebergen2008discontinuous,fraysse2016upwind,chen2017entropy}, or ADER methods \cite{dumbser2009ader,dumbser2013high}. Among these the DG methods have gained substantial popularity over the years. The semi-discrete form of the DG method is proven to satisfy an entropy inequality for square entropy functions in scalar conservation laws \cite{jiang1994cell}, which was extended to symmetric systems in \cite{hou2007solutions}. 

In \cite{gassner2016split}, Gassner and coauthors have proposed an entropy stable high-order scheme for the compressible Euler equations using the discontinuous Galerkin spectral elements method (DGSEM), which was extended to general conservation laws in \cite{chen2017entropy}. They used the general framework for conservative elementwise flux differencing schemes \cite{fisher2013high} satisfying a semi-discrete entropy inequality for the cell-averaged entropy. The DGSEM is based on collocation of quadrature nodes with interpolation points using the Gauss-Lobatto quadrature rules \cite{kopriva2010quadrature}. The scheme was shown to satisfy the summation-by-parts (SBP) property \cite{gassner2013SBP-SAT} for the discrete operators which allows to take into account the numerical quadrature that approximates integrals compared to other techniques that require their exact evaluation \cite{jiang1994cell,hiltebrand2014entropy,hiltebrand2018entropy}. Such a form of the nodal DG method has found tremendous use in the development of entropy stable high-order schemes for the compressible Euler equations \cite{gassner2016split,chen2017entropy} and multicomponent Euler equations \cite{renac2020multicomp}, the shallow water equations \cite{wintermeyer2017entropy}, the magnetohydrodynamic (MHD) equations \cite{liu2018entropy,bohm2018entropy,winters2016affordable} and gradient flows \cite{sun2018entropy}. In the case of nonconservative systems, a semi-discrete framework was proposed in \cite{renac2019entropy} based on the DGSEM formulation that proves to be entropy stable and high-order accurate.

In the present work we utilize the framework from \cite{renac2019entropy} and focus on the design of a high-order entropy stable scheme for the Baer-Nunziato model. This framework is here extended to systems that contain both space derivatives in divergence form and nonconservative products, which is based on a direct generalization of the frameworks of entropy stable finite volume schemes for conservation laws \cite{tadmor1987numerical} and for nonconservative systems \cite{castro2013entropy}. Such generalization has already been proposed for balance laws in \cite{castro2013entropy}. This generalization allows the design of discretizations that reduce to conservative schemes using conservative numerical fluxes when the nonconservative products vanish as it is the case away from material fronts in the Baer-Nunziato model. Using this framework, we modify the integration over cell elements using the SBP operator and replace the physical fluxes with two-point entropy conservative fluxes in fluctuation form \cite{castro2013entropy}, while we use entropy stable fluxes at the cell interfaces \cite{castro2013entropy,renac2019entropy}. The entropy conservative fluxes are derived by using the entropy condition \cite{castro2013entropy}, and we add upwind-type dissipation as advocated in \cite{ismail2009affordable} to obtain the entropy stable numerical fluxes. The scheme is also kinetic energy preserving at the semi-discrete level. The present method is introduced in one space dimension for the sake of clarity and we provide details on its extension to multiple space dimensions on Cartesian meshes in the appendices. The extension of the DGSEM to quadrangles and hexahedra is direct and based on tensor products of one-dimensional basis functions and quadrature rules.

We then focus on high-order integration in time for which we rely on strong stability-preserving explicit Runge-Kutta schemes \cite{shu1988efficient,gottlieb2001strong} which are defined as convex combinations of first-order schemes and keep their properties under some condition on the time step. We analyze the properties of the fully discrete one-step scheme and derive explicit conditions on the time step and numerical parameters to maintain the positivity of the cell-averaged partial densities and a maximum principle for the cell-averaged void fraction. Positivity of the solution is then enforced at nodal values by the use of a posteriori limiters \cite{zhang2010positivity, zhang2010maximum}. Numerical tests in one and two space dimensions are finally performed to assess the properties of the present scheme. 

The plan of the paper is as follows. \Cref{Sec: Baer-Nunziato model} describes the Baer-Nunziato model and highlights its physical and mathematical properties. In \cref{sec: DGSEM framework}, we introduce the DGSEM framework and the semi-discrete scheme. The derivation of entropy conservative and entropy stable numerical fluxes are outlined in \cref{sec: fluctuation fluxes}. The properties of the scheme and the limiters are described in \cref{sec: properties of the scheme}. The results of the numerical experiments in one space dimension are presented in \cref{sec: Numerical tests}, while those in two space dimensions are presented in \cref{sec: extension to multi-space dim}. Finally, concluding remarks on the present work are provided in \cref{sec: Conclusion}.

%
%
\section{Baer-Nunziato model}\label{Sec: Baer-Nunziato model}

We consider the Cauchy problem for the homogeneous Baer-Nunziato two-phase flow model in one space dimension \cite{andrianov2004riemann,dumbser2011simple,tokareva2010hllc,ambroso2012godunov}:
\begin{linenomath*}
\begin{subequations} \label{Eqn: Cauchy prob}
\begin{alignat}{2}
    \partial_t\vecu + \partial_x\textbf{f}(\vecu) +\textbf{c}(\vecu)\partial_x\vecu &= 0, \quad && x\in\mathbb{R},t>0,\label{Eqn: BNM} \\
    \vecu(x,0) &= \vecu_0(x), \label{Eqn: ic_bnm}\quad && x \in \mathbb{R},
\end{alignat}
\end{subequations}
\end{linenomath*}
where
\begin{linenomath*}
\begin{equation} \label{Eqn: BNM vectors}
    \vecu := 
    \begin{pmatrix}
    \alpha_i\\
    \alpha_i\rho_i\\
    \alpha_i\rho_iu_i\\
    \alpha_i\rho_iE_i
    \end{pmatrix}, \quad
    \textbf{f}(\vecu) := 
    \begin{pmatrix}
    0\\
    \alpha_i\rho_iu_i\\
    \alpha_i(\rho_iu^2_i+\mathrm{p}_i)\\
    \alpha_iu_i(\rho_iE_i+\mathrm{p}_i)
    \end{pmatrix},\quad
    \textbf{c}(\vecu)\partial_x\vecu :=
    \begin{pmatrix}
    \uI\\
    0\\
    -\pI\\
    -\pI\uI
    \end{pmatrix}\partial_x\alpha_i, \quad i=1,2.
\end{equation}
\end{linenomath*}

The phase densities are $\rho_i$, the velocities are $u_i$, and the specific total energies are $E_i=e_i + u^2_i/2$ where $e_i$ is the specific internal energy and $i=1,2$ refers to the $i$th phase. The void fraction of each individual phase is denoted as $\alpha_i$ and we assume that both satisfy the saturation condition 
\begin{linenomath*}
\begin{equation}\label{Eqn: void fraction assumption}
    \alpha_1+\alpha_2=1.
\end{equation}
\end{linenomath*}

In one space dimension, the model is a system of seven equations including the evolution equations for the mass, momentum and energy of each phase, along with a transport  equation for the void fraction. Note that the vector of unknowns, $\vecu$, has seven components even though the abstract notation in (\cref{Eqn: BNM vectors}) seems to indicate eight equations. This is due to the saturation condition (\cref{Eqn: void fraction assumption}) which allows the void fraction equation to be expressed in terms of either phases. The solution $\vecu$ belongs to the phase space
\begin{linenomath*}
\begin{equation}\label{eqn: phase_space}
     \Omega_{\BNM} = \big\{\vecu\in\mathbb{R}^7:\, 0<\alpha_i<1,\, \rho_i>0,\, u_i\in\mathbb{R},\, e_i>0,\, i=1,2 \big\}.
\end{equation}
\end{linenomath*}

Space variations of the physical quantities are governed by the flux function $\textbf{f}(\vecu):\Omega_{\BNM}\rightarrow\mathbb{R}^7$ and the nonconservative product $\textbf{c}(\vecu)\partial_x\vecu$ which couples the phases and hinders the system (\cref{Eqn: BNM}) to be written in divergence form. Note that if $\alpha_i$ is uniform in space, the phases decouple into separate systems of compressible Euler equations. 

The pressure of each phase $\press_i$ is related to the density and internal energy through a stiffened gas EOS:
\begin{linenomath*}
\begin{equation}\label{Eqn: EOS}
    \mathrm{p}_i(\rho_i,e_i)=(\gamma_i-1)\rho_ie_i-\gamma_i\mathrm{p}_{\infty,i},
\end{equation}
\end{linenomath*}
where $\gamma_i=\Cp_i/\Cv_i>1$ is the ratio of specific heats of phase $i$ and $\mathrm{p}_{\infty,i} \geqslant 0$ are some constants. System (\cref{Eqn: BNM}) is supplemented with closure laws for the interfacial velocity and pressure, $\uI$ and $\pI$, respectively, that govern the exchange of information at the interface of the two phases. In this work, we use definitions of the interfacial velocity and pressure based on convex combinations of the velocities and pressures of the two phases \cite{coquel2002closure,gallouet2004numerical} and adapted to the treatment of discontinuous solutions:
\begin{linenomath*}
\begin{subequations}\label{eqn: intref_var}
\begin{align}
    \uI &:= \beta u_1 + (1-\beta) u_2, \label{eqn: interf. vel}\\
    \pI &:= \mu \press_1 + (1-\mu)\press_2, \label{eqn: interf. press}
\end{align}
\end{subequations}
\end{linenomath*}
where the convex weights are
\begin{linenomath*}
\begin{equation}\label{eqn: intref_var convex weights}
    \beta = \frac{\chi \alpha_1\rho_1}{\chi \alpha_1\rho_1 + (1-\chi)\alpha_2\rho_2},\quad
    \mu = \frac{(1-\beta)T_2}{\beta T_1 + (1-\beta)T_2}, \quad \chi \in \{0,\tfrac{1}{2}, 1\},
\end{equation}
\end{linenomath*}
and $T_i$ denotes the temperature of the $i$th phase.

Under this particular choice for the closures (\cref{eqn: intref_var}) and (\cref{eqn: intref_var convex weights}), the interfacial velocity $\uI$ corresponds to an eigenvalue for (\cref{Eqn: Cauchy prob}) associated to a linearly degenerate (LD) characteristic field which allows to close the jump relation across material interfaces that are associated to the nonconservative terms in (\cref{Eqn: BNM}). The possible choices for $\chi$ in (\cref{eqn: intref_var convex weights}) are the ones that allow to obtain a conservative equation for the physical entropy for smooth solutions. Physical systems such as the Baer-Nunziato model are indeed naturally equipped with a physical entropy function. The phasic entropies read 
\begin{linenomath*}
\begin{equation}\label{eqn: physical entropy}
    s_i(\rho_i,\theta_i) =-\Cv_i\big(\ln\theta_i+(\gamma_i-1)\ln \rho_i\big), \quad i=1,2,
\end{equation}
\end{linenomath*}
with $\theta_i=\tfrac{1}{T_i}$ the inverse of temperature, and obey the second law of thermodynamics. Smooth solutions of (\cref{Eqn: Cauchy prob}) satisfy
\begin{linenomath*}
\begin{equation}\label{eqn: interf press entropy condition}
    \partial_t \sum^2_{i=1} \alpha_i\rho_is_i + \partial_x \sum^2_{i=1} \alpha_i\rho_is_iu_i = \sum^2_{i=1}(\pI-\press_i)(\uI-u_i)\theta_i \partial_x \alpha_i,
\end{equation}
\end{linenomath*}
which indeed vanishes for the closure of interfacial quantities (\cref{eqn: intref_var}) and (\cref{eqn: intref_var convex weights}):
\begin{linenomath*}
\begin{equation}\label{eqn: partialx alphai in entropy equality}
    \sum^2_{i=1}(\pI-\press_i)(\uI-u_i)\theta_i \partial_x \alpha_i = 0.
\end{equation}
\end{linenomath*}
In the case of non-smooth solutions, such as shocks, admissible weak solutions must satisfy a nonlinear stability condition for the convex entropy function $\eta(\vecu):=-\sum^2_{i=1}\alpha_i\rho_is_i$ and entropy flux $q(\vecu):=-\sum^2_{i=1}\alpha_iu_i\rho_is_i$:
\begin{linenomath*}
\begin{equation}\label{Eqn: entropy ineq}
    \partial_t\eta(\vecu)+\partial_xq(\vecu) \leqslant 0.
\end{equation}
\end{linenomath*}

System (\cref{Eqn: BNM}) can also be written in quasi-linear form as
\begin{linenomath*}
\begin{equation}\label{Eqn: compact BNM}
    \partial_t\vecu+\textbf{A}(\vecu)\partial_x\vecu = 0, \quad x\in\mathbb{R},t>0,
\end{equation}
\end{linenomath*}
where $\textbf{A}:\Omega_{\BNM}\ni\vecu\mapsto\textbf{A}(\vecu)=\textbf{f}'(\vecu)+\textbf{c}(\vecu)\in\mathbb{R}^{7\times7}$ is a matrix-valued function for smooth solutions of (\cref{Eqn: Cauchy prob}). The system (\cref{Eqn: compact BNM}) is hyperbolic over the phase space (\cref{eqn: phase_space}) and $\textbf{A(\vecu)}$ admits real eigenvalues 
\begin{linenomath*}
\begin{equation}\label{Eqn: eigenvalues}
        \begin{aligned}
            \lambda_1(\vecu) = u_1-c_1,\; \lambda_2(\vecu) = u_2-c_2,\; \lambda_3(\vecu) = u_1,\; \lambda_4(\vecu) = \uI,\; \lambda_5(\vecu) = u_2,\; \lambda_6(\vecu) = u_1+c_1,\; \lambda_7(\vecu) = u_2+c_2,
        \end{aligned}
    \end{equation}
\end{linenomath*}
associated to linearly independent eigenvectors. Here $c_i(\rho_i,e_i)^2=\gamma_i(\gamma_i-1)(\rho_ie_i-p_{\infty,i})/\rho_i$ is the speed of sound for the EOS (\cref{Eqn: EOS}). Observe, in (\cref{Eqn: eigenvalues}), that $\lambda_3,\lambda_4$ and $\lambda_5$ are associated to LD fields, whereas the others ones, $\lambda_1,\lambda_2,\lambda_6$ and $\lambda_7$, are associated to genuinely nonlinear (GNL) fields. Note that (\cref{Eqn: compact BNM}) is only weakly hyperbolic when $\uI$ is equal to one transport velocity, $u_1$ or $u_2$, for $\chi = 1$ or $0$ in (\cref{eqn: intref_var convex weights}). In this work we assume that (\cref{Eqn: compact BNM}) is hyperbolic and well-posed and exclude resonance phenomena \cite{coquel2017positive}:
\begin{linenomath*}
\begin{equation}\label{eqn: resonance}
    \alpha_i \neq 0, \quad \uI \neq u_i\pm c_i, \quad i = 1,2.
\end{equation}
\end{linenomath*}
When resonance occurs, the system turns degenerate as the right eigenvectors no longer span the whole phase space (\cref{eqn: phase_space}).

During the remaining course of this work we will be interested in discretizing the initial value problem (\cref{Eqn: Cauchy prob}). We will discretize the system in space using the DGSEM framework from \cite{renac2019entropy} and propose numerical fluxes that maintain the nonlinear stability condition (\cref{Eqn: entropy ineq}) at the semi-discrete level in addition to several other properties.

%
%
\section{Space discretization with the DGSEM} \label{sec: DGSEM framework}
We discretize the physical domain using a grid $\Omega_h:=\cup_{j\in\mathbb{Z}}\kappa_j$ containing cells $\kappa_j=[x_{j-\frac{1}{2}},x_{j+\frac{1}{2}}]$, $x_{j+\frac{1}{2}}=jh$ with cell size $h>0$, see \cref{mesh}. Here the mesh is assumed to be uniform without loss of generality.

\begin{figure}
    \centering
    \includegraphics[scale=0.35]{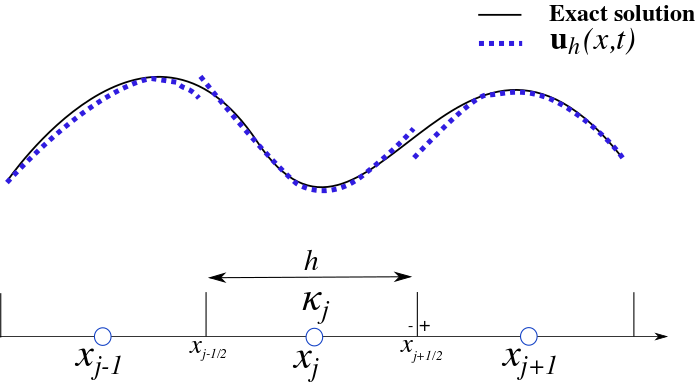}
    \caption{A one-dimensional representation of the mesh with cells $\kappa_j$ of size $h$. The left and right interfaces of cell $\kappa_j$ are at $x_{j\mp\frac{1}{2}}$, and representation of the left and right traces at $x_{j+\frac{1}{2}}$.}
    \label{mesh}
\end{figure}

\subsection{Numerical solution}\label{Numerical soln}
We look for approximate solutions in the function space of piecewise polynomials
\begin{linenomath*}
\begin{equation}\label{eqn: function space}
\mathcal{V}^p_h = \big\{v_h\in L^2(\Omega_h): v_h|_{\kappa_j} \in \mathcal{P}_p(\kappa_j), \kappa_j \in \Omega_h\big\},    
\end{equation}
\end{linenomath*}
where $\mathcal{P}_p(\kappa_j)$ denotes the space of polynomials of degree at most $p$ in the element $\kappa_j$. The approximate solution to (\cref{Eqn: Cauchy prob}) is sought as
\begin{linenomath*}
\begin{equation}
    \vecu_h(x,t) = \sum^p_{k=0} \phi^k_j(x) \textbf{U}^k_j(t) \quad \forall x\in \kappa_j, \kappa_j\in \Omega_h, t > 0,
\end{equation}
\end{linenomath*}
where the subset $(\phi^0_j,...,\phi^p_j)$ constitutes a basis of $\mathcal{V}^p_h$ restricted onto $\kappa_j$ and $\textbf{U}^{0\leqslant k\leqslant p}_{j}$ are the associated degrees of freedom (DOFs). Here we use the Lagrange interpolation polynomials $\ell_{0\leqslant k\leqslant p}$ associated to the Gauss-Lobatto nodes over the reference element $I=[-1,1]$: $-1=s_0<s_1<\cdots<s_p=1$. The basis functions thus satisfy the relation
\begin{linenomath*}
\begin{equation}\label{lagrange poly}
    \ell_k(s_l)=\delta_{kl}, \quad 0\leqslant k,l \leqslant p,
\end{equation}
\end{linenomath*}
where $\delta_{kl}$ is the Kronecker symbol. The basis functions with support in a given element $\kappa_j$ are written as $\phi^k_j(x)=\ell_k(\sigma_j(x))$, where $\sigma_j(x) = 2(x-x_j)/h$ and $x_j=(x_{j+\frac{1}{2}}+x_{j-\frac{1}{2}})/2$ denotes the center of the element.

The DOFs thus correspond to the point values of the solution: given $0\leqslant k\leqslant p$, $j \in \mathbb{Z}$, and $t \geqslant 0$, we have $\vecu_h(x^k_j,t)=\textbf{U}^k_j(t)$ for $x^k_j=x_j+s_kh/2$. Likewise, the left and right traces of the solution at the element interfaces are $\vecu_h(x^-_{j+1/2},t)=\textbf{U}^p_{j}(t)$ and $\vecu_h(x^+_{j-1/2},t)=\textbf{U}^0_{j}(t)$, respectively. The integrals over the elements are approximated using the Gauss-Lobatto quadrature rule, where the points are collocated with the interpolation points of the numerical solution:
\begin{linenomath*}
\begin{equation}\label{quadrature}
    \int_{\kappa_j}f(x)dx \approx \frac{h}{2}\sum^p_{l=0} \omega_lf(x^l_j),
\end{equation}
\end{linenomath*}
where $\omega_l > 0$, with $\sum^{p}_{l=0}\omega_l = 2$, are the quadrature weights and $x^l_j$ the quadrature points. This allows to define the discrete inner product in the element $\kappa_j$ as
\begin{linenomath*}
\begin{equation}\label{eq:1d_inner_product}
    \langle f,g \rangle^p_j := \frac{h}{2}\sum^p_{l=0} \omega_lf(x^l_j)g(x^l_j).
\end{equation}
\end{linenomath*}

We also introduce the discrete difference matrix 
\begin{linenomath*}
\begin{equation}\label{diff matrix}
    D_{kl}=\ell'_l(s_k)=\frac{h}{2}d_x\phi^l_j(x^k_j), \quad 0\leqslant k,l \leqslant p.
\end{equation}
\end{linenomath*}
This operator satisfies the summation-by-parts property, as noticed in \cite{kopriva2010quadrature},
\begin{linenomath*}
\begin{equation}\label{eqn: SBP}
    \omega_k D_{kl} + \omega_l D_{lk} =\delta_{kp}\delta_{lp}-\delta_{k0}\delta_{l0} \quad \forall 0 \leqslant k,l \leqslant p,
\end{equation}
\end{linenomath*}
which is the discrete analogue of the following integration-by-parts
\begin{linenomath*}
\begin{equation}
    \int_{\kappa_j} \phi^k_j(x)d_x\phi^l_j(x)dx + \int_{\kappa_j} d_x\phi^k_j(x)\phi^l_j(x)dx = \big[\phi^k_j(x)\phi^l_j(x)\big]^{x^-_{j+1/2}}_{x^+_{j-1/2}},
\end{equation}
\end{linenomath*}
since the Gauss-Lobatto quadrature rule is exact for polynomial integrands up to degree $2p-1$. Furthermore the property $\sum^p_{l=0} \ell_l \equiv 1$ implies
\begin{linenomath*}
\begin{equation}\label{eq:sum_Dkl_vanishes}
    \sum^p_{l=0} D_{kl}= 0 \quad \forall 0 \leqslant k \leqslant p.
\end{equation}
\end{linenomath*}

\subsection{Semi-discrete form}\label{Semi-discrete form}
The semi-discrete DGSEM formulation of (\cref{Eqn: BNM}), see \cite{rhebergen2008discontinuous,franquet2012runge,renac2019entropy}, reads: find $\vecu_h$ in $(\mathcal{V}^p_h)^7$ such that
\begin{linenomath*}
\begin{equation}\label{semi-discrete eqn}
\begin{aligned}
    \int_{\Omega_h}v_h\partial_t\vecu_h dx + \int_{\Omega_h}v_h\big(\partial_x{\bf f}(\vecu_h)+\textbf{c}(\vecu_h)\partial_x\vecu_h\big)dx &+ \sum_{j\in\mathbb{Z}} v_h\big(x^-_{j+\frac{1}{2}}\big) \D^-\big(\textbf{U}^p_j(t),\textbf{U}^0_{j+1}(t)\big)\\
    &+ \sum_{j\in\mathbb{Z}} v_h\big(x^+_{j-\frac{1}{2}}\big) \D^+\big(\textbf{U}^p_{j-1}(t),\textbf{U}^0_{j}(t)\big) =0 \quad \forall v_h \in \mathcal{V}^p_h,t > 0,
\end{aligned}
\end{equation}
\end{linenomath*}
where $\D^{\pm}(\cdot,\cdot)$ are the numerical fluxes at the interfaces in fluctuation form which will be defined below.

Upon substituting $v_h$ for the Lagrange interpolation polynomials $\phi^k_j(x)=\ell_k(\sigma_j(x))$, defined by (\cref{lagrange poly}), and using the quadrature rule (\cref{quadrature}) to approximate the volume integrals, (\cref{semi-discrete eqn}) becomes
\begin{linenomath*}
\begin{equation}\label{DG semi-discrete}
    \frac{\omega_kh}{2}\frac{d\textbf{U}^k_j}{dt}+\omega_k\sum^{p}_{l=0}D_{kl} \big(\textbf{f}(\textbf{U}^l_j)+\textbf{c}(\textbf{U}^k_j)\textbf{U}^l_j\big)+\delta_{kp}\D^-(\textbf{U}^p_j,\textbf{U}^0_{j+1}) +\delta_{k0}\D^+(\textbf{U}^p_{j-1},\textbf{U}^0_{j}) = 0\quad \forall j\in\mathbb{Z}, 0\leqslant k\leqslant p,
\end{equation}
\end{linenomath*}
along with the projection of the initial condition (\cref{Eqn: ic_bnm}) on the function space:
\begin{linenomath*}
\begin{equation}
    \textbf{U}^k_j(0) = \vecu_0(x^k_j) \quad \forall j\in\mathbb{Z},0\leqslant k\leqslant p.
\end{equation}
\end{linenomath*}

\subsection{Numerical fluxes} \label{Numerical fluxes}
We rely on numerical fluxes in fluctuation form \cite{pares2006numerical} that satisfy the properties of entropy conservation and entropy stability for the semi-discrete form (\cref{DG semi-discrete}). Here we recall their definition from \cite{castro2013entropy}.
\begin{definition}\label{Defn: entropy cons, stable}
Let $\normalfont{\D}^{\pm}_{ec}$ be Lipschitz continuous and consistent numerical fluxes in fluctuation form, $\D^\pm_{ec}(\vecu,\vecu)=0$ for all $\vecu$ in $\Omega_{\BNM}$, and $(\eta,q)$ be an entropy-entropy flux pair for (\cref{Eqn: BNM}), then $\normalfont{\D}^{\pm}_{ec}$ are said to be entropy conservative if they satisfy the following relation: 
\begin{linenomath*}
\begin{equation}\label{Eqn: entropy conserv defn nc}
    \V(\vecu^-)^\top\D^-_{ec}(\vecu^-,\vecu^+)+\V(\vecu^+)^\top\D^+_{ec}(\vecu^-,\vecu^+)=q(\vecu^+)-q(\vecu^-) \quad\forall \vecu^\pm \in \Omega_{\BNM},
\end{equation}
\end{linenomath*}
where $\V(\vecu^\pm) := \eta'(\vecu^\pm)$ denote the entropy variables. 
\end{definition}

In this work we look for entropy conservative fluxes with the following form
\begin{linenomath*}
\begin{subequations}\label{Eqn: fluctuation flux}
\begin{align}
    \D^-_{ec}(\vecu^-,\vecu^+) &= \textbf{h}(\vecu^-,\vecu^+)-\textbf{f}(\vecu^-)+\normalfont{\textbf{d}}^-(\vecu^-,\vecu^+),\label{Dec1}\\
    \D^+_{ec}(\vecu^-,\vecu^+) &= \textbf{f}(\vecu^+)-\textbf{h}(\vecu^-,\vecu^+)+\normalfont{\textbf{d}}^+(\vecu^-,\vecu^+),\label{Dec2}
\end{align}
\end{subequations}
\end{linenomath*}
where $\textbf{h}(\vecu^-,\vecu^+)$ is a numerical flux that approximates the traces of the physical fluxes, $\textbf{f}(\vecu^\pm)$, and $\textbf{d}^\pm(\vecu^-,\vecu^+)$ are fluctuation fluxes for the discretization of the nonconservative term in (\cref{Eqn: BNM}). The numerical fluxes satisfy the consistency conditions:
\begin{linenomath*}
\begin{equation}\label{eq:consistency_h_d}
     \normalfont{\textbf{h}}(\vecu,\vecu) = \textbf{f}(\vecu),\quad \textbf{d}^\pm(\vecu,\vecu)=0 \quad \forall \vecu \in \Omega_{\BNM}.
\end{equation}
\end{linenomath*}

The condition for entropy conservation now becomes
\begin{linenomath*}
\begin{equation}\label{Eqn: entropy conserv defn}
    \V(\vecu^-)^\top{\bf d}^-(\vecu^-,\vecu^+)+\V(\vecu^+)^\top{\bf d}^+(\vecu^-,\vecu^+) + \lb\V^\top{\bf f} -q \rb = \textbf{h}(\vecu^-,\vecu^+)\lb \V\rb \quad\forall \vecu^\pm \in \Omega_{\BNM},
\end{equation}
\end{linenomath*}
where $\lb a\rb = a^+ - a^-$ denotes the jump operator. This relation is a direct generalization of entropy conditions in \cite{tadmor1987numerical,castro2013entropy} for systems with conservative and nonconservative terms.

Furthermore, we seek entropy stable fluxes by adding dissipation to the entropy conservative fluxes as advocated in \cite{ismail2009affordable} for conservation laws:
\begin{linenomath*}
\begin{equation}\label{Eqn: entropy stable defn}
    \D^\pm(\vecu^-,\vecu^+):= \D^\pm_{ec}(\vecu^-,\vecu^+) \pm \D_\nu(\vecu^-,\vecu^+),
\end{equation}
\end{linenomath*}
where $\D_\nu(\vecu^-,\vecu^+)$ is a numerical dissipation that satisfies consistency and entropy dissipation:
\begin{linenomath*}
\begin{equation} \label{Eqn: entropy constraint}
   \normalfont{\D}_\nu(\vecu,\vecu)=0, \quad \lb\V(\vecu)\rb^\top \textbf{D}_\nu(\vecu^-,\vecu^+) \geqslant 0 \quad \forall \vecu,\vecu^\pm\in\Omega_{\BNM}.
\end{equation}
\end{linenomath*}

Observe, in the semi-discrete form (\cref{DG semi-discrete}), that the discrete volume integral does not bear proper constraints towards entropy conservation or dissipation. In other words we cannot control the sign of its scalar product with the entropy variables. Therefore, we modify the volume integral and replace it with entropy conservative fluctuation fluxes, as in \cite{renac2019entropy}. The semi-discrete scheme now reads
\begin{linenomath*}
\begin{equation}\label{eqn: modified DG semi-discrete}
    \frac{\omega_kh}{2}\frac{d\textbf{U}^k_j}{dt}+\textbf{R}^k_j(\vecu_h) = 0, 
\end{equation}
\end{linenomath*}
where
\begin{linenomath*}
\begin{equation}\label{eqn: residual}
    \textbf{R}^k_j(\vecu_h) = \omega_k \sum^p_{l=0}D_{kl}\tD(\textbf{U}^k_j,\textbf{U}^l_j) + \delta_{kp}\D^-(\textbf{U}^p_j,\textbf{U}^0_{j+1}) +\delta_{k0}\D^+(\textbf{U}^p_{j-1},\textbf{U}^0_{j}),
\end{equation}
\end{linenomath*}
and
\begin{linenomath*}
\begin{subequations}\label{tilde D}
 \begin{align}
    \tD(\vecu^-,\vecu^+) &:= \D^-_{ec}(\vecu^-,\vecu^+)-\D^+_{ec}(\vecu^+,\vecu^-), \label{tilde Da} \\
		&\overset{(\cref{Eqn: fluctuation flux})}{=} \textbf{h}(\vecu^-,\vecu^+) + \textbf{h}(\vecu^+,\vecu^-) + \nc^-(\vecu^-,\vecu^+) - \nc^+(\vecu^+,\vecu^-). \label{tilde Db}
 \end{align}
\end{subequations}
\end{linenomath*}

Note that in the above relation we do not require $\textbf{h}$ to be symmetric as in \cite{fisher2013high,chen2017entropy}, but rather use the symmetrizer $\tfrac{1}{2}\big(\textbf{h}(\vecu^-,\vecu^+) + \textbf{h}(\vecu^+,\vecu^-)\big)$.

\subsection{Properties of the semi-discrete scheme} \label{prop semi-discrete}
The modification to the integrals over cell elements in (\cref{eqn: residual}) allows for an entropy stable numerical scheme that preserves the high-order accuracy of the scheme. Below we generalize the results from \cite{renac2019entropy} to systems that contain both conservative and nonconservative terms. 

\begin{mythm}\label{theorem entropy ineq}
Let $\normalfont{\D}^\pm$ be consistent and entropy stable fluctuation fluxes (\cref{Eqn: entropy stable defn}) and (\cref{Eqn: entropy constraint}) in (\cref{eqn: residual}) and $\tD$ defined by (\cref{tilde Db}) with consistent and entropy conservative fluctuation fluxes (\cref{Eqn: entropy conserv defn}) and (\cref{eq:consistency_h_d}). Then, the semi-discrete numerical scheme (\cref{eqn: modified DG semi-discrete}) satisfies an entropy inequality for the entropy-entropy flux pair $(\eta,q)$ in (\cref{Eqn: entropy ineq}):
\begin{linenomath*}
\begin{equation}\label{semi-discrete entropy ineq}
    h\frac{d\langle\eta(\vecu_h)\rangle_j}{dt}+Q(\normalfont{\textbf{U}}^p_j,\normalfont{\textbf{U}}^0_{j+1})-Q(\normalfont{\textbf{U}}^p_{j-1},\normalfont{\textbf{U}}^0_{j})\leqslant 0,
\end{equation}
\end{linenomath*}
where $\displaystyle\langle\eta(\vecu_h)\rangle_j(t)=\sum^p_{k=0}\frac{\omega_k}{2}\eta\big(\textbf{U}^k_j(t)\big)$ is the cell averaged entropy and the conservative numerical entropy flux is defined by
\begin{linenomath*}
\begin{equation}
    Q(\vecu^-,\vecu^+) = \frac{q(\vecu^-)+q(\vecu^+)}{2}+\frac{1}{2}\V(\vecu^-)^\top\D^-(\vecu^-,\vecu^+)-\frac{1}{2}\V(\vecu^+)^\top\D^+(\vecu^-,\vecu^+).
\end{equation}
\end{linenomath*}

Further assuming that ${\bf d}^\pm$ in (\cref{tilde Db}) have the form
\begin{linenomath*}
\begin{subequations}\label{eqn: high-order approximation proof}
\begin{align}
\normalfont{\bf d}^{\pm}(\vecu^-,\vecu^+) &= \mathcal{C}^\pm(\vecu^-,\vecu^+)\lb\vecu\rb, \label{eqn: high-order approximation proof a}\\
\mathcal{C}(\vecu^-,\vecu^+) &:= \mathcal{C}^+(\vecu^-,\vecu^+) + \mathcal{C}^-(\vecu^-,\vecu^+), \label{eqn: high-order approximation proof b}\\
\mathcal{C}(\vecu^-,\vecu^+) + \mathcal{C}(\vecu^+,\vecu^-) &= \normalfont{\textbf{c}}(\vecu^-)+\textbf{c}(\vecu^+), \label{eqn: high-order approximation proof c}\\
\mathcal{C}(\vecu,\vecu) &=\normalfont{\textbf{c}}(\vecu), \label{eqn: high-order approximation proof d}
\end{align}
\end{subequations}
\end{linenomath*}
where $\lb \vecu\rb = \vecu^+-\vecu^-$, then semi-discrete DGSEM (\cref{eqn: modified DG semi-discrete}) is a high-order approximation in space of smooth solutions for the nonconservative system (\cref{Eqn: BNM}) that satisfies
\begin{linenomath*}
\begin{equation}\label{Eqn: cell avg conserv}
    h\frac{d\langle\vecu_h\rangle_j}{dt} + \langle\textbf{c}(\vecu_h),d_x\vecu_h\rangle^p_j + \D^-(\textbf{U}^p_j,\textbf{U}^0_{j+1})+{\bf f}(\textbf{U}^p_j)+\D^+(\textbf{U}^p_{j-1},\textbf{U}^0_{j})-{\bf f}(\textbf{U}^0_j) = 0,
\end{equation}
\end{linenomath*}
for the cell averaged solution
\begin{linenomath*}
\begin{equation}\label{cell avg defn}
    \langle\vecu_h\rangle_j(t):=\frac{1}{h}\int_{\kappa_j}\vecu_h(x,t)dx=\frac{1}{2}\sum^p_{k=0}\omega_k\textbf{U}^k_j(t).
\end{equation}
\end{linenomath*}
\end{mythm}

\begin{proof}
 These results are consequences of, e.g., \cite[Theorem 3.3]{chen2017entropy} for the conservative terms and \cite[Theorems 3.1 and 3.2]{renac2019entropy} for the nonconserative ones. First, the entropy inequality has been proved in \cite[Theorem 3.1]{renac2019entropy} by using the definition (\cref{tilde Da}) of the volume terms together with the entropy condition (\cref{Eqn: entropy conserv defn nc}). High-order accuracy of the discretization in the volume integral in (\cref{eqn: residual}) has been proved in \cite[Theorem 3.3]{chen2017entropy} for the conservative terms by using the symmetric flux $\tfrac{1}{2}\big(\textbf{h}(\vecu^-,\vecu^+) + \textbf{h}(\vecu^+,\vecu^-)\big)$ in (\cref{tilde Db}) and the SBP property (\cref{eqn: SBP}), and in \cite[Theorem 3.2]{renac2019entropy} by using (\cref{eqn: high-order approximation proof}) and the SBP property. Finally, by summing (\cref{eqn: modified DG semi-discrete}) over $0\leqslant k\leqslant p$ and using (\cref{eqn: residual}) and (\cref{cell avg defn}) we obtain
\begin{linenomath*}
\begin{equation*}
  h\frac{d\langle\vecu_h\rangle_j}{dt} + \sum_{k=0}^p\sum^p_{l=0}\omega_kD_{kl}\tD(\textbf{U}^k_j,\textbf{U}^l_j) + \D^-(\textbf{U}^p_j,\textbf{U}^0_{j+1}) + \D^+(\textbf{U}^p_{j-1},\textbf{U}^0_{j}) = 0,
\end{equation*}
\end{linenomath*}
where
\begin{linenomath*}
\begin{align*}
 \sum_{k,l=0}^p\omega_kD_{kl}\tD(\textbf{U}^k_j,\textbf{U}^l_j) &\overset{(\cref{tilde Db})}{=} \sum_{k,l=0}^p\omega_kD_{kl} \big({\bf h}(\textbf{U}^k_j,\textbf{U}^l_j)+{\bf d}^-(\textbf{U}^k_j,\textbf{U}^l_j)\big) + \sum_{k,l=0}^p\omega_kD_{kl} \big({\bf h}(\textbf{U}^l_j,\textbf{U}^k_j)-{\bf d}^+(\textbf{U}^l_j,\textbf{U}^k_j)\big) \\
&\overset{(\cref{eqn: SBP})}{\underset{(\cref{eq:consistency_h_d})}{=}} \sum_{k,l=0}^p\omega_kD_{kl} \big({\bf h}(\textbf{U}^k_j,\textbf{U}^l_j)+{\bf d}^-(\textbf{U}^k_j,\textbf{U}^l_j)\big) - \sum_{k,l=0}^p\omega_lD_{lk} \big({\bf h}(\textbf{U}^l_j,\textbf{U}^k_j)-{\bf d}^+(\textbf{U}^l_j,\textbf{U}^k_j)\big) + {\bf f}({\bf U}_j^p) - {\bf f}({\bf U}_j^0) \\
&\overset{(\cref{eqn: high-order approximation proof a})}{\underset{(\cref{eqn: high-order approximation proof b})}{=}} \sum_{k,l=0}^p\omega_kD_{kl} {\cal C}(\textbf{U}^k_j,\textbf{U}^l_j)(\textbf{U}^l_j-\textbf{U}^k_j) + {\bf f}({\bf U}_j^p) - {\bf f}({\bf U}_j^0) \\
&\overset{(\cref{eqn: SBP})}{\underset{(\cref{eqn: high-order approximation proof d})}{=}} \sum_{k,l=0}^p\omega_kD_{kl} {\cal C}(\textbf{U}^k_j,\textbf{U}^l_j)\textbf{U}^l_j + \sum_{k,l=0}^p\omega_lD_{lk} {\cal C}(\textbf{U}^k_j,\textbf{U}^l_j)\textbf{U}^k_j  - {\bf c}({\bf U}_j^p){\bf U}_j^p + {\bf c}({\bf U}_j^0){\bf U}_j^0 + {\bf f}({\bf U}_j^p) - {\bf f}({\bf U}_j^0) \\
&\overset{(\cref{eqn: high-order approximation proof c})}{=}  \sum_{k,l=0}^p\omega_kD_{kl} \big( {\bf c}(\textbf{U}^k_j)+{\bf c}(\textbf{U}^l_j) \big)\textbf{U}^l_j - {\bf c}({\bf U}_j^p){\bf U}_j^p + {\bf c}({\bf U}_j^0){\bf U}_j^0 + {\bf f}({\bf U}_j^p) - {\bf f}({\bf U}_j^0) \\
&\overset{(\cref{eqn: SBP})}{=}  \sum_{k,l=0}^p\omega_kD_{kl} {\bf c}(\textbf{U}^k_j)\textbf{U}^l_j - \sum_{k,l=0}^p\omega_kD_{kl} {\bf c}(\textbf{U}^k_j)\textbf{U}^k_j + {\bf f}({\bf U}_j^p) - {\bf f}({\bf U}_j^0) \\
&\overset{(\cref{eq:sum_Dkl_vanishes})}{\underset{(\cref{eq:1d_inner_product})}{=}} \langle\textbf{c}(\vecu_h),d_x\vecu_h\rangle^p_j + {\bf f}({\bf U}_j^p) - {\bf f}({\bf U}_j^0)
\end{align*}
\end{linenomath*}
\end{proof}

Note that (\cref{Eqn: cell avg conserv}) proves that the discretization of the fluxes ${\bf f}$ in (\cref{eqn: residual}) is in conservative form. In the following section we propose numerical fluxes for (\cref{Eqn: compact BNM}) that satisfy the assumptions in \cref{theorem entropy ineq}.

%
%
\section{Numerical fluxes for the Baer-Nunziato model} \label{sec: fluctuation fluxes}
Here we derive the numerical fluxes for the model (\cref{Eqn: BNM}) that satisfy the entropy conservation (\cref{Eqn: entropy conserv defn}) and dissipation (\cref{Eqn: entropy stable defn}) properties together with the assumptions in \cref{theorem entropy ineq}. An essential tool which would help in the algebraic manipulations are the Leibniz identities, which we recall here. Let $a^+,a^-,b^+,b^-,c^+,c^-$ in $\mathbb{R}$ have finite values, then we have
\begin{linenomath*}
\begin{equation}\label{Defn: Leibniz rule}
    \displaystyle \lb ab \rb = \xoverline{a}\lb b\rb + \xoverline{b}\lb a\rb, \quad \lb abc \rb = \xoverline{a}(\xoverline{b}\lb c\rb + \xoverline{c} \lb b\rb) + \xoverline{bc}\lb a\rb,
\end{equation}
\end{linenomath*}
where $\xoverline{a}= \displaystyle \frac{a^++a^-}{2}$ is the arithmetic mean and $\lb a\rb = a^+-a^-$ is the jump operator.

\subsection{Entropy conservative fluxes}\label{ssec: entropy conserv}
We begin by proposing entropy conservative numerical fluxes.

\begin{prop} \label{thm EC FF}
The numerical fluxes (\cref{Eqn: fluctuation flux}) with the following definitions are consistent and entropy conservative fluxes that satisfy the assumptions (\cref{eqn: high-order approximation proof}) of \cref{theorem entropy ineq} for the Baer-Nunziato model (\cref{Eqn: BNM}) with the EOS (\cref{Eqn: EOS}) and the interface variables (\cref{eqn: intref_var}).
\begin{linenomath*}
\begin{equation}\label{eqn: EC fluxes}
    \normalfont{\textbf{h}}(\vecu^-,\vecu^+) := 
    \begin{pmatrix}
    0\\
    h_{\rho_i}\\
    h_{\rho u_i}\\
    h_{\rho E_i}
    \end{pmatrix}
    -\beta_s\frac{\lb \alpha_i\rb}{2}
    \begin{pmatrix}
    1\\ 
    \tilde{h}_{\rho_i}\\ 
    \tilde{h}_{\rho u_i} \\ 
    \tilde{h}_{\rho E_i}
    \end{pmatrix}, \quad \normalfont{\textbf{d}}^\pm(\vecu^-,\vecu^+):=\frac{\lb\alpha_i\rb}{2}
    \begin{pmatrix}
    \uI^{\pm}\\
    0\\ 
    -\pI^{\pm}\\
    -\pI^\pm\uI^\pm
    \end{pmatrix},
\end{equation}
\end{linenomath*}
where
\begin{linenomath*}
\begin{equation}\label{eqn: EC fluxes values}
    \begin{aligned}
        (h_{\rho_i}, h_{\rho u_i}, h_{\rho E_i}) &= \left(\xoverline{\alpha}_i\xoverline{u}_i\hat{\rho}_i,\; \xoverline{\alpha}_i\left(\xoverline{u}^2_i\hat{\rho}_i+\tfrac{\xoverline{\mathrm{p}_i\theta_i}}{\xoverline{\theta}_i}\right),\; \xoverline{\alpha}_i\xoverline{u}_i\left(\hat{\rho}_i\left(\tfrac{\Cv_i}{\hat{\theta}_i}+\tfrac{u^-_iu^+_i}{2}\right)+\tfrac{\xoverline{\press_i\theta_i}}{\xoverline{\theta_i}}+\press_{\infty,i}\right)\right),\\
        (\tilde{h}_{\rho_i}, \tilde{h}_{\rho u_i}, \tilde{h}_{\rho E_i}) &= \left(\hat{\rho}_i,\; \hat{\rho}_i\xoverline{u}_i,\; \hat{\rho}_i\left(\tfrac{\Cv_i}{\hat{\theta}_i}+\tfrac{u^-_iu^+_i}{2}\right) + \press_{\infty,i}\right),
    \end{aligned}
\end{equation}
\end{linenomath*}
$\beta_s \geqslant 0$ is a user-defined coefficient (see \cref{thm: positivity}) and $\displaystyle\hat{a} = \frac{\lb a\rb}{\lb \ln a\rb}$ is the logarithmic mean \cite{ismail2009affordable}.
\end{prop}

\begin{proof}
Consistency of the numerical flux ${\bf h}$ follows from consistency of the arithmetic and logarithmic means and the fact that $\rho_i e_i=\rho_i\Cv_iT_i+\press_{\infty,i}$ for the EOS (\cref{Eqn: EOS}). It can be easily checked that ${\bf d}^\pm$ satisfy (\cref{eqn: high-order approximation proof}) and consistency ${\bf d}^\pm({\bf u},{\bf u})=0$.

Now let us recall the entropy variables associated to the entropy in (\cref{Eqn: entropy ineq}):
\begin{linenomath*}
\begin{equation}\label{Eqn: entropy variables}
    \V(\vecu)=
    \begin{pmatrix}
    (-1)^i\left(\mathrm{p}_1\theta_1-\mathrm{p}_2\theta_2\right)\\
    -s_i+ \left(\displaystyle h_i-\frac{u^2_i}{2}\right)\theta_i\\
    u_i\theta_i\\ 
    -\theta_i
    \end{pmatrix},
\end{equation}
\end{linenomath*}
where $h_i(\rho_i,e_i)=e_i+\tfrac{\mathrm{p}_i(\rho_i,e_i)}{\rho_i}=\Cp_iT_i$ is the specific enthalpy for phase $i=1,2$. Then, the discrete counterpart of (\cref{eqn: partialx alphai in entropy equality}) holds for  interface closure laws (\cref{eqn: intref_var}) and reads
\begin{linenomath*}
\begin{equation}\label{Eqn:pIpuIu}
 \sum^2_{i=1}\xoverline{(\pI-\press_i)(\uI-u_i)\theta_i}\lb\alpha_i\rb = 0.
\end{equation}
\end{linenomath*}

Entropy conservation requires the fluxes (\cref{Eqn: fluctuation flux}) to satisfy (\cref{Eqn: entropy conserv defn}) so we have to check that
\begin{linenomath*}
\begin{equation}\label{eqn: delta eta EC flux}
\Delta Q(\vecu^-,\vecu^+) :=  -\textbf{h}(\vecu^-,\vecu^+)\cdot \lb\V(\vecu)\rb+\V(\vecu^-)\cdot\textbf{d}^-(\vecu^-,\vecu^+) + \V(\vecu^+)\cdot\textbf{d}^+(\vecu^-,\vecu^+) + \lb\textbf{f}(\vecu)\cdot\V(\vecu)-q(\vecu)\rb = 0.
\end{equation}
\end{linenomath*}
Below we detail each term in the above relation by using the Liebniz identities (\cref{Defn: Leibniz rule}) for the numerical fluxes (\cref{eqn: EC fluxes}). Note that direct manipulations give
\begin{linenomath*}
\begin{equation}\label{Eqn:thermo_jump}
 \lb\mathrm{p}_i\theta_i\rb \overset{(\cref{Eqn: EOS})}{=} (\gamma_i-1)\Cv_i\lb\rho_i\rb-\press_{\infty,i}\lb\theta_i\rb, \quad \lb h_i\theta_i\rb = 0, \quad \lb s_i\rb \overset{(\cref{eqn: physical entropy})}{=} - \Cv_i\lb\ln \theta_i\rb-(\gamma_i-1)\Cv_i\lb\ln \rho_i\rb, \quad \xoverline{u}_i^2-\xoverline{\frac{u_i^2}{2}}=\frac{u_i^-u_i^+}{2}.
\end{equation}
\end{linenomath*}

Then, by (\cref{eqn: EC fluxes}) and (\cref{Eqn: entropy variables}), we have
\begin{linenomath*}
\begin{equation}\label{eqn: vh}
\begin{aligned}
    \lb \V(\vecu) \rb \cdot \textbf{h}(\vecu^-,\vecu^+) = &\sum^2_{i=1} \xoverline{\alpha}_i\hat{\rho}_i\xoverline{u}_i\lb (h_i-u^2_i/2)\theta_i - s_i \rb + \xoverline{\alpha}_i\left(\hat{\rho}_i\xoverline{u}^2_i+\tfrac{\xoverline{\mathrm{p}_i\theta_i}}{\xoverline{\theta}_i}\right) \lb u_i\theta_i \rb - \xoverline{\alpha}_i\xoverline{u}_i\left(\hat{\rho}_i\left(\tfrac{\Cv_i}{\hat{\theta}_i}+\tfrac{u^-_iu^+_i}{2}\right)+\tfrac{\xoverline{\press_i\theta_i}}{\xoverline{\theta_i}}+\press_{\infty,i}\right) \lb \theta_i\rb\\
    &-\beta_s\tfrac{\lb \alpha_i\rb}{2}\left( -\lb\press_i\theta_i\rb + \hat{\rho}_i\lb (h_i-u^2_i/2)\theta_i-s_i\rb + \hat{\rho}_i\xoverline{u}_i\lb u_i\theta_i\rb - \left(\hat{\rho}_i\left(\tfrac{\Cv_i}{\hat{\theta}_i}+\tfrac{u^-_iu^+_i}{2}\right)+\press_{\infty,i}\right)\lb\theta_i\rb \right)\\
    \overset{(\cref{Eqn:thermo_jump})}{=} &\sum^2_{i=1} - \xoverline{\alpha}_i\hat{\rho}_i\xoverline{u}_i \Big( \xoverline{u}_i\xoverline{\theta}_i\lb u_i\rb + \xoverline{u^2_i/2}\lb\theta_i\rb - \Cv_i\lb\ln \theta_i\rb-(\gamma_i-1)\Cv_i\lb\ln \rho_i\rb\Big)\\
    &+\xoverline{\alpha}_i\left(\hat{\rho}_i\xoverline{u}^2_i+\tfrac{\xoverline{\mathrm{p}_i\theta_i}}{\xoverline{\theta}_i}\right) \lb u_i\theta_i \rb - \xoverline{\alpha}_i\xoverline{u}_i\left(\hat{\rho}_i\left(\tfrac{\Cv_i}{\hat{\theta}_i}+\tfrac{u^-_iu^+_i}{2}\right)+\tfrac{\xoverline{\press_i\theta_i}}{\xoverline{\theta_i}}+\press_{\infty,i}\right) \lb \theta_i\rb\\
    &-\beta_s\tfrac{\lb\alpha_i\rb}{2}\bigg(-(\gamma_i-1)\Cv_i\lb\rho_i\rb+\press_{\infty,i}\lb\theta_i\rb - \hat{\rho}_i\left( \xoverline{u}_i\xoverline{\theta}_i\lb u_i\rb + \xoverline{u^2_i/2}\lb\theta_i\rb - \Cv_i\lb\ln \theta_i\rb-(\gamma_i-1)\Cv_i\lb\ln \rho_i\rb\right) \\
    &+\hat{\rho}_i\xoverline{u}_i\left(\xoverline{u}_i\lb\theta_i\rb+\xoverline{\theta}_i\lb u_i\rb\right) - \left(\hat{\rho}_i\left(\tfrac{\Cv_i}{\hat{\theta}_i}+\tfrac{u^-_iu^+_i}{2}\right)+\press_{\infty,i}\right)\lb\theta_i\rb\bigg)\\
		\overset{(\cref{Defn: Leibniz rule})}{\underset{(\cref{Eqn:thermo_jump})}{=}}& \sum^2_{i=1} - \xoverline{\alpha}_i\xoverline{u}_i\hat{\rho}_i \Big( \xoverline{u}_i\xoverline{\theta}_i\lb u_i\rb + \xoverline{u^2_i/2}\lb\theta_i\rb - \Cv_i\lb\ln \theta_i\rb-(\gamma_i-1)\Cv_i\lb\ln \rho_i\rb\Big)\\
    &+\xoverline{\alpha}_i\left(\xoverline{u}^2_i\hat{\rho}_i+\tfrac{\xoverline{\mathrm{p}_i\theta_i}}{\xoverline{\theta}_i}\right) \big(\xoverline{u_i}\lb\theta_i\rb+\xoverline{\theta_i}\lb u_i\rb\big)-\xoverline{\alpha}_i\xoverline{u}_i\left(\hat{\rho}_i\left(\tfrac{\Cv_i}{\hat{\theta}_i}+\tfrac{u^-_iu^+_i}{2}\right)+\tfrac{\xoverline{\press_i\theta_i}}{\xoverline{\theta_i}}+\press_{\infty,i}\right) \lb \theta_i\rb.
\end{aligned}
\end{equation}
\end{linenomath*}

Furthermore, using (\cref{eqn: EC fluxes}) we easily obtain
\begin{linenomath*}
\begin{equation}\label{eqn vd- + vd+}
    \begin{aligned}
        \V(\vecu^-)\cdot\textbf{d}^-(\vecu^-,\vecu^+) + \V(\vecu^+)\cdot\textbf{d}^+(\vecu^-,\vecu^+) = \sum^2_{i=1} \xoverline{(\pI\uI-\pI u_i-\press_i\uI)\theta_i}\lb\alpha_i\rb \equaltext{(\cref{Eqn:pIpuIu})} -\sum^2_{i=1} \xoverline{\press_iu_i\theta_i}\lb\alpha_i\rb,
    \end{aligned}
\end{equation}
\end{linenomath*}
and
\begin{linenomath*}
\begin{equation}\label{eqn: fv-q}
    \begin{aligned}
        \lb\textbf{f}(\vecu)\cdot\V(\vecu)-q(\vecu)\rb = &\sum^2_{i=1} \lb-\alpha_i\rho_iu_i\big(s_i-(h_i-u^2_i/2)\theta_i\big) + \alpha_i(\rho_iu^2_i+\press_i)u_i\theta_i - \alpha_i(\rho_iE_i+\press_i)u_i\theta_i+\alpha_i\rho_is_iu_i\rb\\
        = &\sum^2_{i=1} \lb \alpha_i\press_iu_i\theta_i\rb
        \equaltext{(\cref{Defn: Leibniz rule})} \sum^2_{i=1} \xoverline{\press_iu_i\theta_i}\lb\alpha_i\rb + \xoverline{\alpha}_i \xoverline{\press_i\theta_i}\lb u_i\rb +\xoverline{\alpha}_i\xoverline{u}_i\lb\press_i\theta_i\rb \\
				\equaltext{(\cref{Eqn:thermo_jump})}& \sum^2_{i=1} \xoverline{\press_iu_i\theta_i}\lb\alpha_i\rb + \xoverline{\alpha}_i \xoverline{\press_i\theta_i}\lb u_i\rb +\xoverline{\alpha}_i\xoverline{u}_i\big((\gamma_i-1)\Cv_i\lb\rho_i\rb-\press_{\infty,i}\lb\theta_i\rb\big).
    \end{aligned}
\end{equation}
\end{linenomath*}

Substituting (\cref{eqn: vh}), (\cref{eqn vd- + vd+}) and (\cref{eqn: fv-q}) into (\cref{eqn: delta eta EC flux}) and collecting terms proportional to $\lb\rho_i\rb$, $\lb u_i\rb$, and $\lb\theta_i\rb$, we get
\begin{linenomath*}
\begin{equation*}
    \begin{aligned}
        \Delta Q(\vecu^-,\vecu^+) \overset{(\cref{Eqn:thermo_jump})}{=} &\sum^2_{i=1} \xoverline{\alpha}_i\left(\hat{\rho}_i\xoverline{u}_i^2\xoverline{\theta}_i-\Big(\hat{\rho}_i\xoverline{u}^2_i+\tfrac{\xoverline{\mathrm{p}_i\theta_i}}{\xoverline{\theta}_i}\Big)\xoverline{\theta}_i+\xoverline{\mathrm{p}_i\theta_i}\right)\lb u_i\rb \\
				&+\xoverline{\alpha}_i\xoverline{u}_i\left(\hat{\rho}_i\Big(\xoverline{\tfrac{u^2_i}{2}}-\Cv_i\tfrac{\lb\ln \theta_i\rb}{\lb\theta_i\rb}\Big) - \hat{\rho}_i\xoverline{u}_i^2-\tfrac{\xoverline{\mathrm{p}_i\theta_i}}{\xoverline{\theta_i}}+\hat{\rho}_i\Big(\tfrac{\Cv_i}{\hat{\theta}_i}+\tfrac{u^-_iu^+_i}{2}\Big) + \tfrac{\xoverline{\mathrm{p}_i\theta_i}}{\xoverline{\theta_i}} + \press_{\infty,i} - \press_{\infty,i} \right)\lb \theta_i\rb \\
				&-(\gamma_i-1)\Cv_i\xoverline{\alpha}_i\xoverline{u}_i\Big(\hat{\rho}_i\lb\ln\rho_i\rb - \lb\rho_i\rb \Big) = 0,
    \end{aligned}
\end{equation*}
\end{linenomath*}
which concludes the proof.
\end{proof}

\begin{remark}\label{rk:LD_dissip}
The contributions to the volume integral in (\cref{eqn: residual}) of the terms associated to $\beta_s$ in (\cref{eqn: EC fluxes}) vanish due to the symmetrizer $\normalfont{\textbf{h}}(\vecu^-,\vecu^+) + \textbf{h}(\vecu^+,\vecu^-)$ in (\cref{tilde Db}). They will however play an important role in the design of the entropy stable fluxes at interfaces (see \cref{thm: positivity}). They may be compared to the upwinding term in the Lax-Friedrichs flux derived in \cite{saurel2003multiphase} for (\cref{Eqn: BNM}). The main motivation for including this term was to introduce stabilizing mechanisms in the transport equation for the void fraction, as is evident from the first component of $\normalfont{\bf h}$ in (\cref{eqn: EC fluxes}). However, $\uI$ is associated to a LD field, so the remaining terms $\tilde{h}_{\rho_i}$, $\tilde{h}_{\rho u_i}$, and $\tilde{h}_{\rho E_i}$ are further included so that this dissipation does not affect the entropy balance as shown in the proof above.
\end{remark}

\begin{remark}
Assuming perfect gas EOS in (\cref{Eqn: EOS}), $\mathrm{p}_{\infty,i}=0$, and uniform void fractions, $\lb\alpha_i\rb=0$, then the numerical flux ${\bf h}(\vecu^-,\vecu^+)$ in (\cref{eqn: EC fluxes}) for both phases reduce to the entropy conservative Chandraskhar flux \cite{Chandrashekar2013kep_es} for the compressible Euler equations. This numerical flux has been here extended to stiffened gas EOS (\cref{Eqn: EOS}).
\end{remark}

\subsection{Entropy stable fluxes}
We here follow the procedure in \cite{ismail2009affordable} and build entropy stable fluxes (\cref{Eqn: entropy stable defn}) by adding upwind-type dissipation to the entropy conservative numerical fluxes (\cref{Eqn: fluctuation flux}). We introduce numerical dissipation to the equations of mass, momentum and energy for each phase. The rationales for this particular choice of the numerical dissipation are as follows. First, we do not add numerical dissipation to the void fraction equation as it is associated to a LD field. We stress that the conservative flux in (\cref{eqn: EC fluxes}) already adds dissipation through an upwinding term without altering the entropy balance (see \cref{rk:LD_dissip}). Second, since we exclude resonance effects according to the assumption (\cref{eqn: resonance}), the void fractions remain uniform across shocks leading to uncoupled phases. It is, thus, appropriate to include dissipation phase by phase. 

\begin{prop}\label{thm: ES FF}
A class of entropy stable fluxes (\cref{Eqn: entropy stable defn}) that satisfy (\cref{Eqn: entropy constraint}) can be obtained for the Baer-Nunziato model (\cref{Eqn: BNM}) where the numerical dissipation takes the form
\begin{linenomath*}
\begin{equation*}
    \normalfont{\textbf{D}}_\nu (\vecu^-,\vecu^+)= \begin{pmatrix} 
    0 & 0      & 0      & 0      \\
    0 & k_{22} & 0      & 0      \\
    0 & k_{32} & k_{33} & 0      \\
    0 & k_{42} & k_{43} & k_{44} \\
    \end{pmatrix}\begin{pmatrix}0 \\ \lb\rho_i\rb \\ \lb u_i\rb \\ \lb T_i\rb\end{pmatrix},
\end{equation*}
\end{linenomath*}
where the matrix entries satisfy the following conditions
\begin{linenomath*}
\begin{equation}\label{Eqn: ES unknowns}
 k_{22} \geqslant 0, \quad k_{33} \geqslant 0, \quad k_{44} \geqslant 0,\quad k_{32} = \xoverline{u_i} k_{22}, \quad k_{43} = \xoverline{u_i}k_{33},\quad k_{42} = \left(\frac{\Cv_i}{\hat{\theta_i}} + \frac{u^-_iu^+_i}{2} \right)k_{22}.
\end{equation}
\end{linenomath*}
\end{prop}

\begin{proof}
By construction we have $\normalfont{\textbf{D}}_\nu (\vecu,\vecu)=0$. Then, using (\cref{Eqn: entropy variables}) and (\cref{eqn: physical entropy}), we get
\begin{linenomath*}
\begin{equation*}
    \begin{aligned}
\lb\V(\vecu)\rb\cdot\textbf{D}_\nu(\vecu^-,\vecu^+) = \sum^2_{i=1} & k_{22}(\gamma_i-1)\Cv_i\lb\rho_i\rb\lb\ln\rho_i\rb +k_{33}\xoverline{\theta_i}\lb u_i\rb^2-k_{44}\lb T_i\rb\lb\theta_i\rb\\
    &+\xoverline{\theta_i}(k_{32}-k_{22}\xoverline{u_i})\lb\rho_i\rb \lb u_i\rb - \left(k_{42}-\xoverline{u_i}k_{32}-k_{22}\left(\frac{\Cv_i}{\hat{\theta_i}} - \frac{\xoverline{u^2_i}}{2} \right)\right)\lb\rho_i\rb \lb \theta_i\rb\\
    &- (k_{43}-\xoverline{u_i}k_{33})\lb u_i\rb\lb\theta_i\rb \\
		\equaltext{(\cref{Eqn: ES unknowns})}  \sum^2_{i=1} & k_{22}(\gamma_i-1)\Cv_i\lb\rho_i\rb\lb\ln\rho_i\rb +k_{33}\xoverline{\theta_i}\lb u_i\rb^2-k_{44}\lb T_i\rb\lb\theta_i\rb \geqslant 0.
    \end{aligned}
\end{equation*}
\end{linenomath*}

\end{proof}

Using dimensional arguments, we define $k_{33} = \xoverline{\rho}_i k_{22}$ and $k_{44} = \xoverline{\rho}_i\Cv_i k_{22}$, and $k_{22}=\tfrac{\epsilon_\nu}{2}\max\big(\rho_\textbf{A}(\vecu^-),\rho_\textbf{A}(\vecu^+)\big)$, with $\epsilon_\nu\geqslant0$ and $\rho_\textbf{A}(\vecu)$ the spectral radius of $\textbf{A}(\vecu)$ in (\cref{Eqn: compact BNM}), to get the following numerical dissipation
\begin{linenomath*}
\begin{equation}\label{eqn final dissipation}
    \normalfont{\textbf{D}}_\nu (\vecu^-,\vecu^+)= \frac{\epsilon_\nu}{2}\max\big(\rho_\textbf{A}(\vecu^-),\rho_\textbf{A}(\vecu^+)\big)\begin{pmatrix} 
    0 \\ \lb\rho_i\rb \\ \lb\rho_iu_i\rb \\ \Big(\frac{\Cv_i}{\hat{\theta_i}} + \frac{u^-_iu^+_i}{2}\Big)\lb\rho_i\rb + \xoverline{\rho}_i\lb E_i\rb\end{pmatrix}.
\end{equation}
\end{linenomath*}

\begin{remark}
Nonconservative systems may admit shocks which depend on small scale mechanisms such as viscosity and that numerical methods may fail to capture because the leading viscosity terms in the equivalent equation do not match these mechanisms \cite{lefloch2014numerical}. The jump conditions indeed depend on the family of paths prescribed in the jump relations which should be consistent with the viscous profile. Using (\cref{eqn final dissipation}) the decay rate for the cell-averaged entropy (\cref{semi-discrete entropy ineq}) reads
\begin{linenomath*}
\begin{equation*}
        h\frac{d\langle\eta(\vecu_h)\rangle_j}{dt}+Q(\normalfont{\textbf{U}}^p_j,\normalfont{\textbf{U}}^0_{j+1})-Q(\normalfont{\textbf{U}}^p_{j-1},\normalfont{\textbf{U}}^0_{j}) = - \frac{\epsilon_\nu}{2} \sum^2_{i=1} \frac{(\gamma_i-1)\Cv_i\lb\rho_i\rb^2}{\hat{\rho}_i} +\xoverline{\rho_i}\xoverline{\theta_i}\lb u_i\rb^2-\xoverline{\rho_i}\Cv_i\lb T_i\rb\lb\theta_i\rb \leqslant 0,
\end{equation*}
\end{linenomath*}
where the two last terms in the RHS are analogous to the ones in the physical model \cite{gallouet2004numerical} for a Prandtl number $Pr_i = 3\gamma_i /4$:
\begin{linenomath*}
\begin{equation*}
    \partial_t\eta(\vecu) + \partial_xq(\vecu) = -\sum_i\frac{4\mu_i}{3}\left(\theta_i(\partial_x u_i)^2 - \frac{3\Cp_i}{4 Pr_i}\partial_x T_i\partial_x\theta_i \right),
\end{equation*}
\end{linenomath*}
and $\mu_i > 0$ is the dynamic viscosity coefficient and are therefore consistent with the small scale mechanisms. The first term in the RHS was seen to improve stability and robustness of the computations despite its lack of physical relevance.
\end{remark}

%
%
\section{Properties of the high-order DGSEM scheme for the Baer-Nunziato model}\label{sec: properties of the scheme}
\subsection{Kinetic energy preservation}
The equation for the kinetic energy of the model (\cref{Eqn: BNM}) can be derived from the mass and momentum equations:
\begin{linenomath*}
\begin{equation*}
    \partial_t K_i +\partial_x K_iu_i + u_i\partial_x\alpha_i\press_i - \pI u_i\partial_x\alpha_i = 0, \quad i=1,2,
\end{equation*}
\end{linenomath*}
where $K_i = \frac{1}{2}\alpha_i\rho_iu^2_i$ is the partial kinetic energy of the $i$th phase. These equations contain nonconservative terms of pressure work and energy transfer between the phases. The property of kinetic energy preservation by numerical schemes was introduced in \cite{jameson2008formulation} for the compressible Euler equations, where a general condition was provided to impose kinetic energy preservation for finite volume schemes, and was seen to be useful in turbulent flow simulations. Kinetic energy preservation was later extended to high-order nodal DG schemes in \cite{gassner2014kinetic,gassner2016split} and we refer to \cite{kuya2018kinetic} for split forms of the convective terms in the compressible Euler equations that lead to kinetic energy preserving schemes. According to \cite[Theorem~2]{gassner2016split} it is sufficient to show that the volume terms of the advective part of the cell-averaged kinetic energy can be written in conservation form.

\begin{mythm}\label{thm: proof of KE preservation}
 The discretization of the volume integral in (\cref{eqn: residual}) with the numerical fluxes (\cref{eqn: EC fluxes}) is kinetic energy preserving.
\end{mythm}

\begin{proof}
 Let us consider the time derivative and volume term of the advective parts of the mass and momentum equations of phase $i=1,2$ in (\cref{eqn: residual}). Using (\cref{eqn: EC fluxes}) they read
\begin{linenomath*}
\begin{equation*}
 \Delta K_{i,j}^{\alpha\rho,k}=\tfrac{\omega_kh}{2}d_t(\alpha_{i,j}^k\rho_{i,j}^k)+\sum_{l=0}^p2\omega_kD_{kl}h_i^{\alpha\rho}({\bf U}_j^k,{\bf U}_j^l), \quad \Delta K_{i,j}^{\alpha\rho u,k}=\tfrac{\omega_kh}{2}d_t(\alpha_{i,j}^k\rho_{i,j}^ku_{i,j}^k)+\sum_{l=0}^p2\omega_kD_{kl}\tfrac{u_{i,j}^k+u_{i,j}^l}{2}h_i^{\alpha\rho}({\bf U}_j^k,{\bf U}_j^l),
\end{equation*}
\end{linenomath*}
with $h_i^{\alpha\rho}({\bf u}^-,{\bf u}^+)=\tfrac{1}{2}\big(h_{\rho_i}({\bf u}^-,{\bf u}^+)+h_{\rho_i}({\bf u}^+,{\bf u}^-)\big)\overset{(\cref{eqn: EC fluxes})}{=}\xoverline{\alpha}_i\xoverline{u}_i\hat\rho_i$. Introducing $K_{i,j}^k=\tfrac{1}{2}\alpha_{i,j}^k\rho_{i,j}^k(u_{i,j}^k)^2$, we have
\begin{linenomath*}
\begin{align*}
 \sum_{k=0}^p u_{i,j}^k\Delta K_{i,j}^{\alpha\rho u,k}-\tfrac{(u_{i,j}^k)^2}{2}\Delta K_{i,j}^{\alpha\rho,k} &= \sum_{k=0}^p \tfrac{\omega_kh}{2}d_t(K_{i,j}^k) + \sum_{k,l=0}^p2\omega_kD_{kl}\big(u_{i,j}^k\tfrac{u_{i,j}^k+u_{i,j}^l}{2}-\tfrac{(u_{i,j}^k)^2}{2}\big)h_i^{\alpha\rho}({\bf U}_j^k,{\bf U}_j^l) \\
=& d_t\langle K_i(\vecu_h)\rangle_j + \sum_{k,l=0}^p2\omega_kD_{kl}\tfrac{u_{i,j}^ku_{i,j}^l}{2}h_i^{\alpha\rho}({\bf U}_j^k,{\bf U}_j^l) \\
\overset{(\cref{eqn: SBP})}{=}& d_t\langle K_i(\vecu_h)\rangle_j + \sum_{k,l=0}^p\omega_kD_{kl}\tfrac{u_{i,j}^ku_{i,j}^l}{2}h_i^{\alpha\rho}({\bf U}_j^k,{\bf U}_j^l) - \sum_{k,l=0}^p\omega_lD_{lk}\tfrac{u_{i,j}^ku_{i,j}^l}{2}h_i^{\alpha\rho}({\bf U}_j^k,{\bf U}_j^l) + u_{i,j}^pK_{i,j}^p - u_{i,j}^0K_{i,j}^0 \\
=&  d_t\langle K_i(\vecu_h)\rangle_j + u_{i,j}^pK_{i,j}^p - u_{i,j}^0K_{i,j}^0,
\end{align*}
\end{linenomath*}
by symmetry of $h_i^{\alpha\rho}({\bf u}^-,{\bf u}^+)$, which concludes the proof.
\end{proof}

\subsection{Positivity of the numerical solution}
High-order time integration is made through the use of strong stability-preserving explicit Runge-Kutta schemes \cite{shu1988efficient} that are convex combinations of explicit first-order schemes in time. Therefore, we focus on the fully discrete scheme by using a one-step first-order explicit time discretization.

We use the notation $t^{(n)}=n\Delta t$ with $\Delta t>0$ the time step, and set $\lambda=\frac{\Delta t}{h}$, $\vecu^{(n)}_h(\cdot)= \vecu_h(\cdot, t^{(n)})$ and $\textbf{U}^{k,n}_j=\textbf{U}^k_j(t^{(n)})$. The fully discrete scheme reads
\begin{linenomath*}
\begin{equation} \label{Eqn: fully discrete residual}
    \frac{\omega_k}{2}(\textbf{U}^{k,n+1}_j-\textbf{U}^{k,n}_j) + \lambda\textbf{R}^k_j (\vecu^{(n)}_j) =0,
\end{equation}
\end{linenomath*}
where $\textbf{R}^k_j (\cdot)$ is defined in (\cref{eqn: residual}). Our analysis of the discrete scheme provides conditions on the numerical parameters that guarantee the positivity of the cell-averaged partial densities and a maximum principle on the cell-averaged void fraction. Unfortunately, we were not able to derive conditions for positivity of the partial internal energies, i.e., $\rho_ie_i>\mathrm{p}_{i,\infty}$, and we refer to \cite{coquel2017positive} for a first-order scheme that guaranties such condition. 

\begin{mythm}\label{thm: positivity}
Assume that $\rho^{0\leqslant k\leqslant p,n}_{i,j\in \mathbb{Z}} >0$, $\alpha^{0\leqslant k\leqslant p,n}_{i,j\in \mathbb{Z}} >0$ for $i=1,2$ and let $\beta_{s}$, in (\cref{eqn: EC fluxes}), be locally defined at element interfaces, then under the CFL condition
\begin{linenomath*}
\begin{equation}\label{CFL condition}
\begin{aligned}
    \lambda \max_{j\in\mathbb{Z}} \max_{i=1,2} \Bigg( &\max_{0\leqslant k\leqslant p} \frac{1}{\omega_k} \Bigg(\langle \uI_h^{(n)},d_x\phi^k_j\rangle^p_j + \delta_{kp}\frac{\beta_{s_{j+1/2}}-\uI^{p,n}_{j}}{2} + \delta_{k0}\frac{\beta_{s_{j-1/2}}+\uI^{0,n}_{j}}{2}\Bigg),\\
    &\frac{1}{\omega_0} \Bigg(\frac{(\beta_{s_{j-1/2}}-\xoverline{u}_{i,j-1/2})\hat{\rho}_{i,j-1/2}}{2\rho^{0,n}_{i,j}} + \frac{\epsilon_{\nu_{j-1/2}}}{\alpha^{0,n}_{i,j}}\Bigg), \frac{1}{\omega_p} \Bigg(\frac{(\beta_{s_{j+1/2}}+\xoverline{u}_{i,j+1/2})\hat{\rho}_{i,j+1/2}}{2\rho^{p,n}_{i,j}} + \frac{\epsilon_{\nu_{j+1/2}}}{\alpha^{p,n}_{i,j}}\Bigg) \Bigg) < \frac{1}{2},
\end{aligned}    
\end{equation}
\end{linenomath*}
where $\xoverline{u}_{i,j+1/2} = \tfrac{u^{p,n}_{i,j} + u^{0,n}_{i,j+1}}{2}$, $\hat{\rho}_{i,j+1/2}=\tfrac{\rho^{0,n}_{i,j+1}-\rho^{p,n}_{i,j}}{\ln{\rho^{0,n}_{i,j+1}}-\ln{\rho^{p,n}_{i,j}}}$, and
\begin{linenomath*}
\begin{equation}\label{eqn: beta_s}
\beta_{s_{j+1/2}} := \max_{i=1,2}(|u^{p,n}_{i,j}|,|u^{0,n}_{i,j+1}|),
\end{equation}
\end{linenomath*}
we have for the cell averaged solution at time $t^{(n+1)}$
\begin{linenomath*}
\begin{equation*}
    \langle\alpha_{i,h}\rho_{i,h}\rangle^{(n+1)}_j > 0, \quad \langle\alpha_{i,h}\rangle^{(n+1)}_j >0, \quad i=1,2, \quad j\in \mathbb{Z}.
\end{equation*}
\end{linenomath*}

Furthermore,
\begin{linenomath*}
\begin{equation}\label{discrete alpha1}
\begin{aligned}
    \langle\alpha_{i,h}\rangle^{(n+1)}_j = &\sum^p_{k=0} \left(\frac{\omega_k}{2}-\lambda\left(\langle \uI^{(n)}_{h},d_x\phi^k_j\rangle^p_j + \delta_{kp}\frac{\beta_{s_{j+1/2}}-\uI^{p,n}_{j}}{2} + \delta_{k0}\frac{\beta_{s_{j-1/2}}+\uI^{0,n}_{j}}{2} \right)\right)^{(n)}\alpha^{k,n}_{i,j} \\
    &+ \lambda\frac{\beta_{s_{j+1/2}}-\uI^{p,n}_{j}}{2}\alpha^{0,n}_{i,j+1} + \lambda\frac{\beta_{s_{j-1/2}}+\uI^{0,n}_{j}}{2}\alpha^{p,n}_{i,j-1}
\end{aligned}
\end{equation}
\end{linenomath*}
is a convex combination of DOFs at time $t^{(n)}$.
\end{mythm}

\begin{proof}
Summing over $0\leqslant k\leqslant p$ the first component of (\cref{Eqn: fully discrete residual}) for the void fraction we obtain
\begin{linenomath*}
\begin{equation*}\label{eqn: euler alpha}
\begin{aligned}
    \langle \alpha_{i,h}\rangle^{(n+1)}_j :=& \sum^p_{k=0} \frac{\omega_k}{2} \alpha^{k,n+1}_{i,j}\\
    =& \sum^p_{k=0} \frac{\omega_k}{2} \alpha^{k,n}_{i,j}- \lambda \Big( \sum_{l=0}^p\omega_kD_{kl}\uI^{k,n}_{j}\alpha^{l,n}_{i,j}  + \delta_{kp} \frac{\uI^{p,n}_{j}-\beta_{s_{j+1/2}}}{2}(\alpha^{0,n}_{i,j+1}-\alpha^{p,n}_{i,j}) + \delta_{k0}\frac{\uI^{0,n}_{j}+\beta_{s_{j-1/2}}}{2}(\alpha^{0,n}_{i,j}-\alpha^{p,n}_{i,j-1}) \Big)\\
    \overset{(\cref{eq:1d_inner_product})}{=}& \sum^{p-1}_{k=1}\left(\frac{\omega_k}{2}-\lambda\langle \uI^{(n)}_{h},d_x\phi^k_j\rangle^p_j\right)\alpha^{k,n}_{i,j} + \left(\frac{\omega_0}{2}-\lambda\left(\langle \uI^{(n)}_{h},d_x\phi^0_j\rangle^p_j+\frac{\beta_{s_{j-1/2}}+\uI^{0,n}_{j}}{2}\right)\right)\alpha^{0,n}_{i,j} \\
    &+\left(\frac{\omega_p}{2} -\lambda\left(\langle \uI^{(n)}_{h},d_x\phi^p_j\rangle^p_j+\frac{\beta_{s_{j+1/2}}-\uI^{p,n}_{j}}{2} \right)\right)\alpha^{p,n}_{i,j} + \lambda\frac{\beta_{s_{j-1/2}}+\uI^{0,n}_{j}}{2}\alpha^{p,n}_{i,j-1} + \lambda\frac{\beta_{s_{j+1/2}}-\uI^{p,n}_{j}}{2}\alpha^{0,n}_{i,j+1},
\end{aligned}
\end{equation*}
\end{linenomath*}
which is a convex combination of DOFs at time $n$ with (\cref{eqn: beta_s}) and the following restriction on the time-step:
\begin{linenomath*}
\begin{equation*}\label{eqn: CFL condn alpha}
   \lambda \left(\langle \uI^{(n)}_{h},d_x\phi^k_j\rangle^p_j + \delta_{kp}\frac{\beta_{s_{j+1/2}} - \uI^{p,n}_{j}}{2} + \delta_{k0}\frac{\beta_{s_{j-1/2}} + \uI^{0,n}_{j}}{2} \right) < \frac{\omega_k}{2}, \quad 0\leqslant k\leqslant p,
\end{equation*}
\end{linenomath*}
since from (\cref{eqn: intref_var}) we have $\beta_{s_{j+1/2}} \geqslant \max(|\uI^{p,n}_{j}|,|\uI^{0,n}_{j+1}|)$.

For the cell-averaged partial densities, we use a similar technique to \cite{zhang2010positivity,perthame1996positivity} and sum over $0\leqslant k\leqslant p$ the second component in (\cref{Eqn: fully discrete residual}) for the partial densities to get
\begin{linenomath*}
\begin{equation*}
    \displaystyle
    \begin{aligned}
    \langle\alpha_{i,h}\rho_{i,h}\rangle^{(n+1)}_j =& \sum^{p}_{k=0} \frac{\omega_k}{2} \alpha^{k,n}_{i,j}\rho^{k,n}_{i,j} - \lambda\Bigg( \bigg(\xoverline{u}_{i,j+1/2} \frac{\alpha^{p,n}_{i,j} + \alpha^{0,n}_{i,j+1}}{2} -\frac{\beta_{s_{j+1/2}}}{2}\big(\alpha^{0,n}_{i,j+1}-\alpha^{p,n}_{i,j}\big)\bigg)\hat{\rho}_{i,j+1/2} - \epsilon_{\nu_{i,j+1/2}}\big(\rho^{0,n}_{i,j+1} - \rho^{p,n}_{i,j}\big)\Bigg)\\
		&+ \lambda\Bigg(\bigg(\xoverline{u}_{i,j-1/2}\frac{\alpha^{p,n}_{i,j-1}+\alpha^{0,n}_{i,j}}{2}-\frac{\beta_{s_{j-1/2}}}{2}\big(\alpha^{0,n}_{i,j} - \alpha^{p,n}_{i,j-1}\big)\bigg)\hat{\rho}_{i,j-1/2} - \epsilon_{\nu_{i,j-1/2}}\big(\rho^{0,n}_{i,j}-\rho^{p,n}_{i,j-1}\big)\Bigg) \\
    =& \sum^{p-1}_{k=1} \frac{\omega_k}{2} \alpha^{k,n}_{i,j}\rho^{k,n}_{i,j} \\
		&+ \Bigg(\frac{\omega_p}{2} - \lambda \bigg(\frac{\beta_{s_{j+1/2}}+\xoverline{u}_{i,j+1/2}}{2}\frac{\hat{\rho}_{i,j+1/2}}{\rho^{p,n}_{i,j}} + \frac{\epsilon_{\nu_{i,j+1/2}}}{\alpha^{p,n}_{i,j}}\bigg) \Bigg)\alpha^{p,n}_{i,j}\rho^{p,n}_{i,j} + \lambda\bigg(\frac{\beta_{s_{j+1/2}}-\xoverline{u}_{i,j+1/2}}{2}\frac{\hat{\rho}_{i,j+1/2}}{\rho^{0,n}_{i,j+1}} + \frac{\epsilon_{\nu_{i,j+1/2}}}{\alpha^{0,n}_{i,j+1}}\bigg)\alpha^{0,n}_{i,j+1}\rho^{0,n}_{i,j+1} \\
		&+ \Bigg(\frac{\omega_0}{2} - \lambda \bigg(\frac{\beta_{s_{j-1/2}}-\xoverline{u}_{i,j-1/2}}{2}\frac{\hat{\rho}_{i,j-1/2}}{\rho^{0,n}_{i,j}} + \frac{\epsilon_{\nu_{i,j-1/2}}}{\alpha^{0,n}_{i,j}}\bigg) \Bigg)\alpha^{0,n}_{i,j}\rho^{0,n}_{i,j} + \lambda\bigg(\frac{\beta_{s_{j-1/2}}+\xoverline{u}_{i,j-1/2}}{2}\frac{\hat{\rho}_{i,j-1/2}}{\rho^{p,n}_{i,j-1}} + \frac{\epsilon_{\nu_{i,j-1/2}}}{\alpha^{p,n}_{i,j-1}}\bigg)\alpha^{p,n}_{i,j-1}\rho^{p,n}_{i,j-1} \\
    \end{aligned}
\end{equation*}
\end{linenomath*}

and is positive if 
\begin{linenomath*}
\begin{equation*}\label{eqn: CFL condn partial density}
 \lambda \bigg(\frac{\beta_{s_{j-1/2}}-\xoverline{u}_{i,j-1/2}}{2}\frac{\hat{\rho}_{i,j-1/2}}{\rho^{0,n}_{i,j}} + \frac{\epsilon_{\nu_{i,j-1/2}}}{\alpha^{0,n}_{i,j}}\bigg) \leqslant\frac{\omega_0}{2}, \quad \lambda \bigg(\frac{\beta_{s_{j+1/2}}+\xoverline{u}_{i,j+1/2}}{2}\frac{\hat{\rho}_{i,j+1/2}}{\rho^{p,n}_{i,j}} + \frac{\epsilon_{\nu_{i,j+1/2}}}{\alpha^{p,n}_{i,j}}\bigg) \leqslant\frac{\omega_p}{2},
\end{equation*}
\end{linenomath*}

provided $\epsilon_{\nu_{i,j\pm 1/2}} \geqslant 0$ and (\cref{eqn: beta_s}).
\end{proof}

\subsection{A posteriori limiters}
The properties of \cref{thm: positivity} hold only for the cell averaged value of the numerical solution at time $t^{(n+1)}$, which can be extended to nodal values by using a posteriori limiters \cite{zhang2010maximum, zhang2010positivity}. We here limit the void fraction with the bounds of its initial value over the whole domain, while we enforce positivity of the partial densities, similar to \cite{renac2019entropy}. The limiter reads
\begin{linenomath*}
\begin{equation}\label{eq:limiter}
    \Tilde{\textbf{U}}^{k,n+1}_j = \theta_j \big(\textbf{U}^{k,n+1}_j - \langle \vecu_h\rangle^{(n+1)}_j\big) + \langle \vecu_h\rangle^{(n+1)}_j, \quad 0\leqslant k\leqslant p, \quad j\in\mathbb{Z},
\end{equation}
\end{linenomath*}
with $0\leqslant \theta_j \leqslant 1$ defined by $\theta_j:=\min (\theta^{\;\rho_i}_j, \theta^{\;\alpha_i}_j : i = 1,2)$ where
\begin{linenomath*}
\begin{equation}\label{eqn: limiter}
    \begin{aligned}
        \theta^{\;\rho_i}_j &= \min \left(\frac{\langle \alpha_{i,h}\rho_{i,h}\rangle^{(n+1)}_j-\epsilon}{\langle \alpha_{i,h}\rho_{i,h}\rangle^{(n+1)}_j -  (\alpha_i\rho_i)^{min}_{j}}, 1 \right), \quad (\alpha_i\rho_i)^{min}_{j} = \min_{0\leqslant k\leqslant p} (\alpha_{i}\rho_{i})^{k,n+1}_j,\\
        \theta^{\;\alpha_i}_j &= \min \left(\frac{\langle \alpha_{i,h}\rangle^{(n+1)}_j-m^{\alpha}_{i,j}}{\langle \alpha_{i,h}\rangle^{(n+1)}_j -  \alpha^{min}_{i,j}}, \frac{M^{\alpha}_{i,j}-\langle\alpha_{i,h}\rangle^{(n+1)}_j}{\alpha^{max}_{{i,j}}-\langle \alpha_{i,h}\rangle^{(n+1)}_j}, 1 \right), \quad \alpha^{min}_{i,j} = \min_{0\leqslant k\leqslant p} \alpha^{k,n+1}_{i,j}, \quad \alpha^{max}_{i,j} = \max_{0\leqslant k\leqslant p} \alpha^{k,n+1}_{i,j},
    \end{aligned}
\end{equation}
\end{linenomath*}
$0< \epsilon \ll 1$ is a parameter (we set $\epsilon = 10^{-8}$ in our numerical tests), and 
\begin{linenomath*}
\begin{equation*}
    m^{\alpha}_{i,j} = \min_{j\in\mathbb{Z}}\min_{0\leqslant k\leqslant p} \alpha^{k,0}_{i,j}, \quad M^{\alpha}_{i,j} = \max_{j\in\mathbb{Z}}\max_{0\leqslant k\leqslant p}\alpha^{k,0}_{i,j}.
\end{equation*}
\end{linenomath*}
The limiter (\cref{eqn: limiter}) guarantees that $\Tilde{\rho}^{0\leqslant k\leqslant p,n+1}_j >0$ together with the following bounds on the void fractions $m^\alpha_{i,j} \leqslant \Tilde{\alpha}^{0\leqslant k \leqslant p, n+1}_{i,j} \leqslant M^\alpha_{i,j}$. 

%
%
\section{Numerical tests in one space dimension} \label{sec: Numerical tests}
In this section we assess the high-order accuracy, robustness, and nonlinear stability of the numerical scheme for the Baer-Nunziato model by considering numerical tests for the initial value problem (\cref{Eqn: Cauchy prob}). We use $\uI = u_2$ and $\pI = \press_1$ as interfacial variables in (\cref{eqn: intref_var}). Unless stated otherwise, all numerical tests are performed with fourth order accuracy in space, $p=3$, on a unit domain $\Omega = [-0.5,0.5]$ discretized with a uniform mesh of 100 cells. The values of the numerical dissipation parameter $\epsilon_\nu$ in (\cref{eqn final dissipation}) lie in the range $[0.1,0.5]$. The time integration is performed by using the three-stage third-order strong stability-preserving Runge-Kutta scheme by Shu and Osher \cite{shu1988efficient}. The limiter (\cref{eq:limiter}) is applied at the end of each stage. The time step is computed through (\cref{CFL condition}). The numerical experiments of \cref{sec: Numerical tests,sec: extension to multi-space dim} have been obtained with the CFD code {\it Aghora} developed at ONERA \cite{renac2015aghora}.

\subsection{Advection of density and void fraction waves}\label{ssec: density wave}
We first test the high-order accuracy of the scheme (\cref{eqn: modified DG semi-discrete}). Let us consider a unit domain with periodic conditions and the following initial condition $\vecu_0(x)$
\begin{linenomath*}
\begin{equation*}
 \alpha_{1,0}(x) = \frac{1}{2} + \frac{1}{4} \sin (4\pi x),\quad
 \rho_{i,0}(x)   = 1 + \frac{1}{2} \sin (2\pi x),\quad
 u_{i,0}(x)      = 1,\quad
 \press_{i,0}(x) = 1, \quad i=1,2,
\end{equation*}
\end{linenomath*}
which results in a density wave and a void fraction wave with different frequencies and amplitudes that are purely advected in a uniform flow. The EOS parameters in (\cref{Eqn: EOS}) are $\gamma_i = 1.4$ and $\press_{\infty_i} = 10$ for $i=1,2$. 

\Cref{tab: high-order accuracy} indicates the values of the norms of the error on $\tfrac{1}{2}(\rho_{1}+\rho_2)$ obtained at final time $T_{max} = 5$ with different polynomial degrees and grid refinements, as well as the associated orders of convergence. We observe, as the mesh is refined, that the expected $p+1$ order of convergence is recovered with the present scheme. 
\begin{table}[ht]
    \centering
    \begin{tabulary}{1.0\textwidth}{ c|l|*{6}{c} }\hline
        $p$ & $h$ & $\norm{e_h}_{L^1(\Omega_h)}$ & $\mathcal{O}_1$ & $\norm{e_h}_{L^2(\Omega_h)}$ & $\mathcal{O}_2$ & $\norm{e_h}_{L^\infty(\Omega_h)}$ & $\mathcal{O}_\infty$\\ \hline
         \multirow{4}{*}{1}
         & 1/8  & 0.46E+00 & -    & 0.52E+00 & -    & 0.74E+00 & -\\
         & 1/16 & 0.18E+00 & 1.36 & 0.20E+00 & 1.39 & 0.30E+00 & 1.31\\
         & 1/32 & 0.40E-01 & 2.16 & 0.48E-01 & 2.04 & 0.98E-01 & 1.62\\
         & 1/64 & 0.92E-02 & 2.11 & 0.12E-01 & 2.02 & 0.32E-01 & 1.63\\ \hline
        \multirow{4}{*}{2}
        & 1/8   & 0.54E-01 & -    & 0.65E-01 & -    & 0.13E+00 & -\\
        & 1/16  & 0.88E-02 & 2.61 & 0.12E-01 & 2.45 & 0.27E-01 & 2.30\\
        & 1/32  & 0.16E-02 & 2.42 & 0.22E-02 & 2.46 & 0.48E-02 & 2.48\\
        & 1/64  & 0.21E-03 & 3.00 & 0.28E-03 & 2.99 & 0.63E-03 & 2.94\\ \hline
        \multirow{4}{*}{3}
        & 1/8  & 0.89E-02 & -    & 0.11E-01 & -    & 0.25E-01 & -\\
        & 1/16 & 0.34E-03 & 4.69 & 0.41E-03 & 4.77 & 0.76E-03 & 5.03\\
        & 1/32 & 0.18E-04 & 4.22 & 0.24E-04 & 4.12 & 0.52E-04 & 3.85\\
        & 1/64 & 0.11E-05 & 4.03 & 0.15E-05 & 4.01 & 0.32E-05 & 3.99\\ \hline
    \end{tabulary}
    \caption{Test for high-order accuracy: different norms of the error on densities under $p$- and $h$-refinements and associated orders of convergence.}
    \label{tab: high-order accuracy}
\end{table}

\subsection{Riemann Problems}
We now consider a series of Riemann problems from \cite{bohm2018entropy,tokareva2010hllc,coquel2017positive} to assess the entropy conservation, robustness, and stability properties of the present scheme. The initial condition reads
\begin{linenomath*}
\begin{equation*}
    \vecu_0(x)=
    \begin{cases}
    \vecu_L, \quad x<x_0,\\
    \vecu_R, \quad x>x_0.
    \end{cases}
\end{equation*}
\end{linenomath*}
\cref{table: RP_IC} contains the initial conditions for the different Riemann problems, while the physical parameters are given in \cref{table: EOS_param}.

\begin{table}[ht]
\centering
    \begin{tabulary}{1.0\textwidth}{ c|*{8}{c} }\hline
    Test case & &$\alpha_1$ & $\rho_1$ & $u_1$ & $\press_1$ & $\rho_2$ & $u_2$ & $\press_2$\\
    \hline
    \multirow{2}{*}{EC}
     &$\vecu_L$ & 0.5 & 1.0 & 0.0 & 1.0 & 1.0 & 0.0 & 1.0\\ 
     &$\vecu_R$ & 0.5 & 1.125 & 0.0 & 1.1 & 1.125 & 0.0 & 1.1\\ \hline
    \multirow{2}{*}{RP1}
     &$\vecu_L$ & 0.1 & 1.0 & 1.0 & 1.0 & 1.5 & 1.0 & 1.0\\ 
     &$\vecu_R$ & 0.9 & 2.0 & 1.0 & 1.0 & 1.0 & 1.0 & 1.0\\ \hline
     \multirow{2}{*}{RP2}
     &$\vecu_L$ & 0.8 &2.0 & 0.0 & 3.0 & 1900.0 & 0.0 & 10.0 \\  
     &$\vecu_R$ & 0.1 &1.0 & 0.0 & 1.0 & 1950.0 & 0.0 & 1000.0 \\ \hline
     \multirow{2}{*}{RP3}
     &$\vecu_L$ & 0.2 & 0.99988 & -1.99931 & 0.4 & 0.99988 & -1.99931 & 0.4\\ 
     &$\vecu_R$ & 0.5 & 0.99988 & 1.99931 & 0.4 & 0.99988 & 1.99931 & 0.4\\
     \hline
     \multirow{2}{*}{RP4}
     &$\vecu_L$ & 0.3 & 1.0 & -19.59741 & 1000.0 & 1.0 & -19.59716 & 1000.0\\ 
     &$\vecu_R$ & 0.8 & 1.0 & -19.59741 & 0.01 & 1.0 & -19.59741 & 0.01\\
     \hline
     \multirow{2}{*}{RP5}
     &$\vecu_L$ & 0.999 & 1.6 & 1.79057 &  5.0 & 2.0     & 1.0     & 10.0\\ 
     &$\vecu_R$ & 0.001 & 2.0 & 1.0     & 10.0 & 2.67183 & 1.78888 & 15.0\\
     \hline
    \end{tabulary}
    \caption{Initial conditions for the Riemann problems.}
    \label{table: RP_IC}
\end{table}

\begin{table}[ht]
\centering
    \begin{tabulary}{1\textwidth}{ c|*{7}{c} }\hline
                   & EC   & RP1  & RP2  & RP3   & RP4 & RP5\\ \hline
    $x_0$          & 0.0  & 0.0  & 0.0  & 0.0   & 0.3 & 0.0\\ 
    $T_{max}$      & 0.15 & 0.25 & 0.15   & 0.15 & 0.007 & 0.05\\
    $\gamma_1$     & 1.4  & 3.0  & 1.35    & 1.4  & 1.4   & 3.0\\
    $\gamma_2$     & 1.4  & 1.4  & 3.0     & 1.4  & 3.0   & 1.4\\
    $\press_{\infty_1}$ & 0.1  & 0.1  & 0.0    & 0.0  & 0.0   & 0.0\\
    $\press_{\infty_2}$ & 0.0  & 0.0  & 3400.0 & 0.0  & 100.0 & 0.0\\
    \hline
    \end{tabulary}
    \caption{Location of discontinuity on $\Omega_h$, final time, EOS parameters from (\cref{Eqn: EOS}).}
    \label{table: EOS_param}
\end{table}

\subsubsection{Test for entropy conservation}\label{sssec: entropy conservation}
The property of entropy conservation of the numerical fluxes (\cref{Eqn: fluctuation flux}) in the modified scheme (\cref{eqn: modified DG semi-discrete}) is validated based from the experimental setup introduced in \cite{bohm2018entropy}. Here we only focus on entropy conservative fluxes, so we choose $\epsilon_\nu = 0$ in (\cref{eqn final dissipation}). The initial condition corresponds to the test case EC in \cref{table: RP_IC} which generates discontinuities of moderate strength in each phase. We impose periodic boundary conditions and the global entropy should remain constant over the computational domain, while being modified only as a result of the time integration. We thus introduce the entropy budget
\begin{linenomath*}
\begin{equation}\label{eqn: entropy_budget}
    \mathcal{E}_{\Omega_h}(t):=h\big|\sum_{\kappa_j\in\Omega_h} \langle\eta(\vecu_h) \rangle_j - \langle\eta(\vecu_0)\rangle_j\big|,
\end{equation}
\end{linenomath*}
which evaluates the variations in the computation of the cell-averaged entropy over the domain $\Omega_h$. The results in \cref{table: EC_results} show that the error (\cref{eqn: entropy_budget}) decreases to machine accuracy when refining the time step, with the order of convergence corresponding to the theoretical approximation order of the time integration scheme. This validates the entropy conservation of the numerical fluxes (\cref{Eqn: fluctuation flux}).

\begin{table}[ht]
\centering
 \begin{tabular}{l |c |c} 
 \hline
 time step & $\mathcal{E}_{\Omega_h}(t)$ & $\mathcal{O}$ \\ [0.5ex] 
 \hline
 $\Delta t$    & 6.85E-06 & -- \\ 
 $\Delta t/2$  & 2.08E-06 & 2.94\\
 $\Delta t/4$  & 2.65E-07 & 2.97 \\
 $\Delta t/8$  & 3.31E-08 & 2.99 \\
 $\Delta t/16$ & 4.14E-09 & 3.00 \\ 
 $\Delta t/32$ & 5.14E-10 & 3.00 \\
 \hline
\end{tabular}
\caption{Global entropy budget and the corresponding order of convergence
$\mathcal{O}$ when refining the time step.}
\label{table: EC_results}
\end{table}

\subsubsection{Riemann problems}

The results of the Riemann problems in \cref{table: RP_IC} are shown in \crefrange{result: abgrall criterion}{result: RP5}, where we compare the numerical results with the exact solutions from \cite{tokareva2010hllc,coquel2017positive}. 

Here the test RP1 consists in the advection of a material interface in a uniform flow and the results in \cref{result: abgrall criterion} show that the velocity and pressure of both phases remain uniform in time which may be related to the so-called criterion of Abgrall \cite{abgrall1996prevent}. The observed smearing of the contact is a consequence of the limiter (\cref{eqn: limiter}) which is a common remark for all Riemann problems that we will encounter here.

The results for tests RP2 and RP 3 in \cref{result: RP2,result: RP3} contain the development of shocks, rarefaction and contacts in both phases. The scheme captures the correct solutions, but the intermediate states contain small oscillations at the shock and rarefaction waves in phase 1 of RP2. The scheme also proves to maintain the positivity of the partial densities in the near vacuum region of RP3, see \cref{result: RP3}. 

The capabilities of the scheme to resolve strong shocks are demonstrated in \cref{result: RP4} for the RP4 test case. Here the left-traveling rarefaction waves and the material discontinuity are well captured, whereas small oscillations are observed around the right-traveling shock in both phases. A possible reason could be that, as the dissipation is introduced in the numerical scheme through the interfaces, the internal DOFs may suffer from a lack of stabilization mechanism.

Finally, the test case RP5 probes the numerical scheme close to resonance (\cref{eqn: resonance}) mimicking pure phases separated by a material interface. Note that  we do not consider pure phases in this work and restrict ourselves to conditions close to resonance. Following \cite{coquel2017positive}, we indicate in \cref{result: RP5} the regions where the corresponding phases are relevant. The results show a correct approximation of the intermediate states where either phase exists, while spurious oscillations occur but in regions where the corresponding phase is absent.

\begin{figure}[H]
 \center
 \subfloat[$\alpha_{1}$]{
    \includegraphics[height=.20\paperwidth,trim=0.2cm 0.2cm 0.2cm 0.2cm,clip=true]{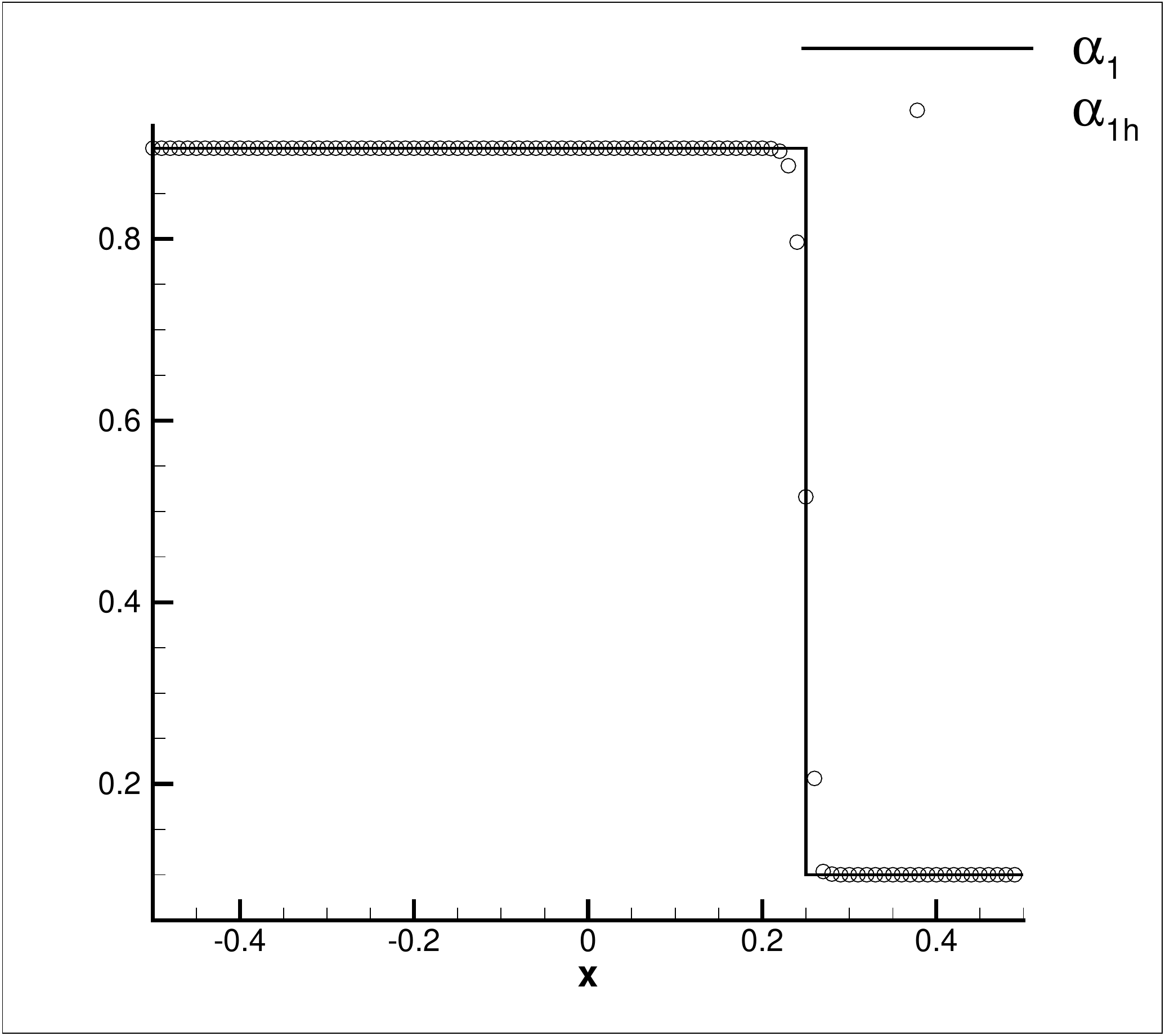}} \\
 \subfloat[$\rho_{1}$]{
  \includegraphics[height=.20\paperwidth,trim=0.2cm 0.2cm 0.2cm 0.2cm,clip=true]{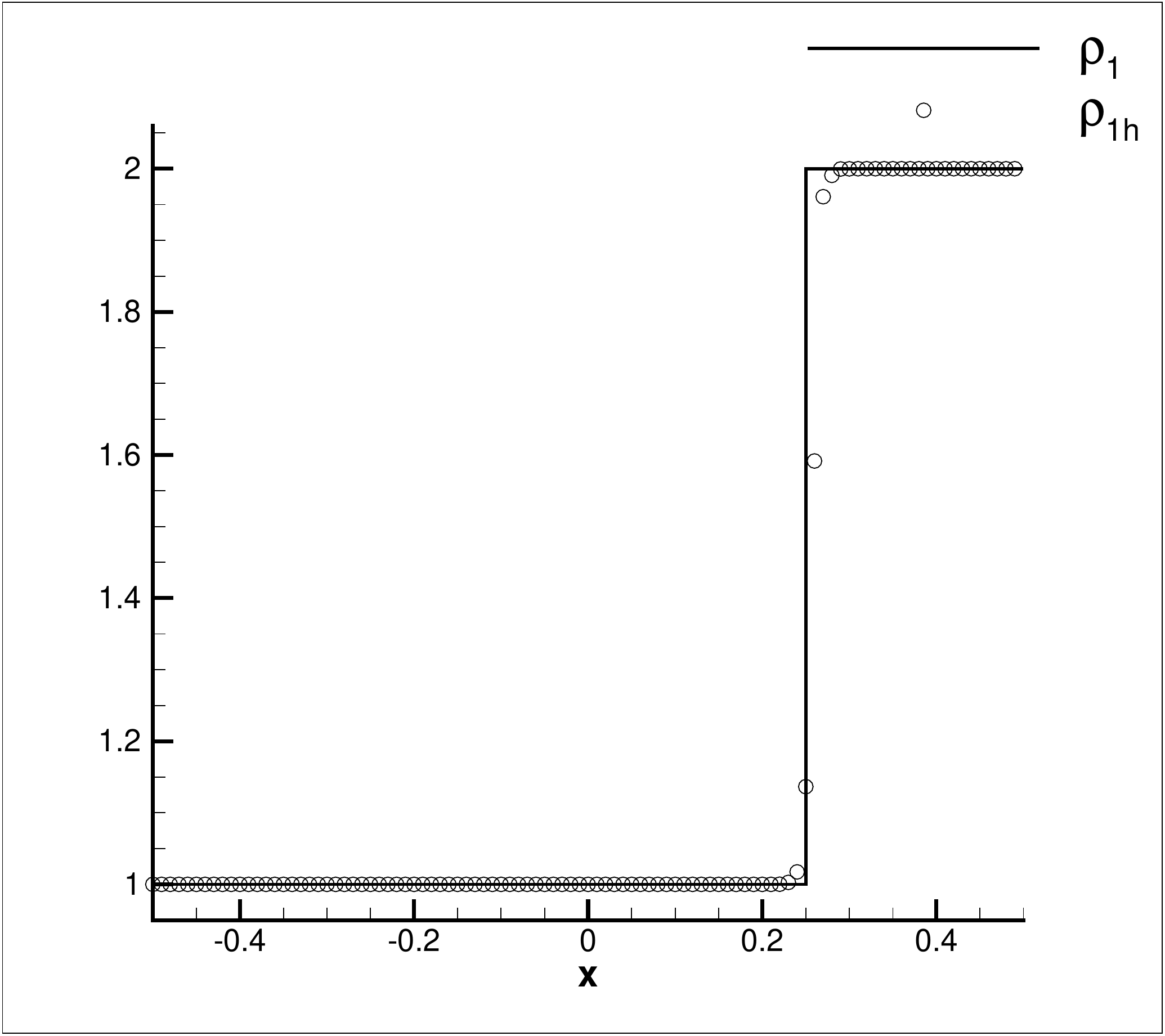}}
  \subfloat[$u_{1}$]{
  \includegraphics[height=.20\paperwidth,trim=0.2cm 0.2cm 0.2cm 0.2cm,clip=true]{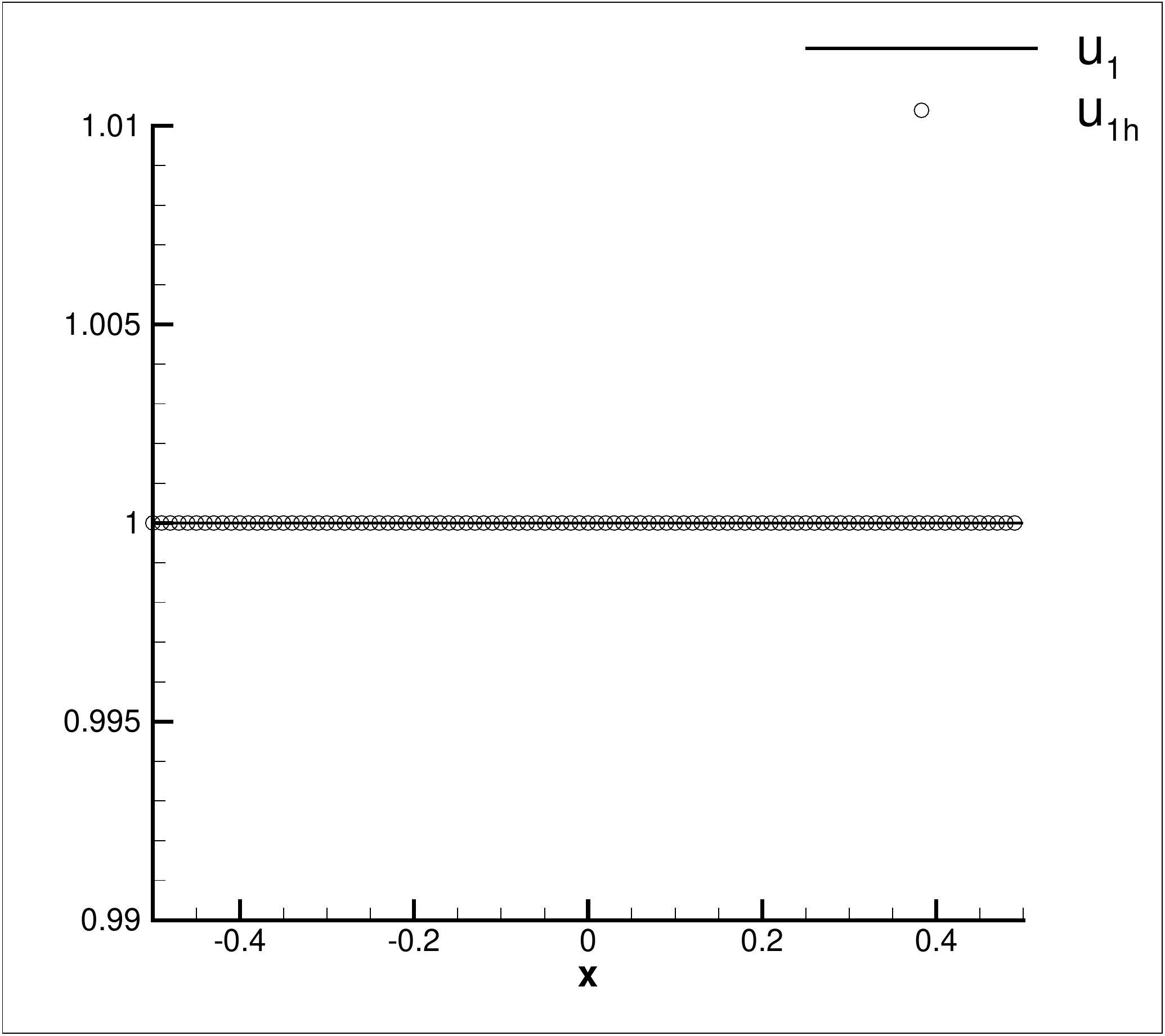}}
  \subfloat[$p_{1}$]{
  \includegraphics[height=.20\paperwidth,trim=0.2cm 0.2cm 0.2cm 0.2cm,clip=true]{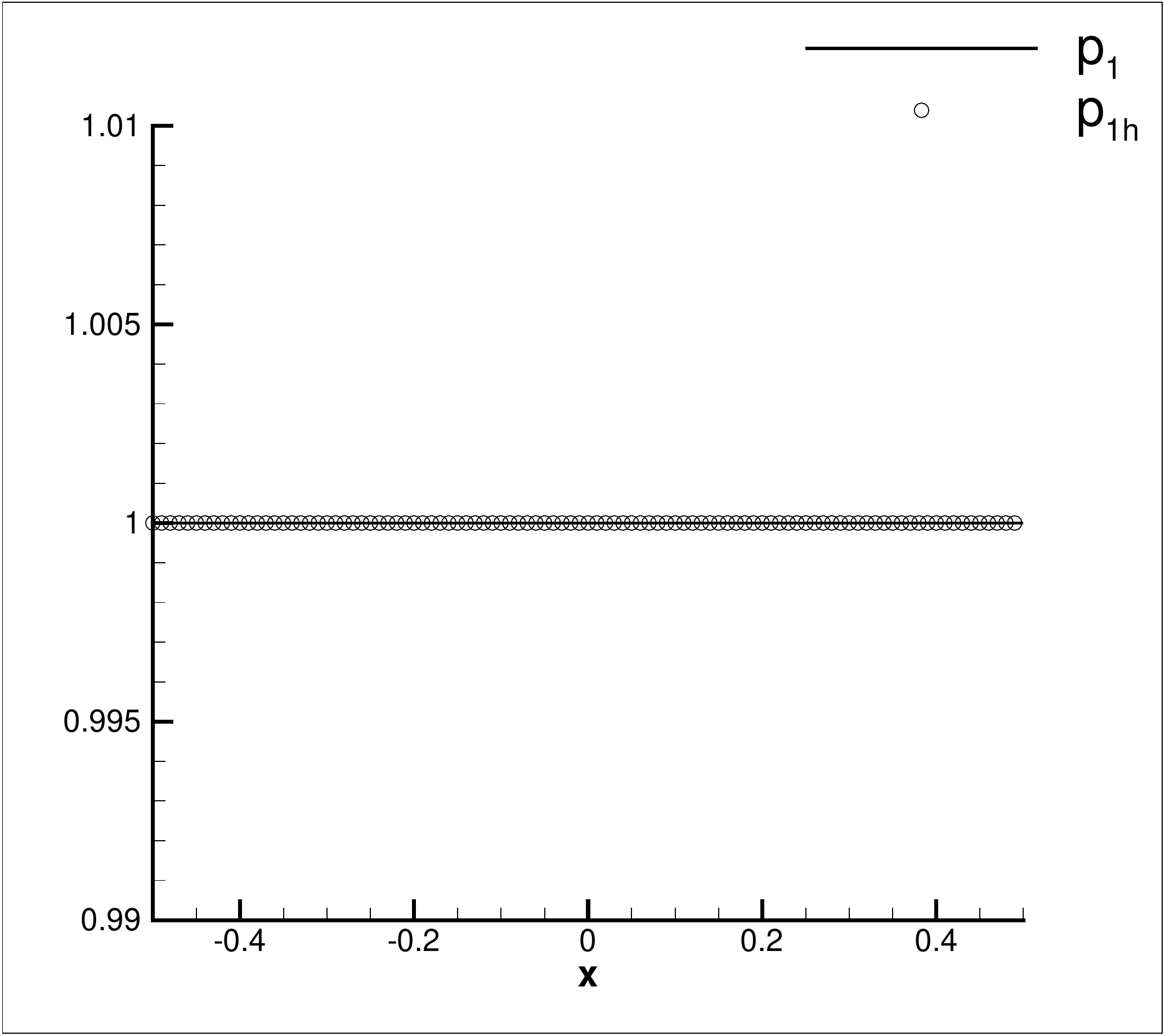}} \\
 \subfloat[$\rho_{2}$]{
  \includegraphics[height=.20\paperwidth,trim=0.2cm 0.2cm 0.2cm 0.2cm,clip=true]{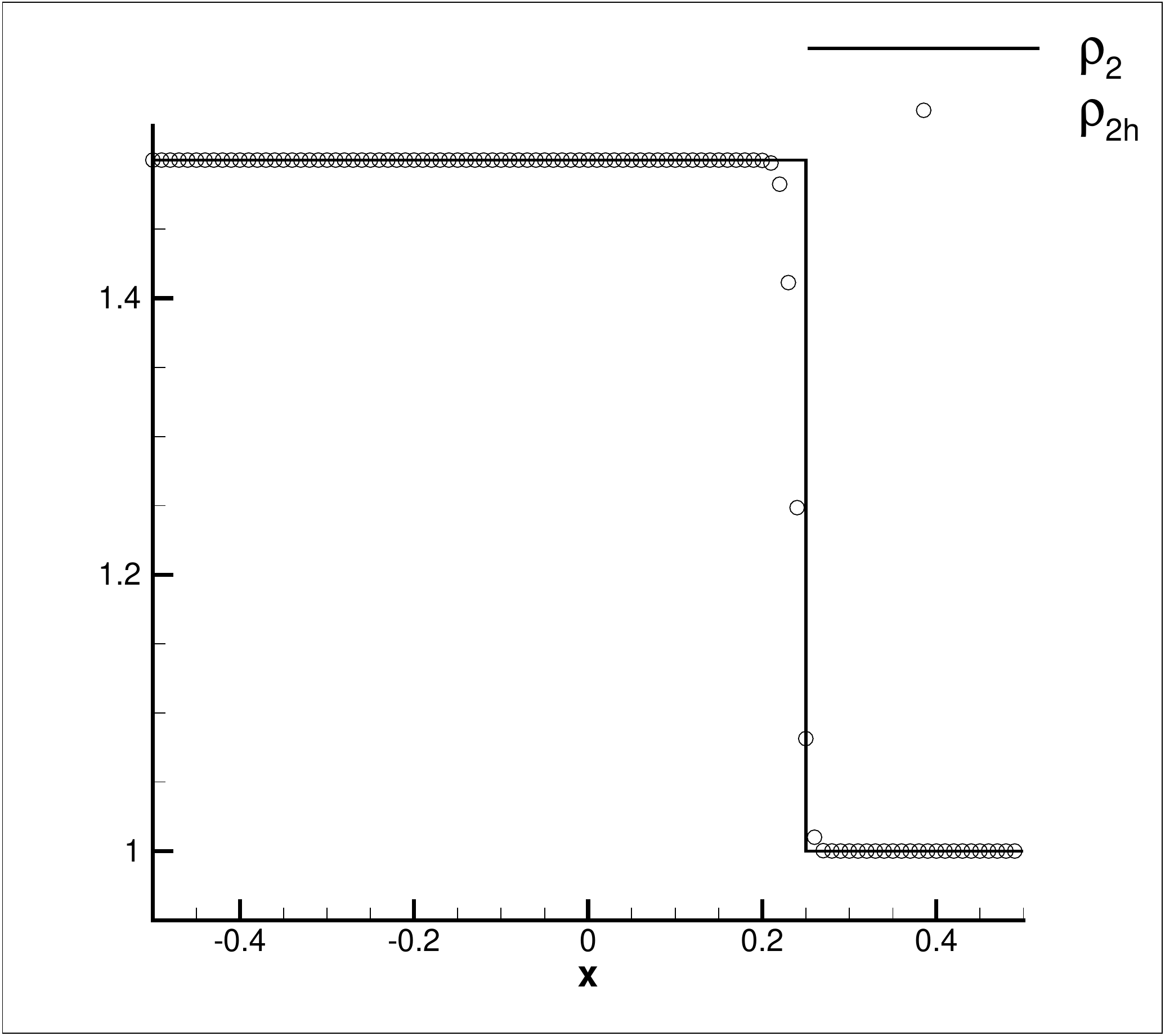}}
  \subfloat[$u_{2}$]{
  \includegraphics[height=.20\paperwidth,trim=0.2cm 0.2cm 0.2cm 0.2cm,clip=true]{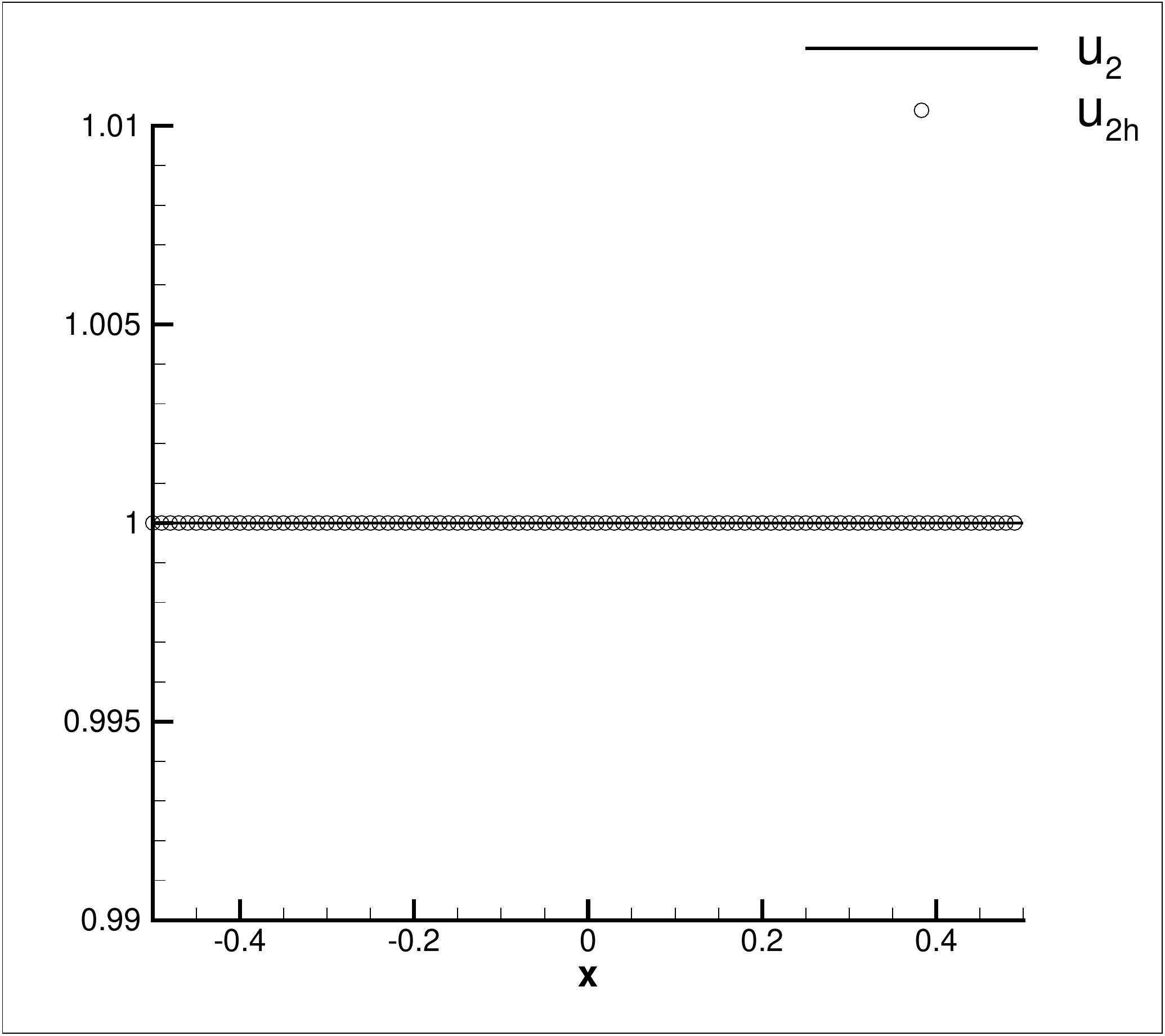}}
  \subfloat[$p_{2}$]{
  \includegraphics[height=.20\paperwidth,trim=0.2cm 0.2cm 0.2cm 0.2cm,clip=true]{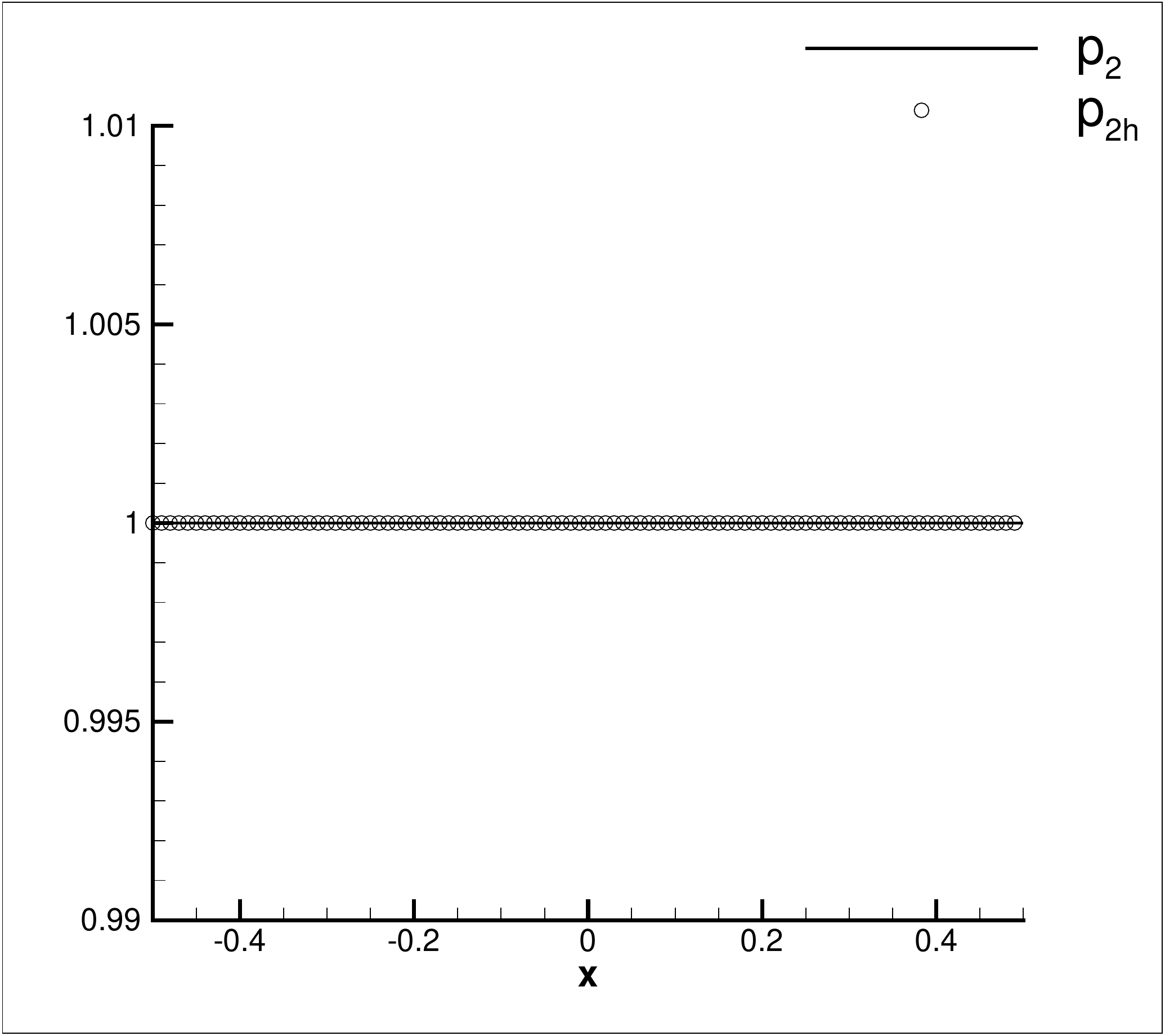}}
 \normalsize\caption{Comparison of the fourth order accurate numerical solution to the exact solution for test case RP1.}
 \label{result: abgrall criterion}
\end{figure}

\begin{figure}[H]
 \center
 \subfloat[$\alpha_{1}$]{
  \includegraphics[height=.20\paperwidth,trim=0.2cm 0.2cm 0.2cm 0.2cm,clip=true]{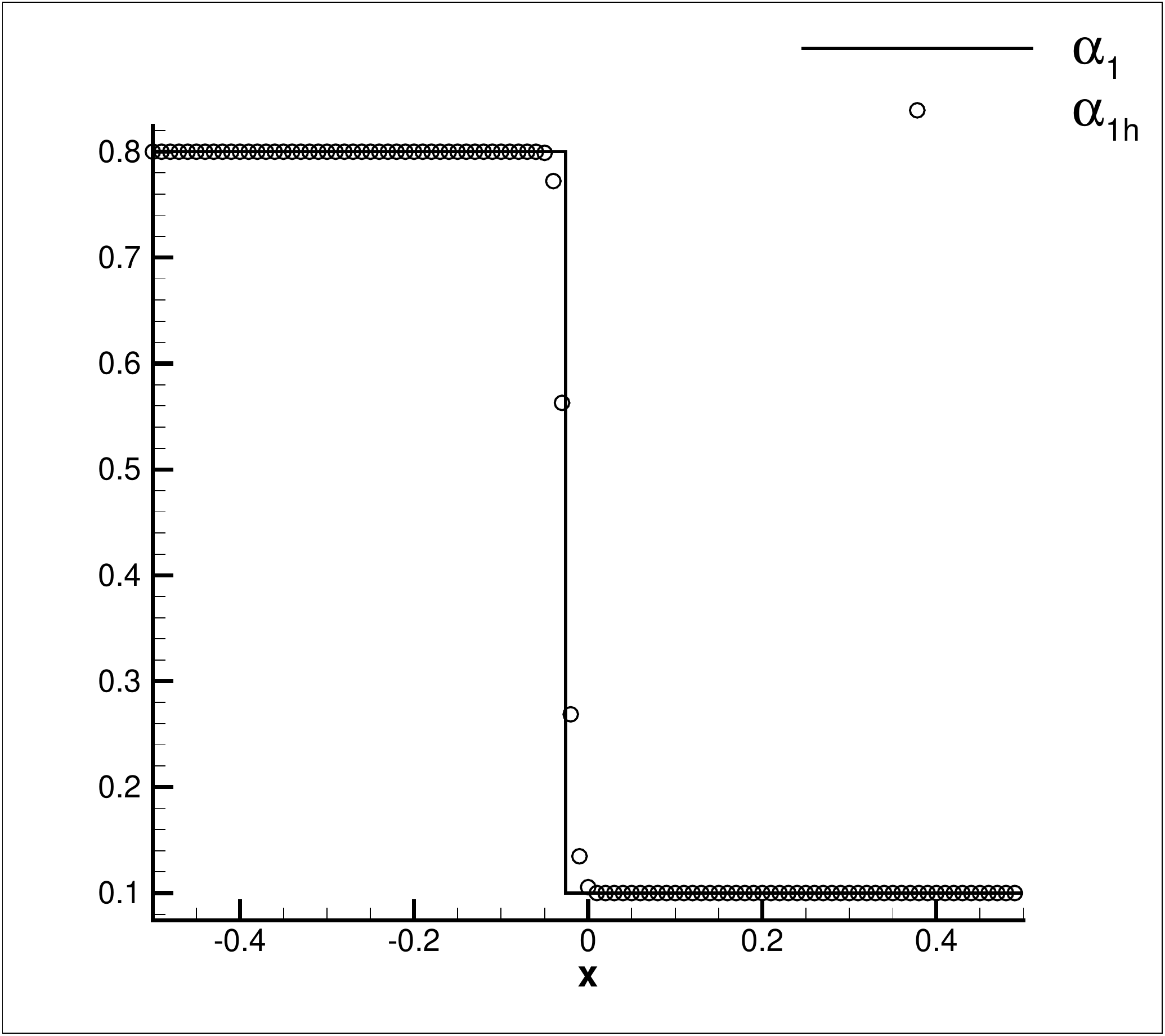}} \\
 \subfloat[$\rho_{1}$]{
  \includegraphics[height=.20\paperwidth,trim=0.2cm 0.2cm 0.2cm 0.2cm,clip=true]{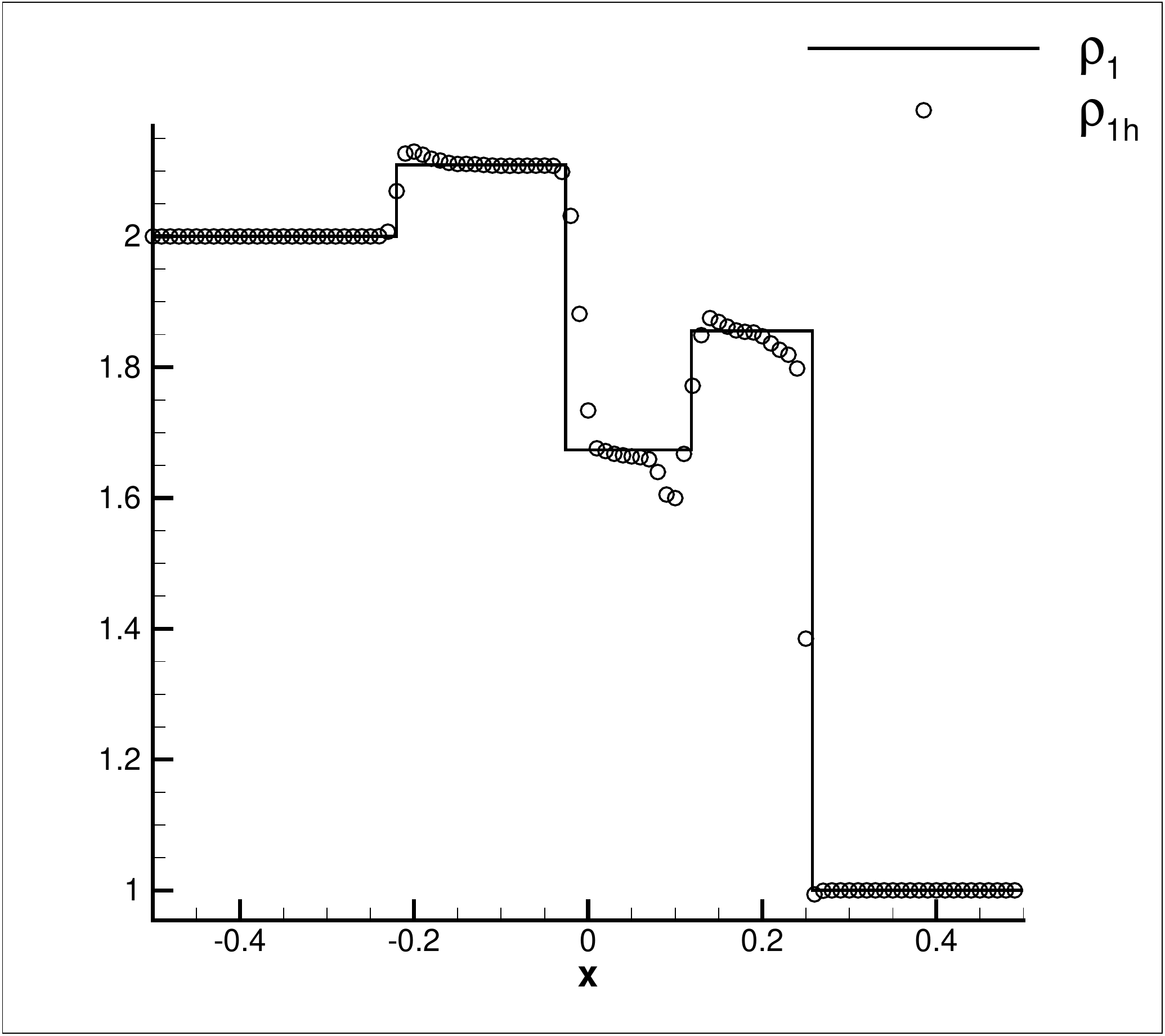}}
  \subfloat[$u_{1}$]{
  \includegraphics[height=.20\paperwidth,trim=0.2cm 0.2cm 0.2cm 0.2cm,clip=true]{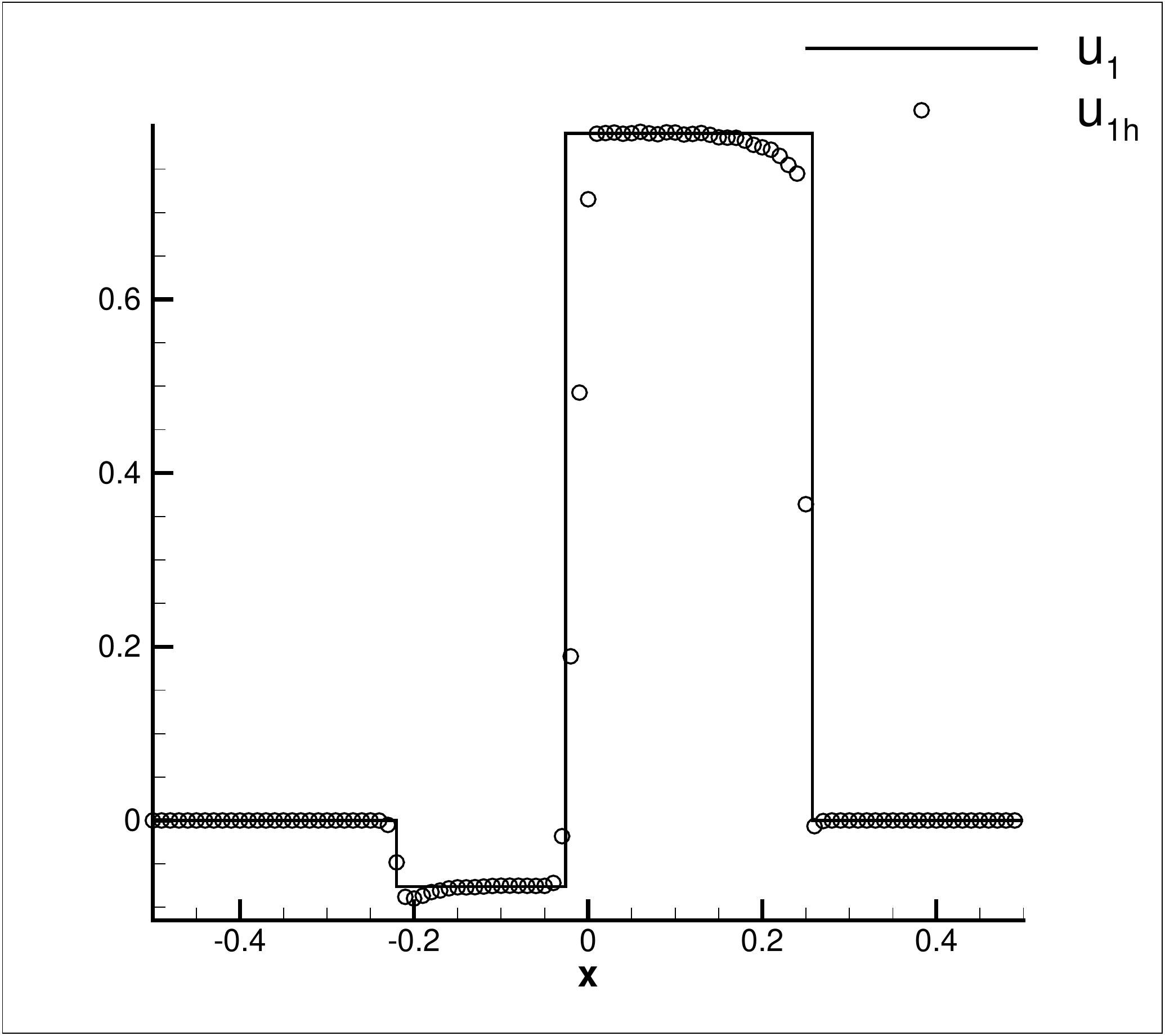}}
  \subfloat[$p_{1}$]{
  \includegraphics[height=.20\paperwidth,trim=0.2cm 0.2cm 0.2cm 0.2cm,clip=true]{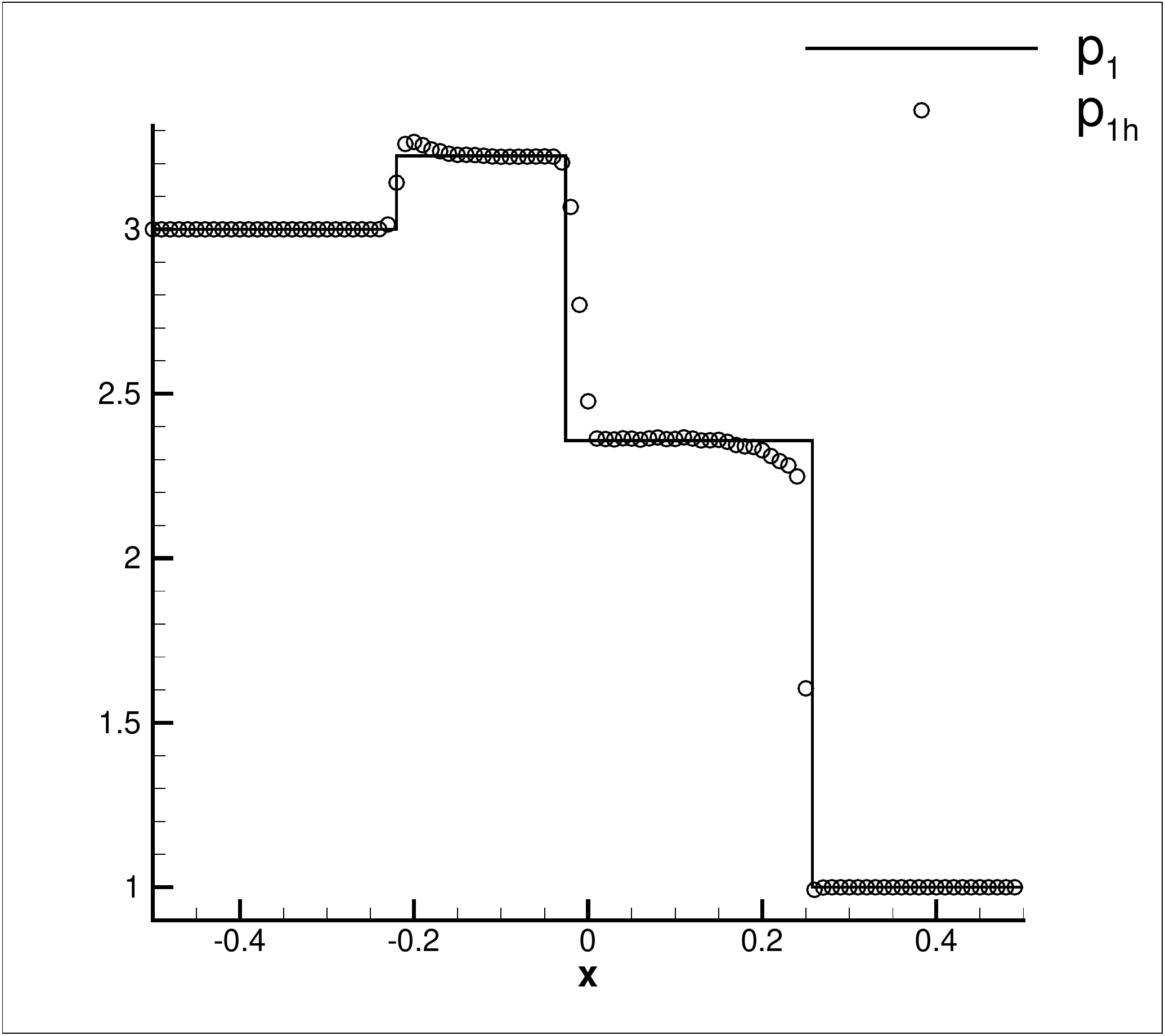}} \\
 \subfloat[$\rho_{2}$]{
  \includegraphics[height=.20\paperwidth,trim=0.2cm 0.2cm 0.2cm 0.2cm,clip=true]{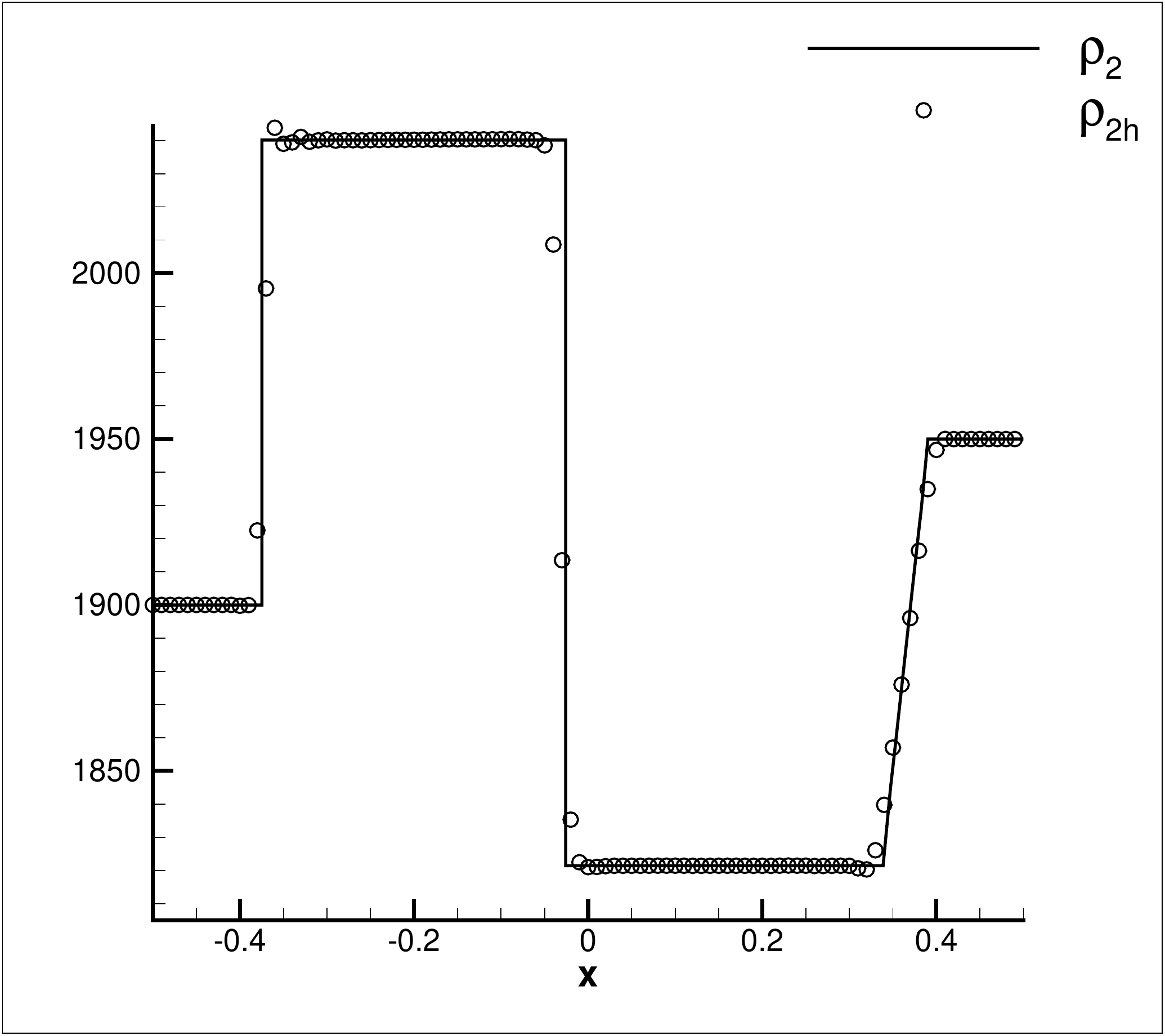}}
  \subfloat[$u_{2}$]{
  \includegraphics[height=.20\paperwidth,trim=0.2cm 0.2cm 0.2cm 0.2cm,clip=true]{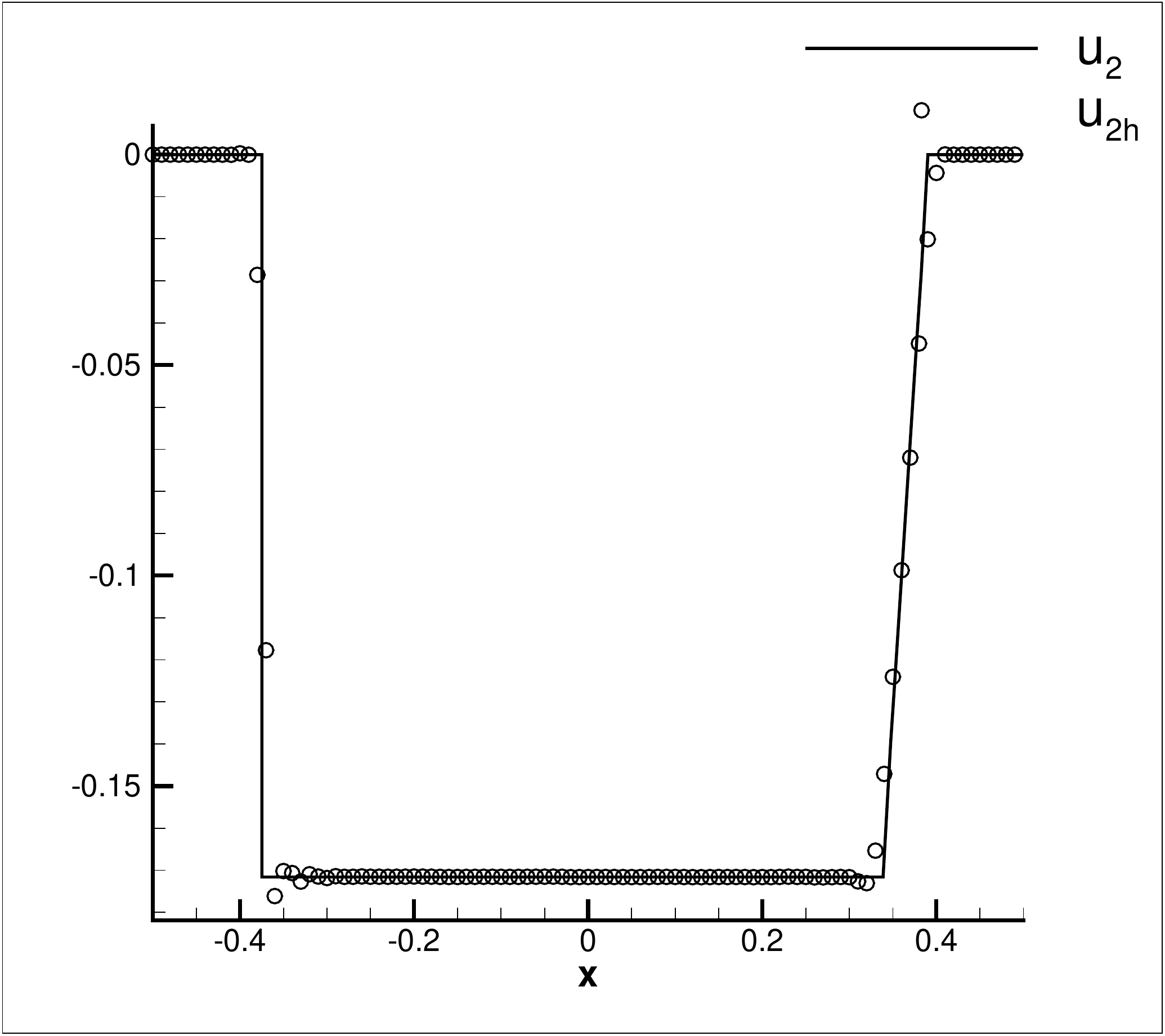}}
  \subfloat[$p_{2}$]{
  \includegraphics[height=.20\paperwidth,trim=0.2cm 0.2cm 0.2cm 0.2cm,clip=true]{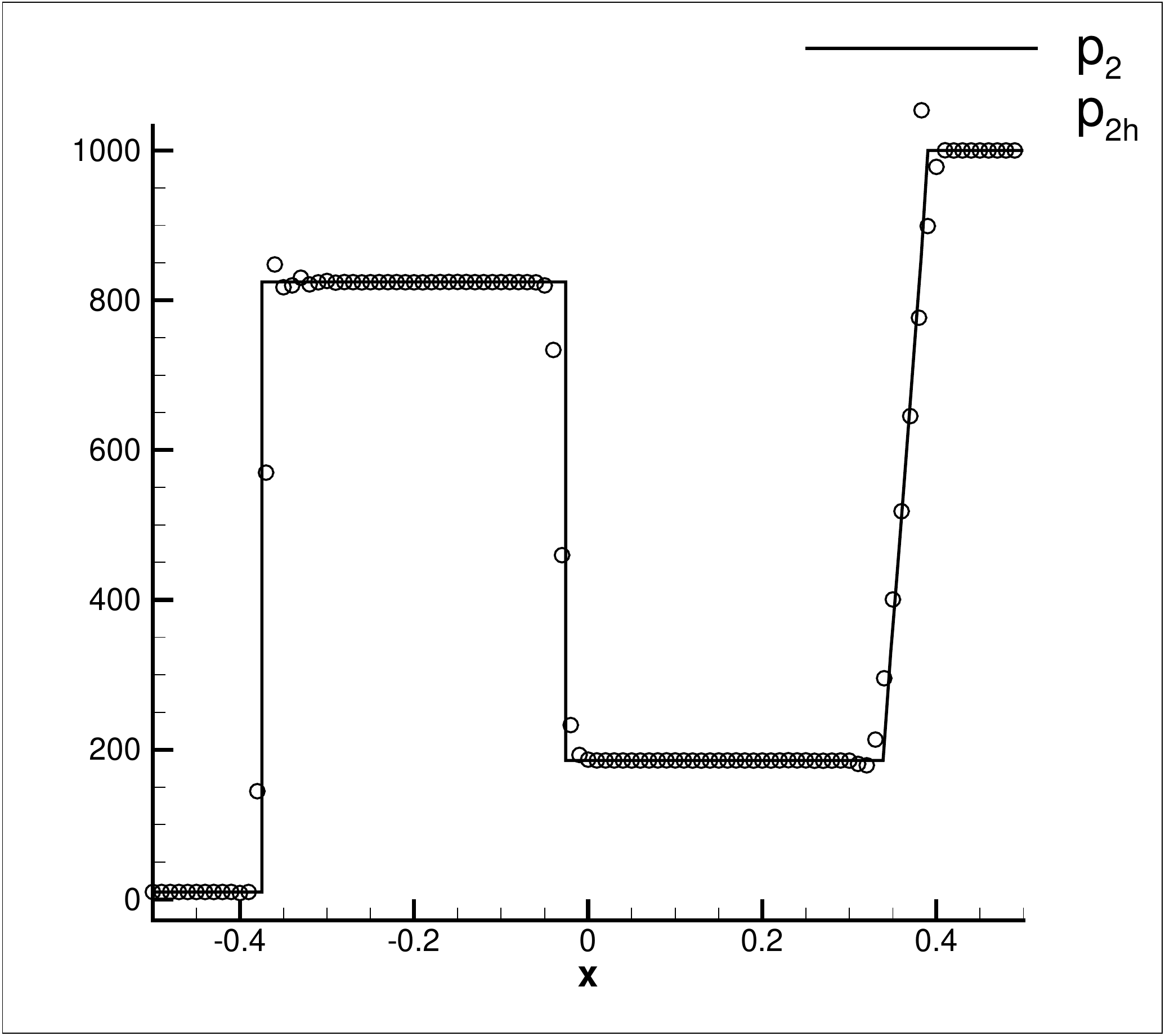}}
 \normalsize\caption{Comparison of the fourth order accurate numerical solution to the exact solution for test case RP2.}  
 \label{result: RP2}
\end{figure}

\begin{figure}[H]
 \center
 \subfloat[$\alpha_{1}$]{
  \includegraphics[height=.20\paperwidth,trim=0.2cm 0.2cm 0.2cm 0.2cm,clip=true]{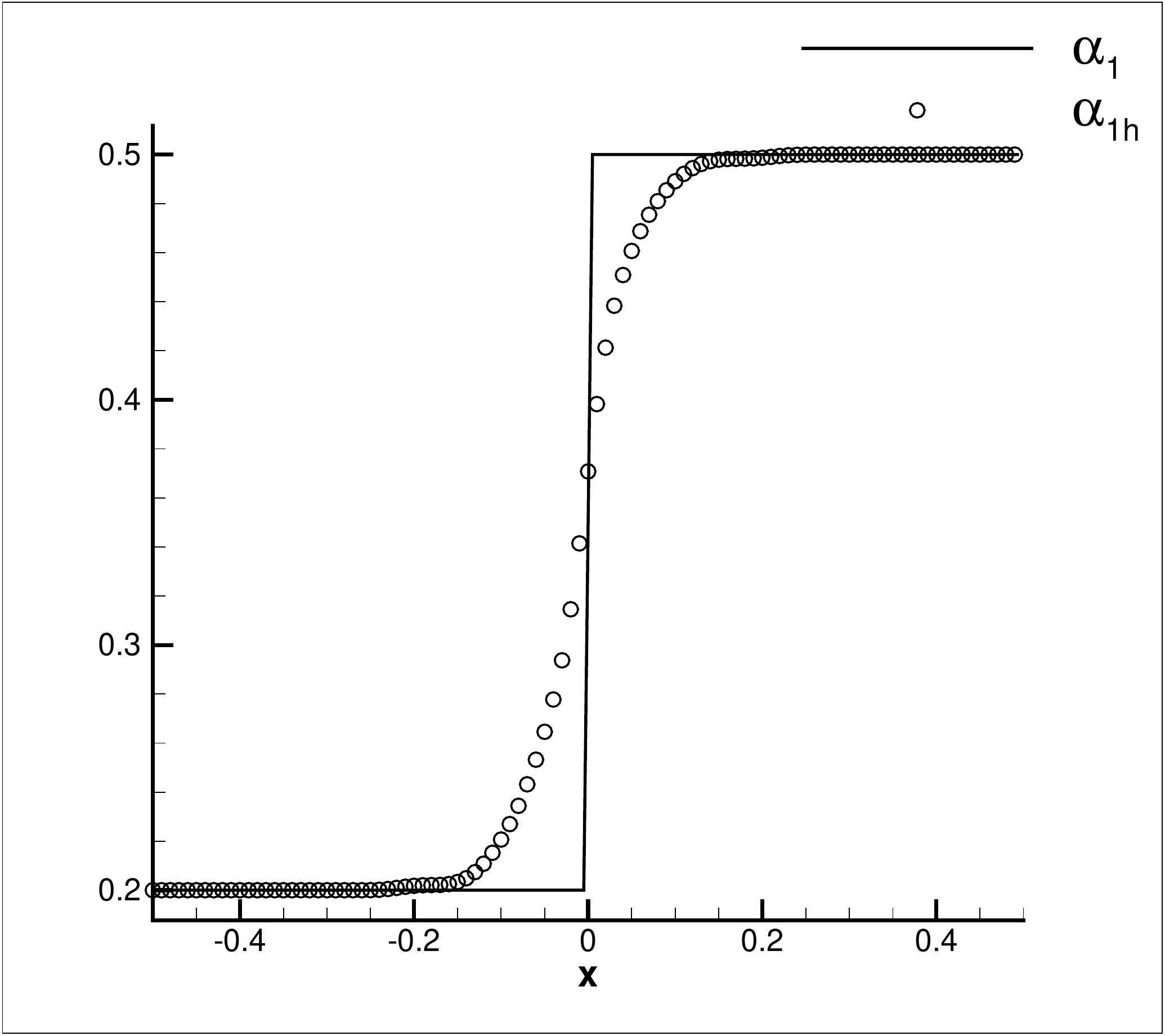}} \\
 \subfloat[$\rho_{1}$]{
  \includegraphics[height=.20\paperwidth,trim=0.2cm 0.2cm 0.2cm 0.2cm,clip=true]{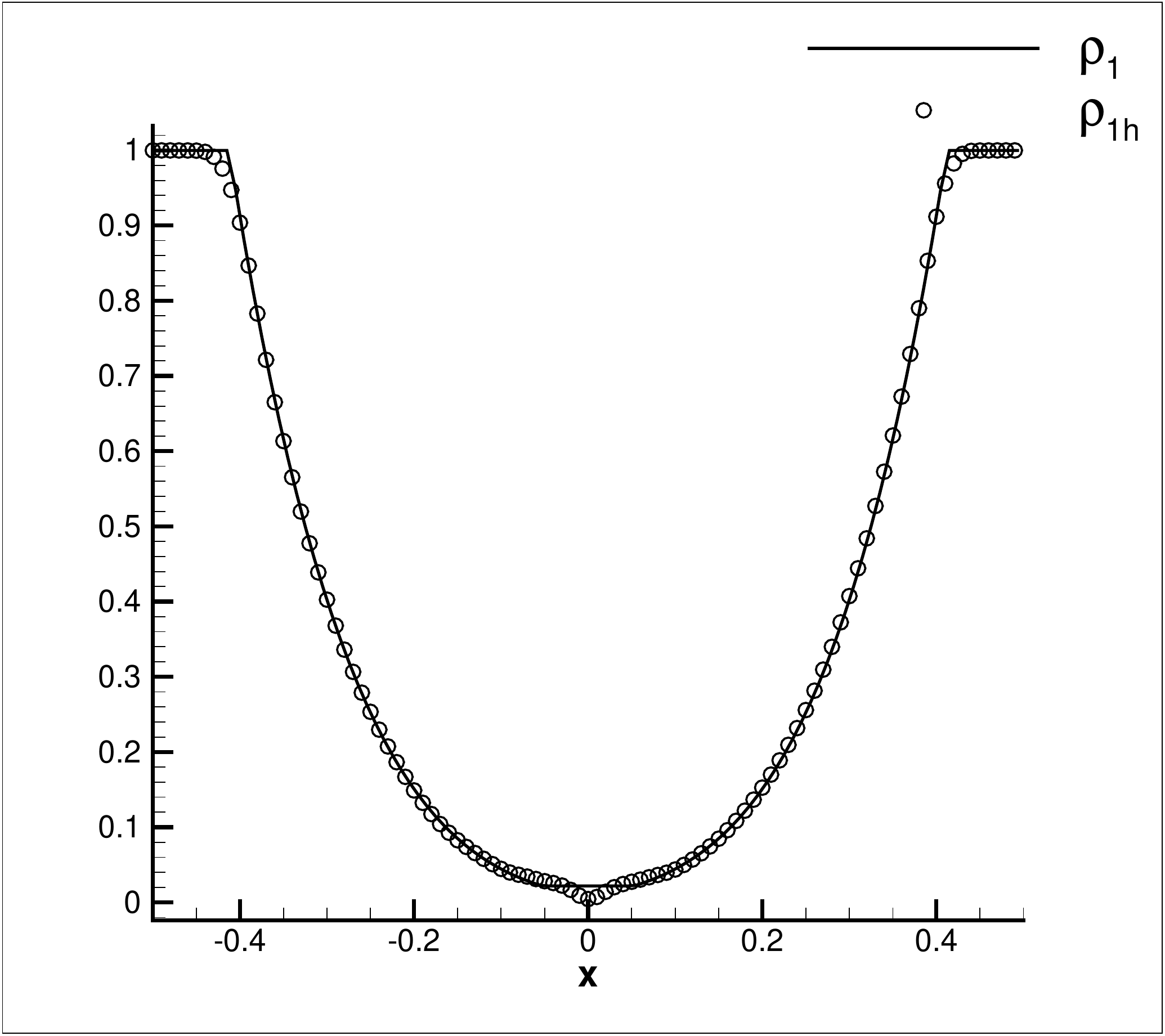}}
  \subfloat[$u_{1}$]{
  \includegraphics[height=.20\paperwidth,trim=0.2cm 0.2cm 0.2cm 0.2cm,clip=true]{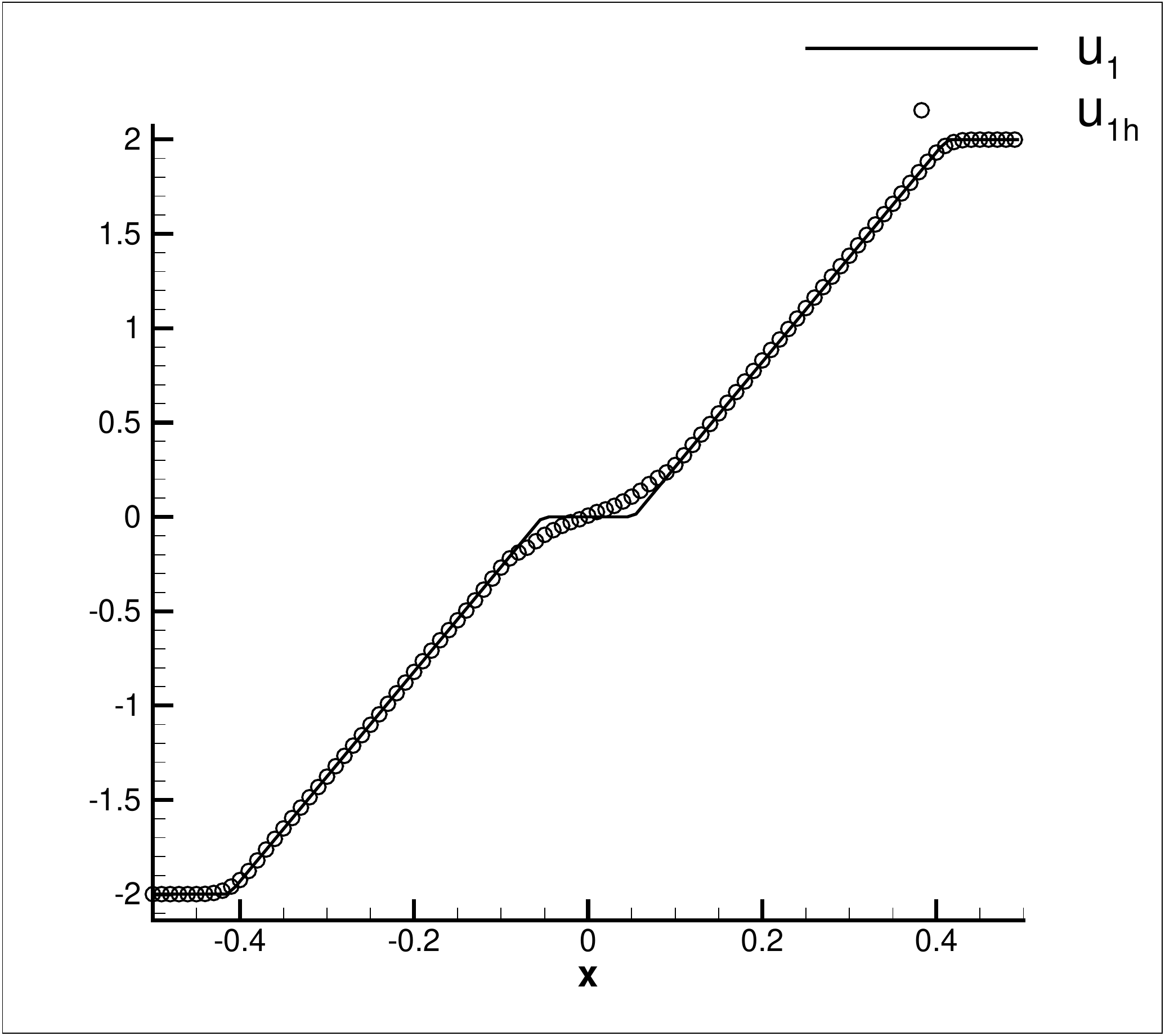}}
  \subfloat[$p_{1}$]{
  \includegraphics[height=.20\paperwidth,trim=0.2cm 0.2cm 0.2cm 0.2cm,clip=true]{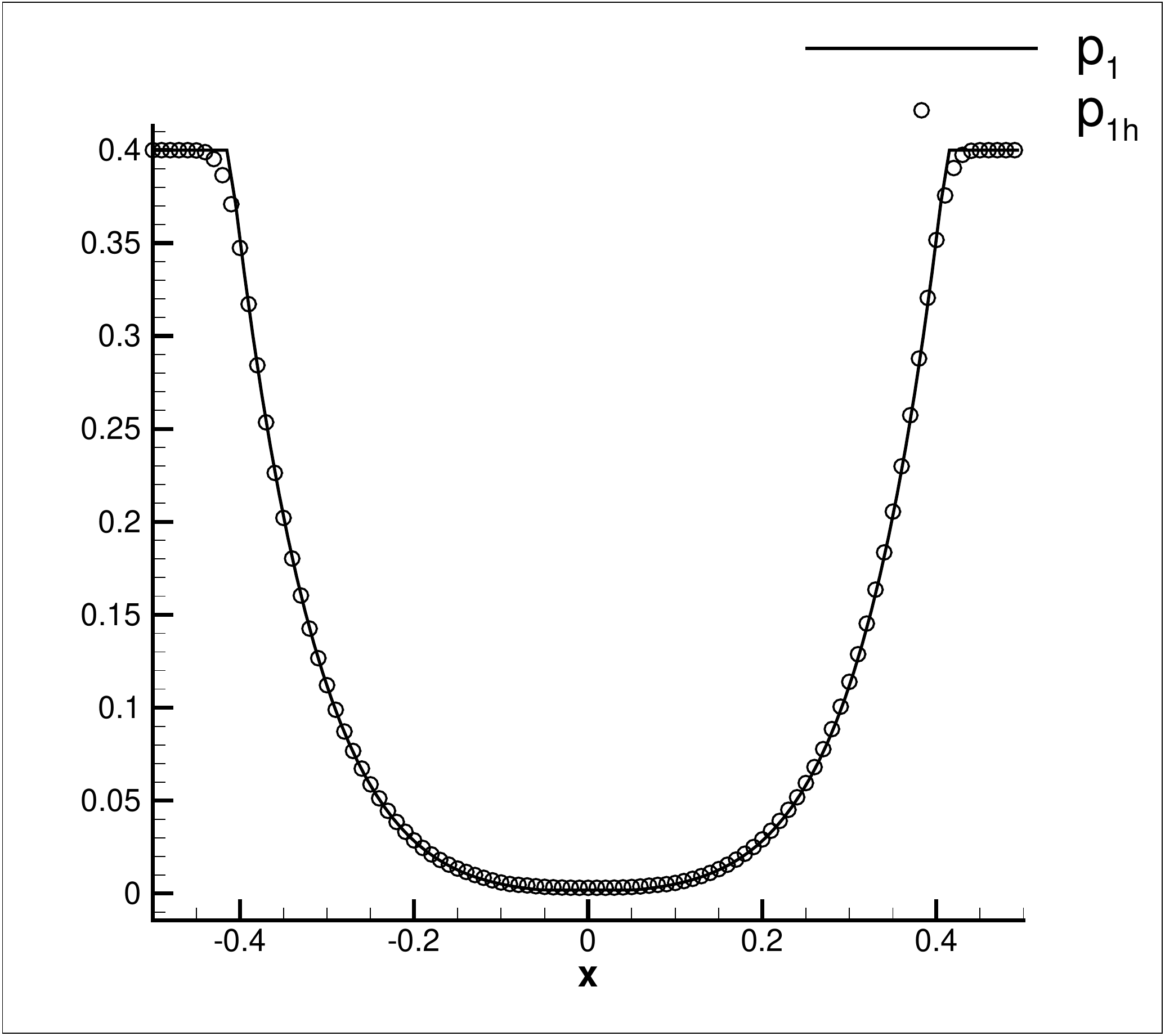}} \\
 \subfloat[$\rho_{2}$]{
  \includegraphics[height=.20\paperwidth,trim=0.2cm 0.2cm 0.2cm 0.2cm,clip=true]{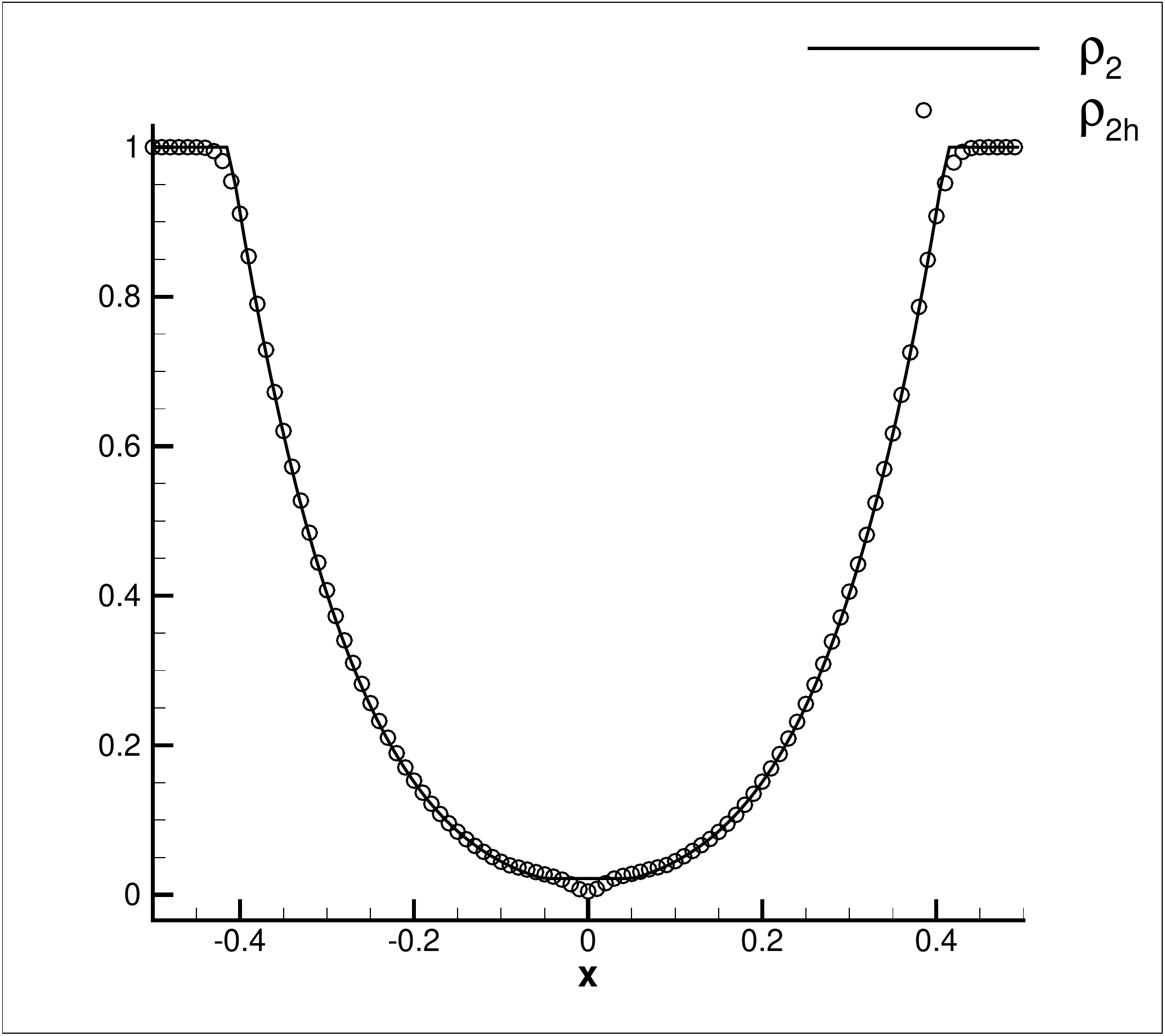}}
  \subfloat[$u_{2}$]{
  \includegraphics[height=.20\paperwidth,trim=0.2cm 0.2cm 0.2cm 0.2cm,clip=true]{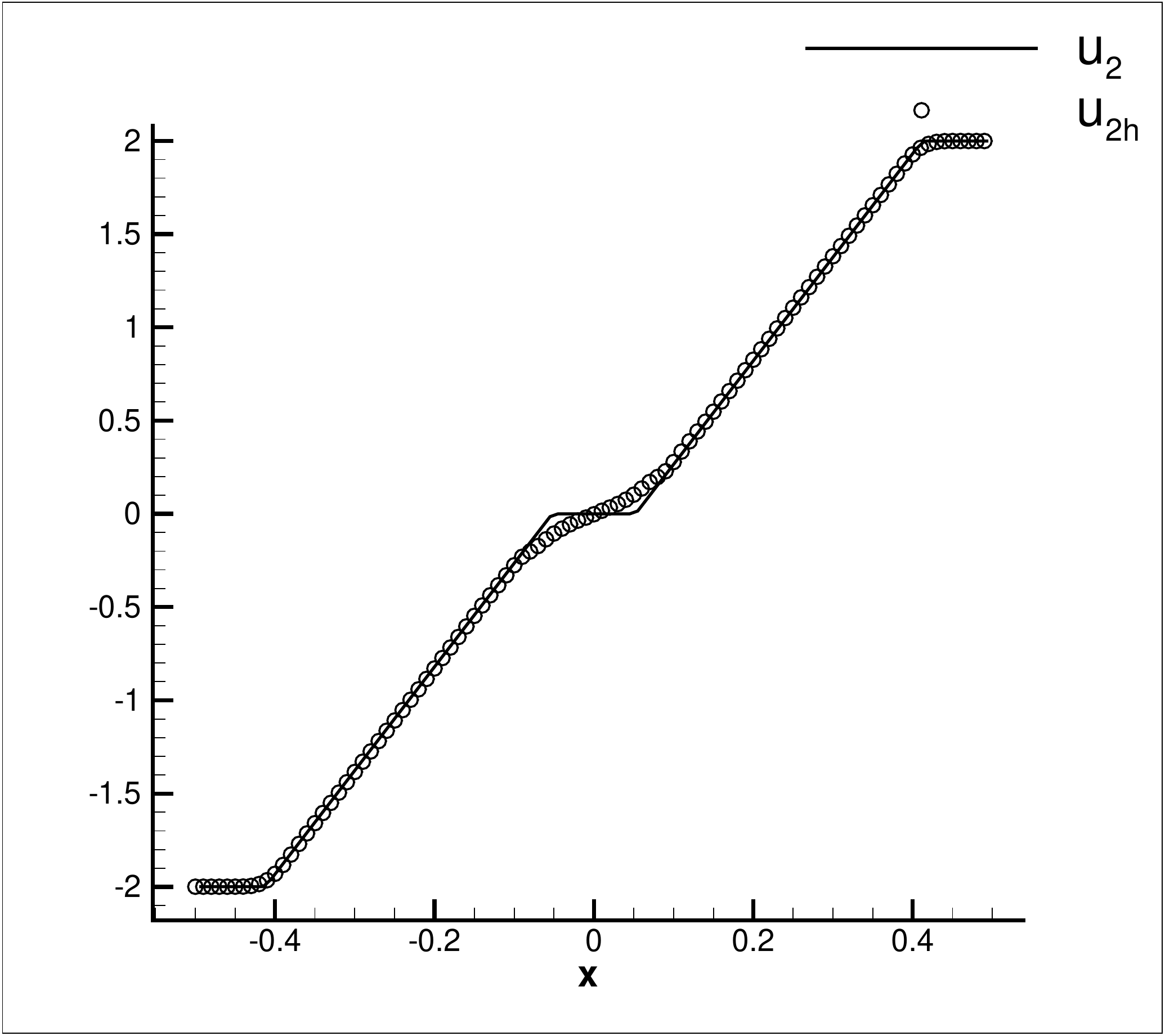}}
  \subfloat[$p_{2}$]{
  \includegraphics[height=.20\paperwidth,trim=0.2cm 0.2cm 0.2cm 0.2cm,clip=true]{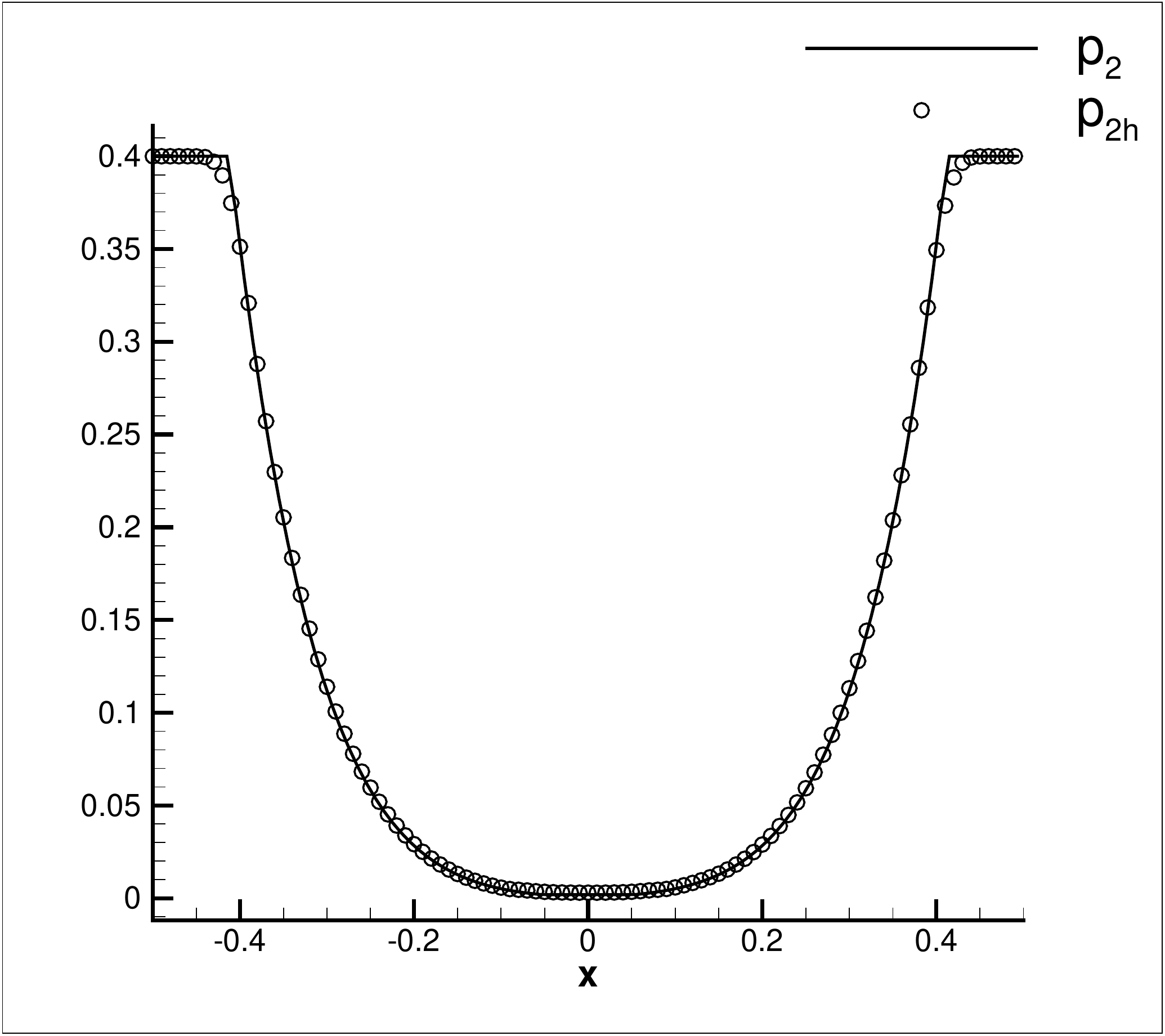}}
 \normalsize\caption{Comparison of the fourth order accurate numerical solution to the exact solution for test case RP3.}  
 \label{result: RP3}
\end{figure}

\begin{figure}[H]
 \center
 \subfloat[$\alpha_{1}$]{
  \includegraphics[height=.20\paperwidth,trim=0.2cm 0.2cm 0.2cm 0.2cm,clip=true]{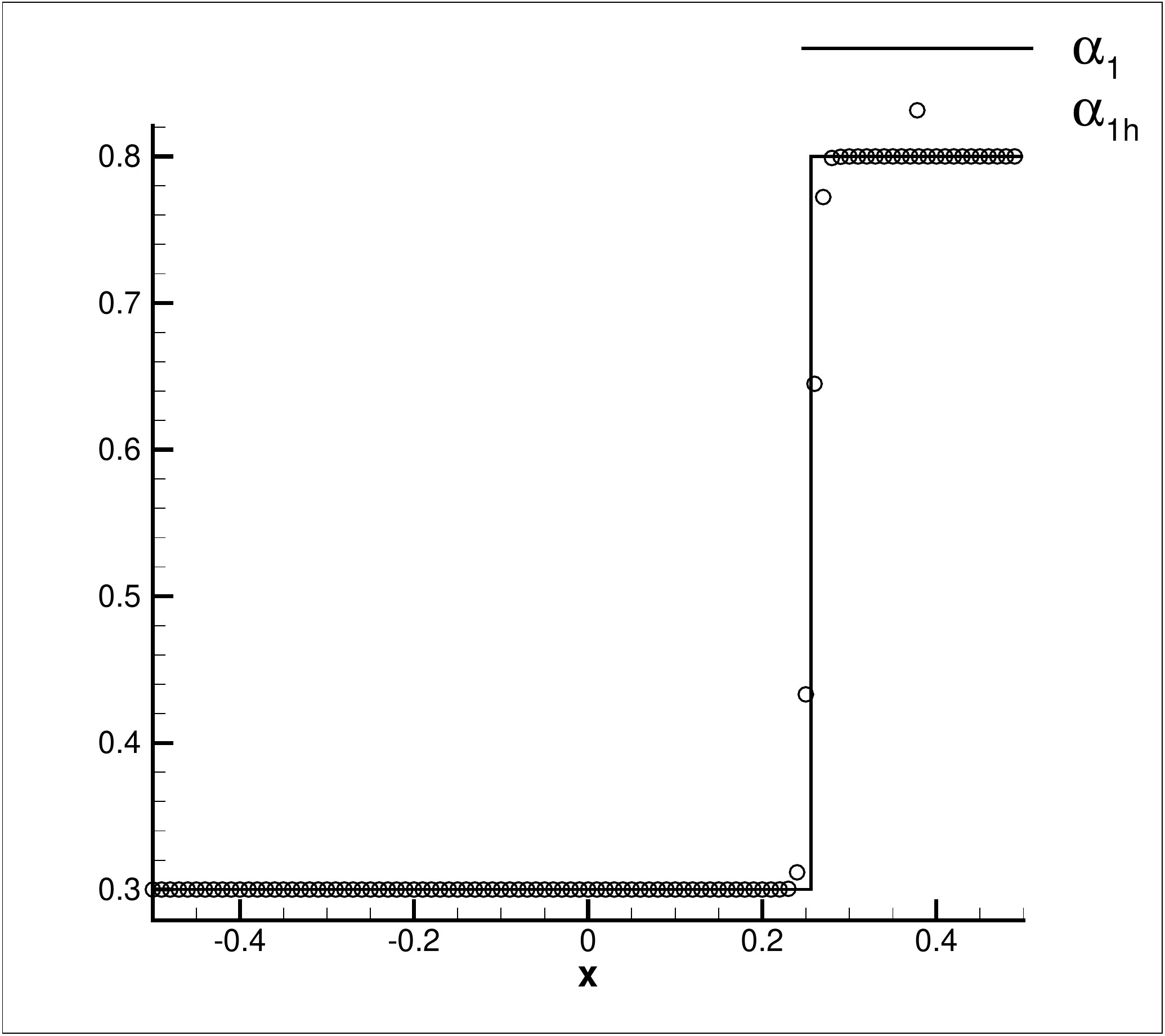}} \\
 \subfloat[$\rho_{1}$]{
  \includegraphics[height=.20\paperwidth,trim=0.2cm 0.2cm 0.2cm 0.2cm,clip=true]{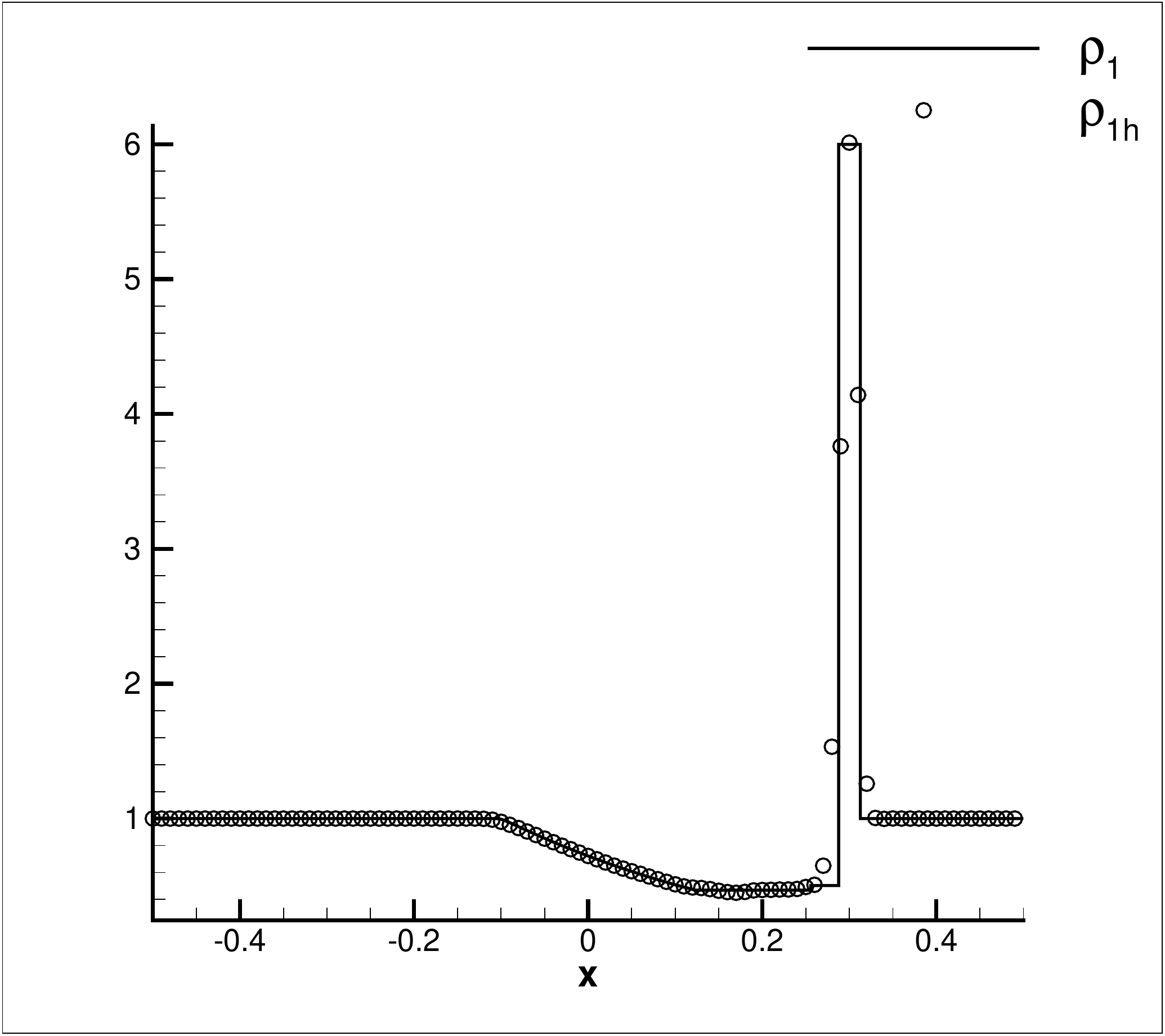}}
  \subfloat[$u_{1}$]{
  \includegraphics[height=.20\paperwidth,trim=0.2cm 0.2cm 0.2cm 0.2cm,clip=true]{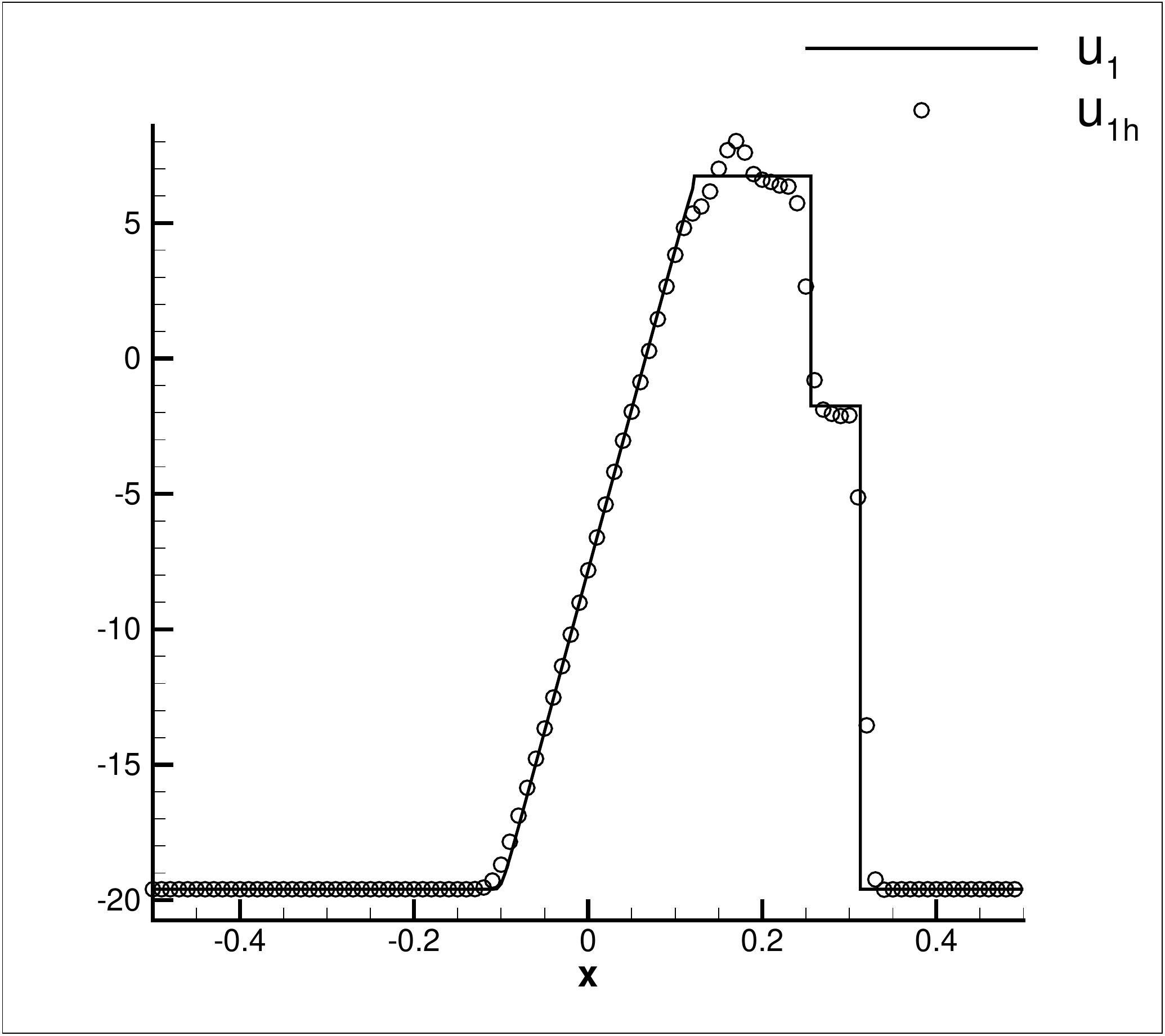}}
  \subfloat[$p_{1}$]{
  \includegraphics[height=.20\paperwidth,trim=0.2cm 0.2cm 0.2cm 0.2cm,clip=true]{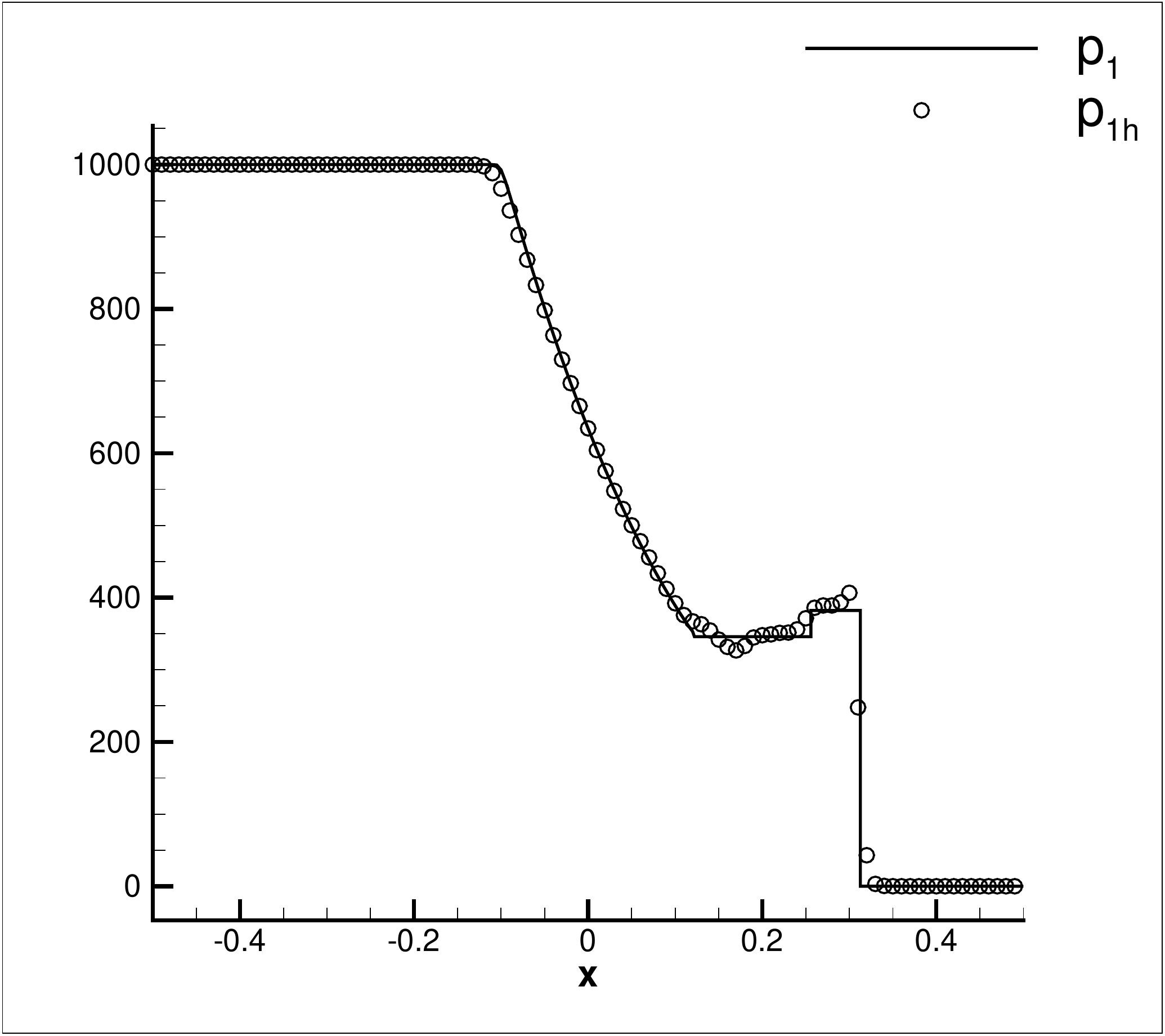}} \\
 \subfloat[$\rho_{2}$]{
  \includegraphics[height=.20\paperwidth,trim=0.2cm 0.2cm 0.2cm 0.2cm,clip=true]{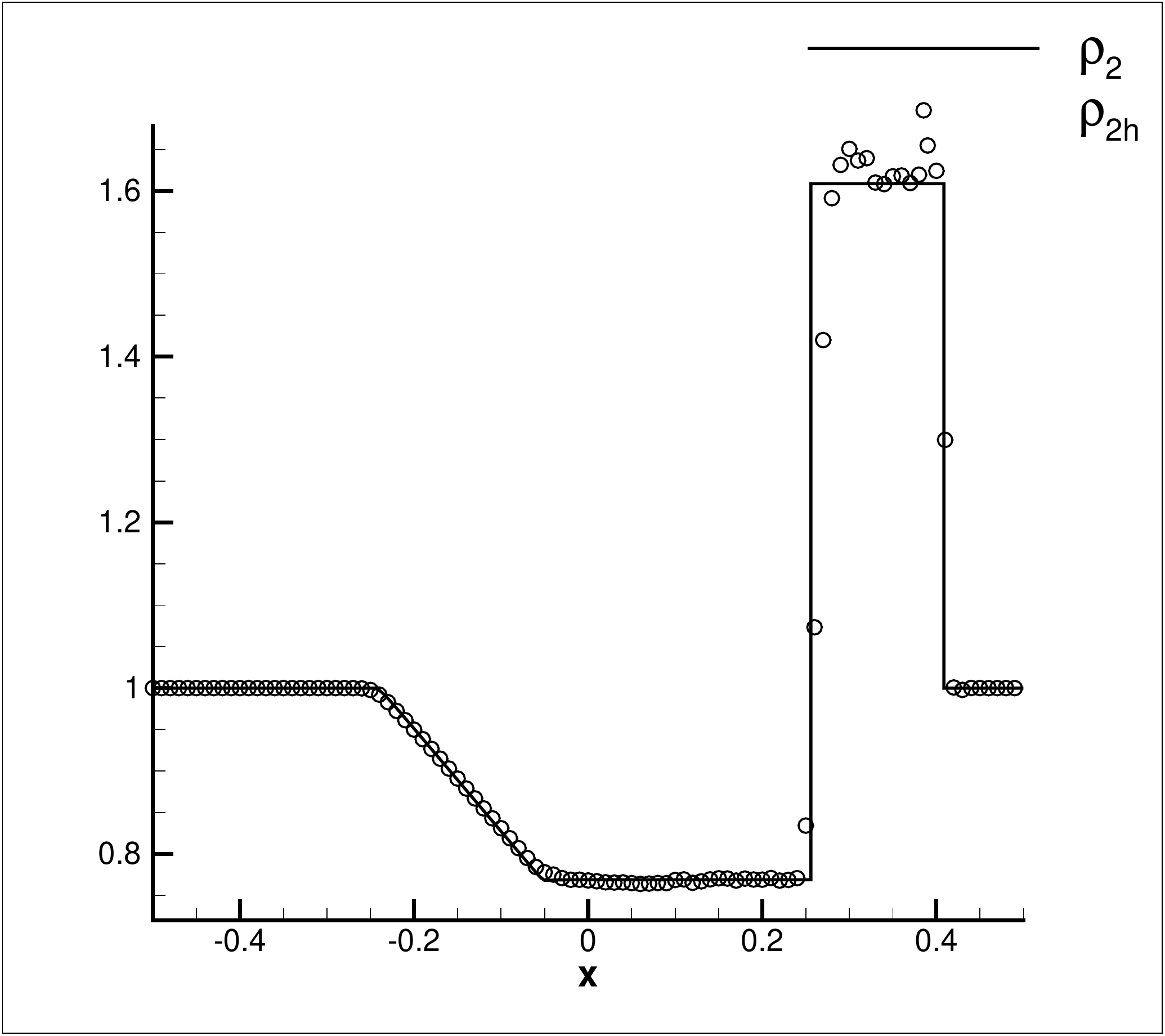}}
  \subfloat[$u_{2}$]{
  \includegraphics[height=.20\paperwidth,trim=0.2cm 0.2cm 0.2cm 0.2cm,clip=true]{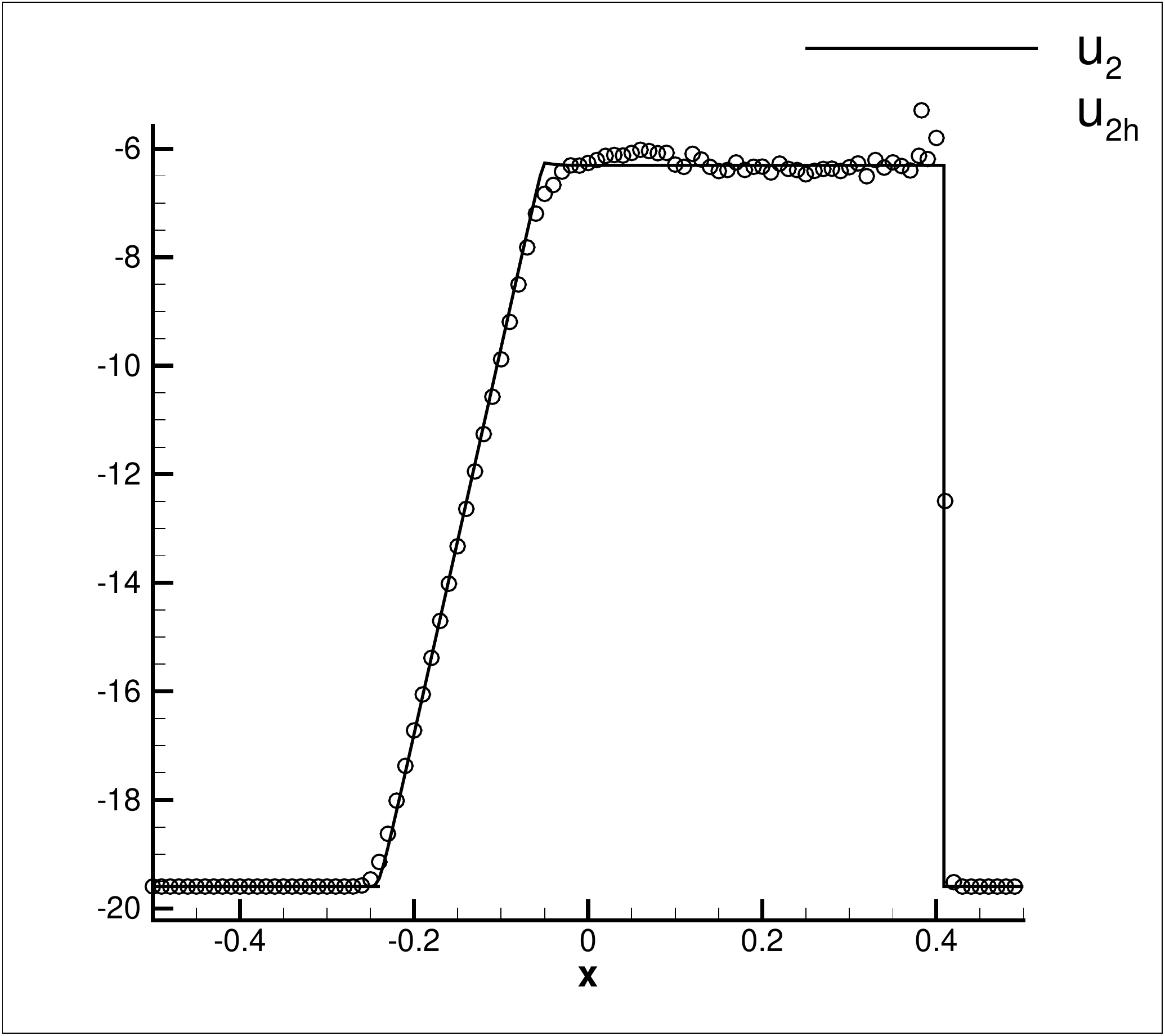}}
  \subfloat[$p_{2}$]{
  \includegraphics[height=.20\paperwidth,trim=0.2cm 0.2cm 0.2cm 0.2cm,clip=true]{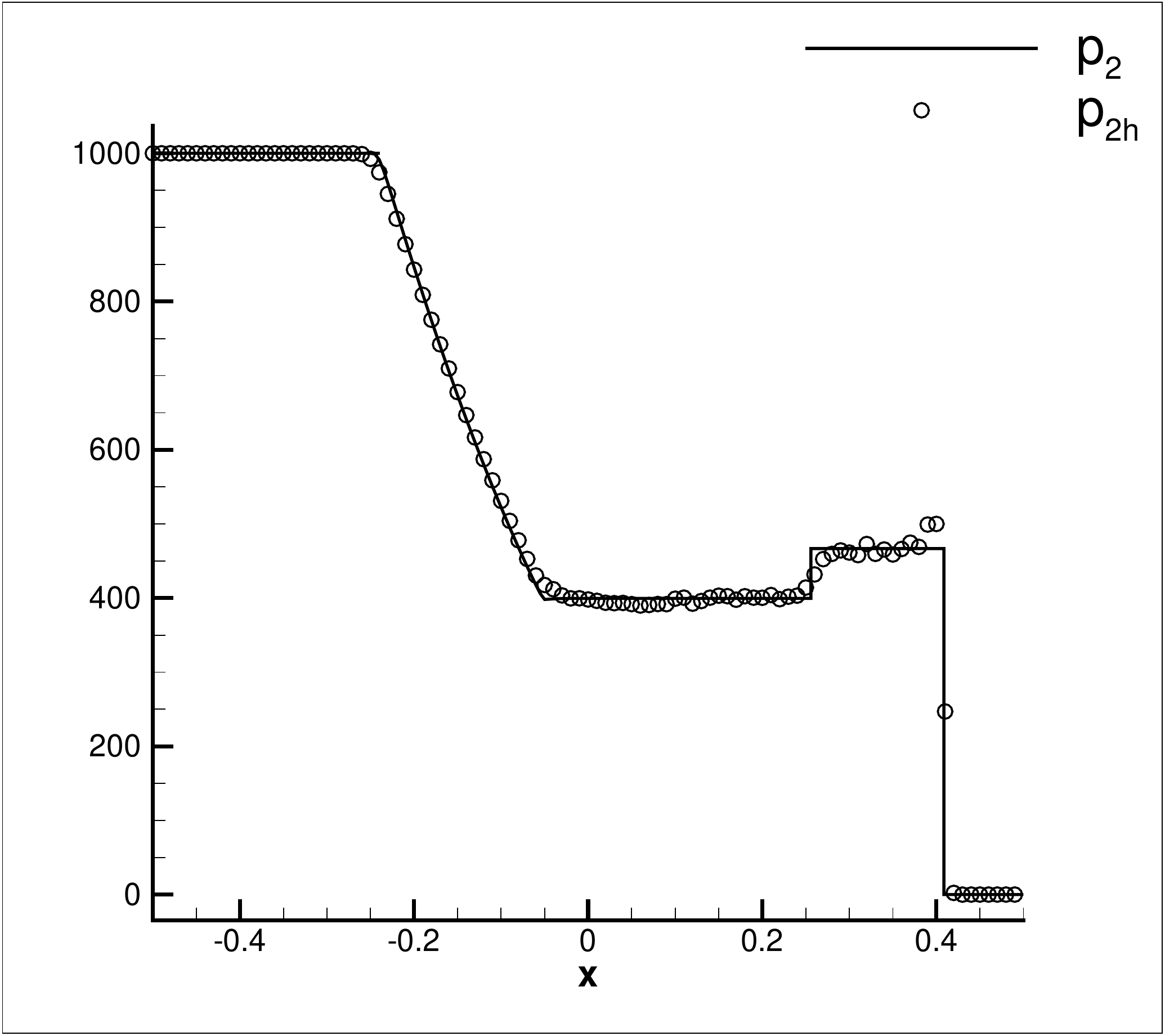}}
 \normalsize\caption{Comparison of the fourth order accurate numerical solution to the exact solution for test case RP4.}  
 \label{result: RP4}
\end{figure}

\begin{figure}[H]
 \center
 \subfloat[$\alpha_{1}$]{
  \includegraphics[height=.20\paperwidth,trim=0.2cm 0.2cm 0.2cm 0.2cm,clip=true]{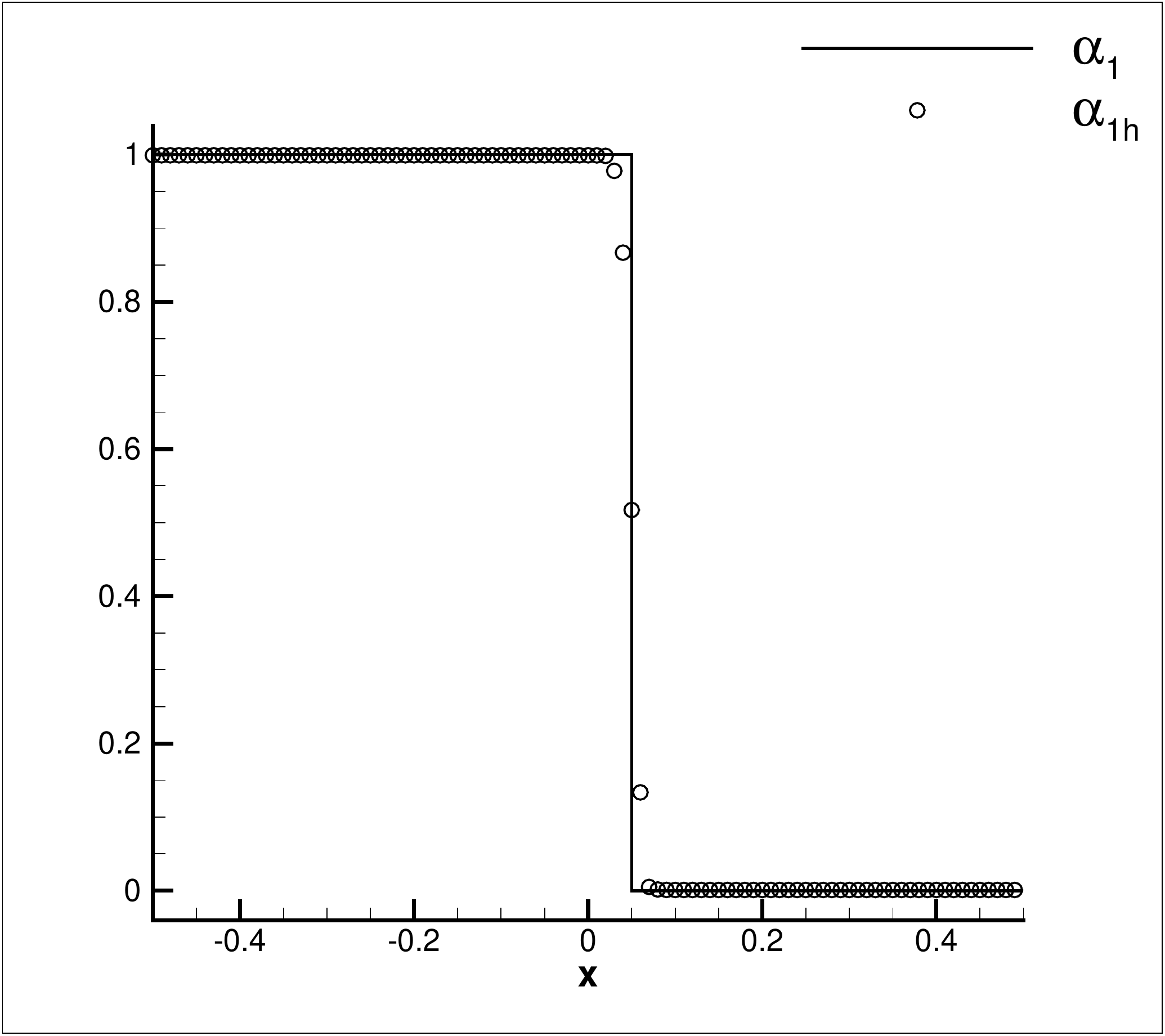}} \\
 \subfloat[$\rho_{1}$]{
  \includegraphics[height=.20\paperwidth,trim=0.2cm 0.2cm 0.2cm 0.2cm,clip=true]{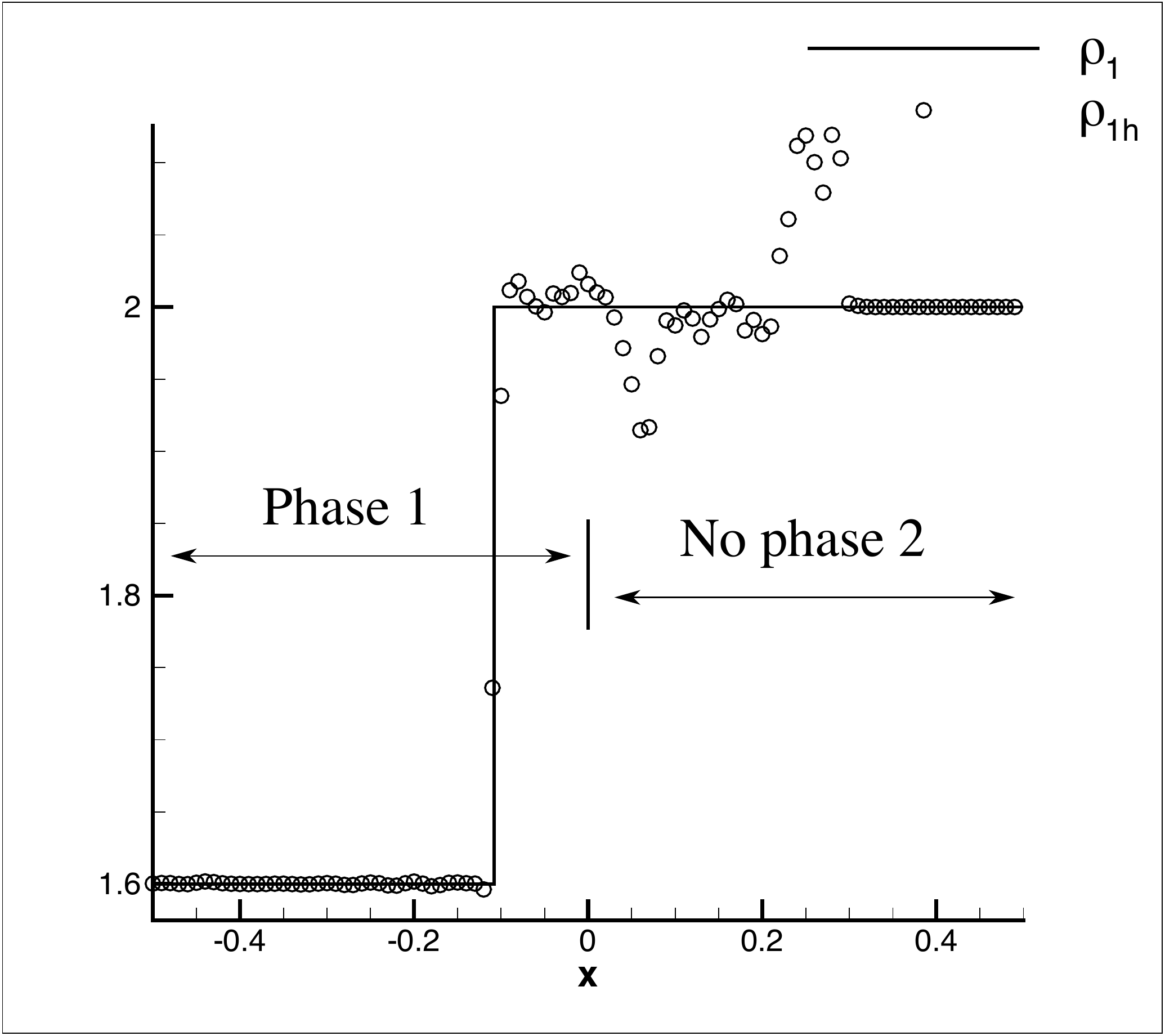}}
  \subfloat[$u_{1}$]{
  \includegraphics[height=.20\paperwidth,trim=0.2cm 0.2cm 0.2cm 0.2cm,clip=true]{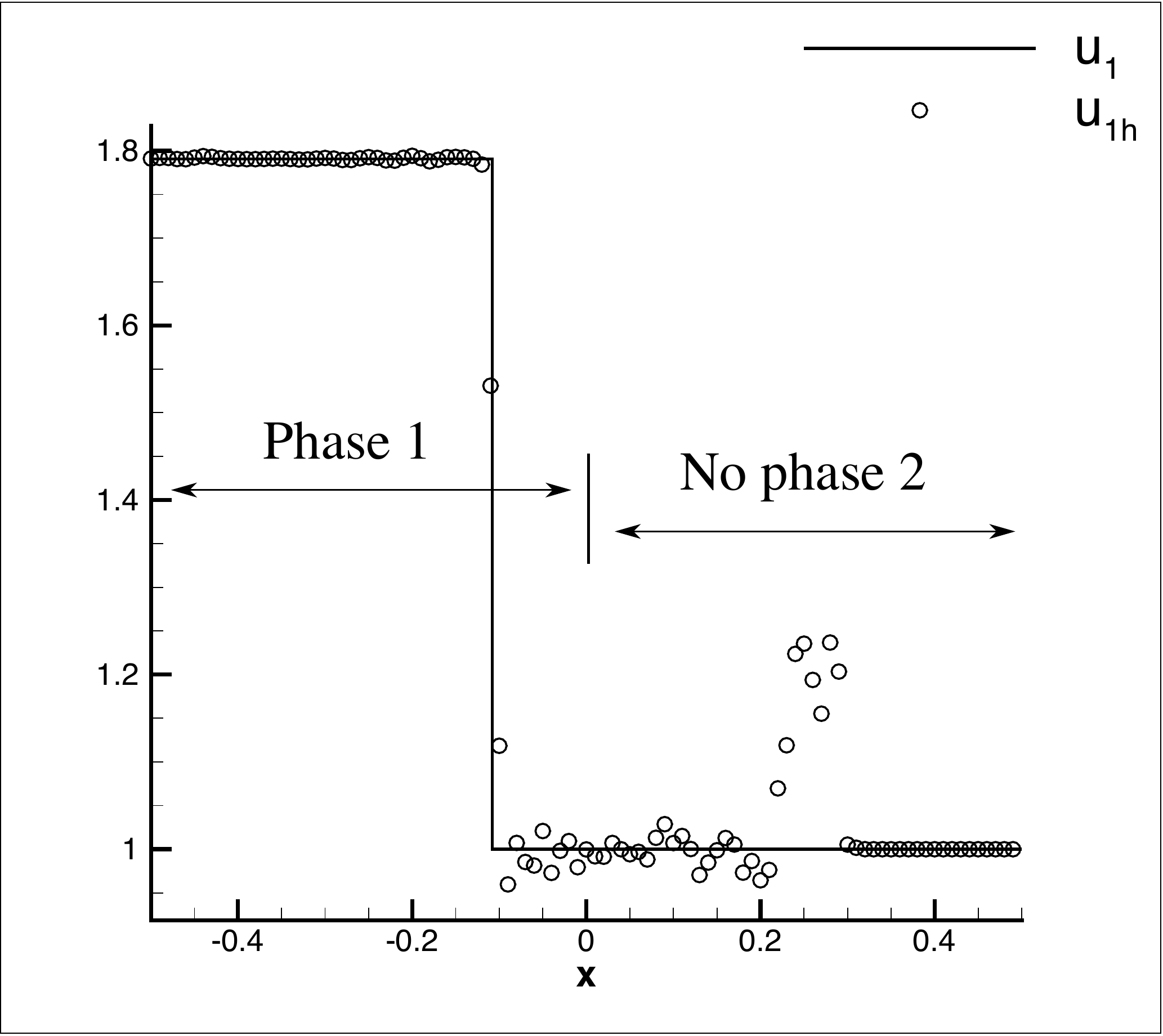}}
  \subfloat[$p_{1}$]{ 
  \includegraphics[height=.20\paperwidth,trim=0.2cm 0.2cm 0.2cm 0.2cm,clip=true]{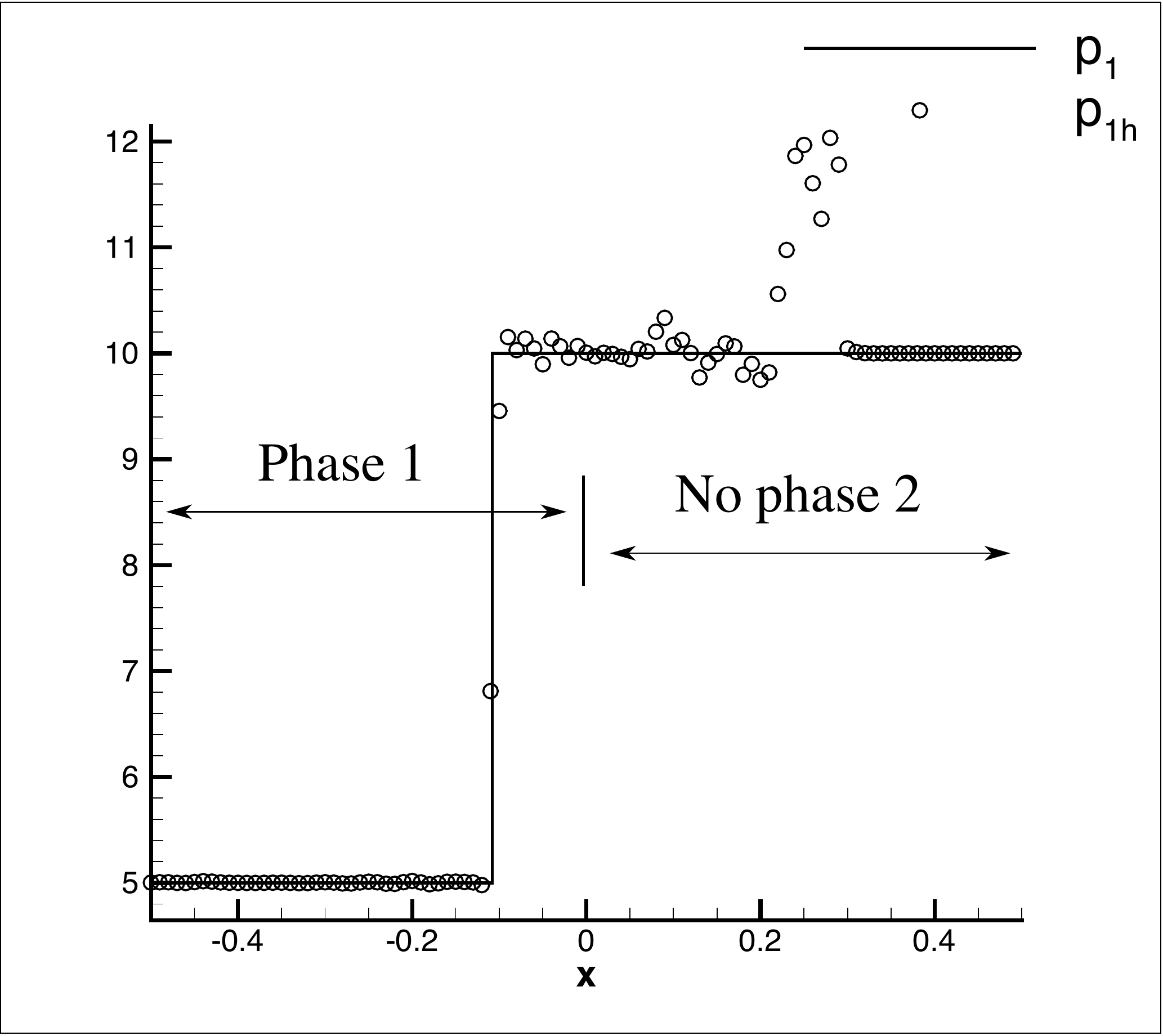}} \\
 \subfloat[$\rho_{2}$]{
  \includegraphics[height=.20\paperwidth,trim=0.2cm 0.2cm 0.2cm 0.2cm,clip=true]{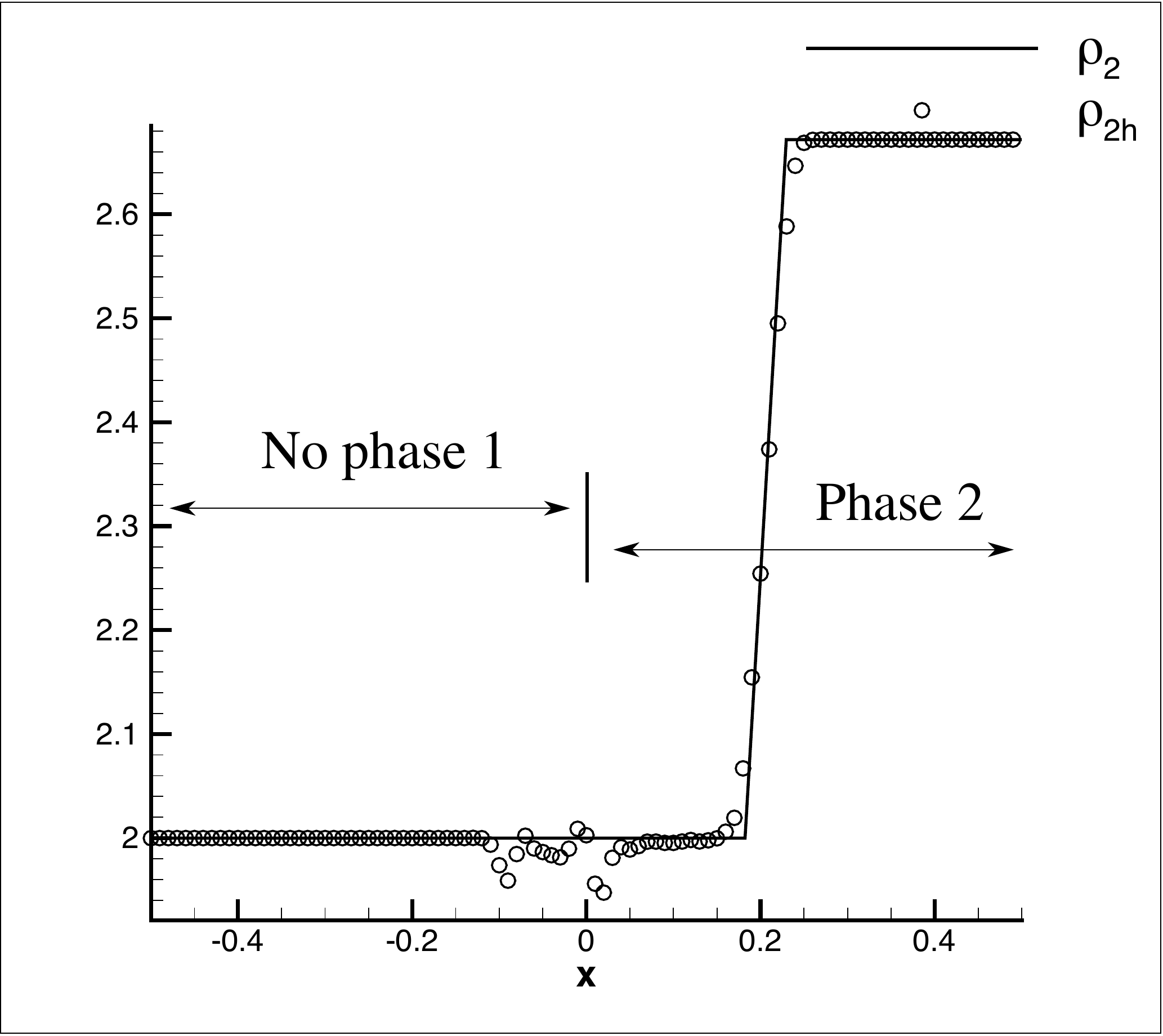}}
  \subfloat[$u_{2}$]{
  \includegraphics[height=.20\paperwidth,trim=0.2cm 0.2cm 0.2cm 0.2cm,clip=true]{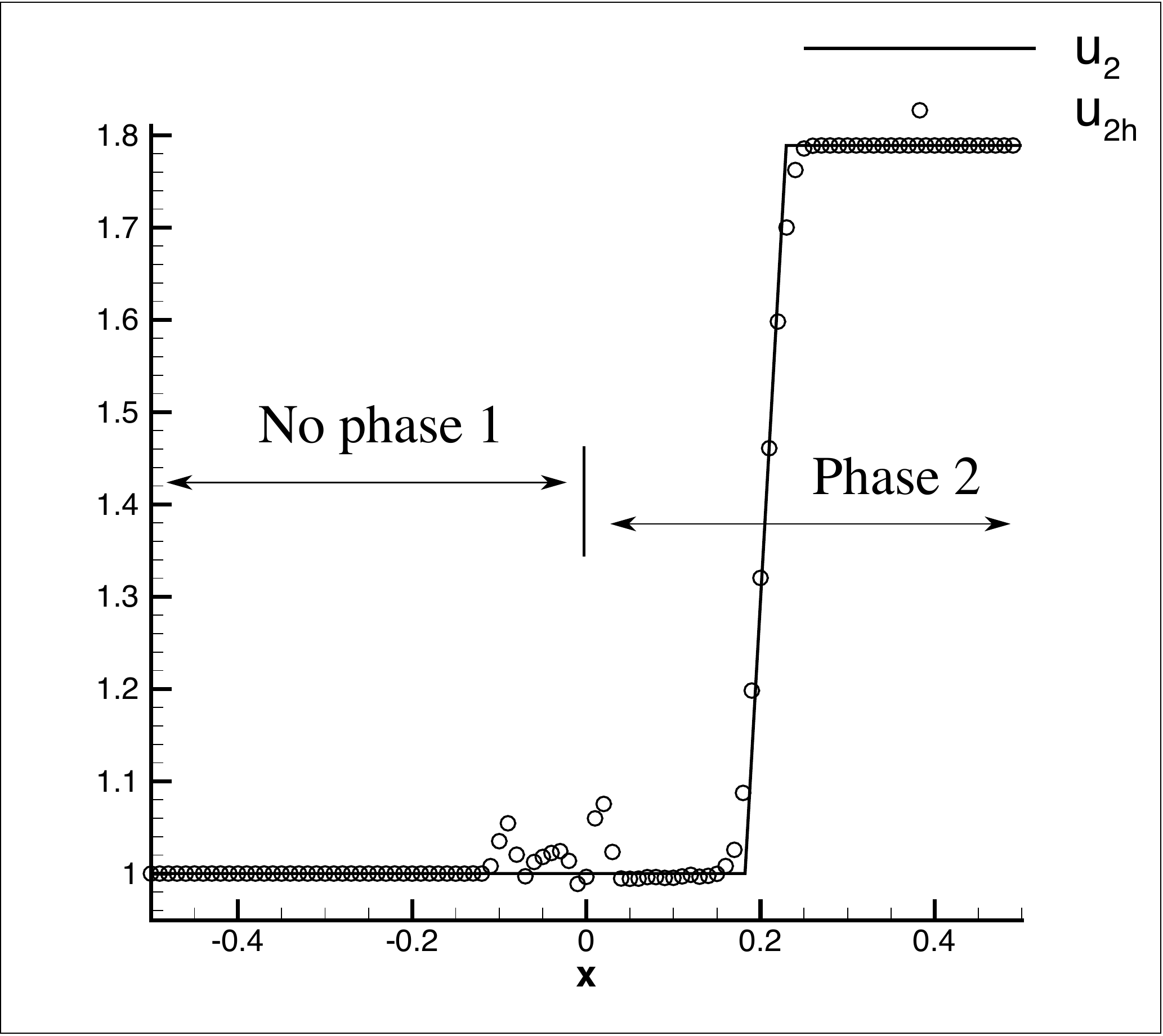}}
  \subfloat[$p_{2}$]{
  \includegraphics[height=.20\paperwidth,trim=0.2cm 0.2cm 0.2cm 0.2cm,clip=true]{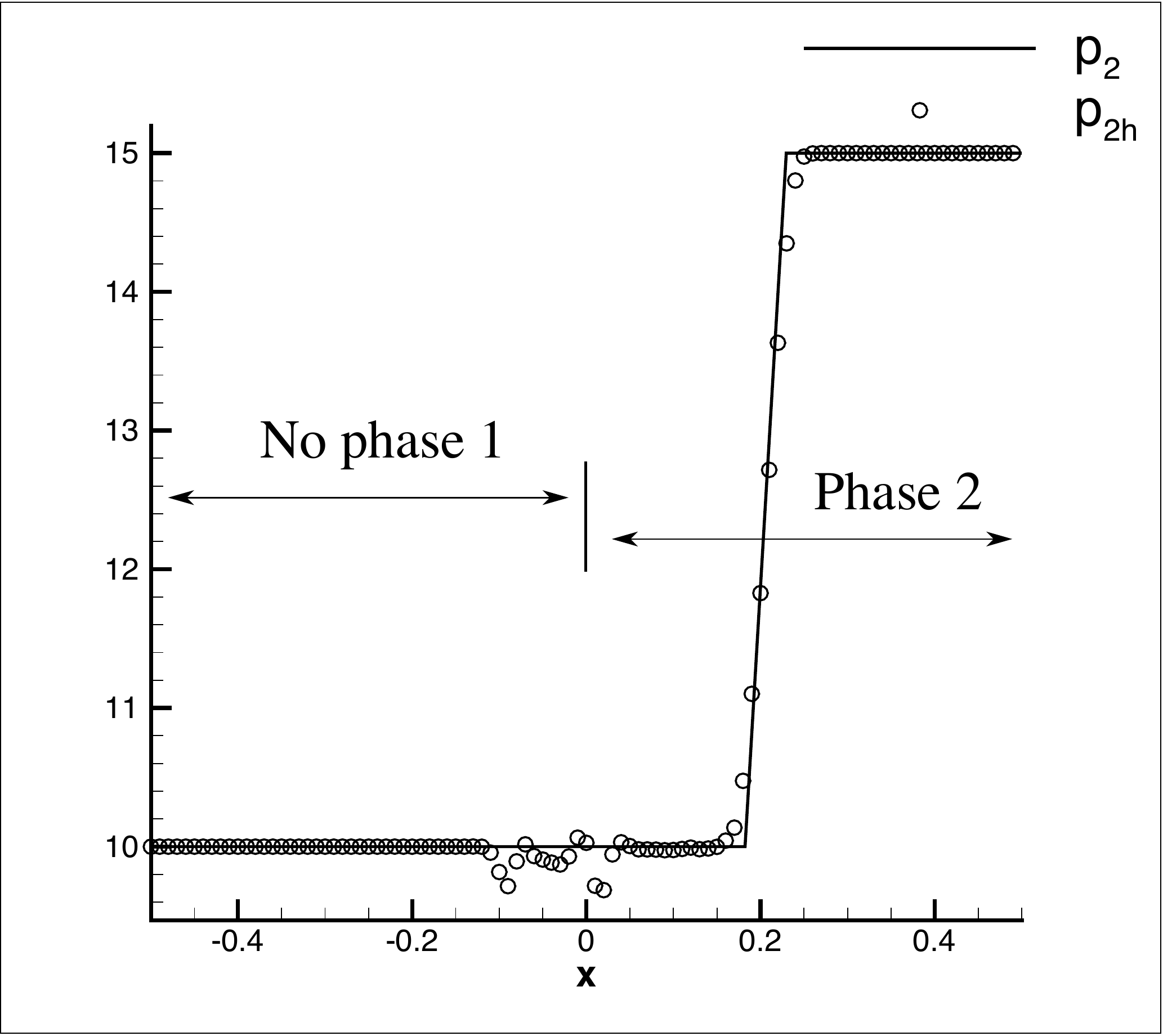}}
 \normalsize\caption{Comparison of the fourth order accurate numerical solution to the exact solution for test case RP5.}  
 \label{result: RP5}
\end{figure}

%
%
\section{Numerical tests in multiple space dimensions}\label{sec: extension to multi-space dim}
The Baer-Nunziato model in multiple space dimensions reads
\begin{linenomath*}
\begin{equation} \label{Eqn: 2D-BNM}
    \partial_t\vecu + \nabla \cdot\textbf{f}(\vecu) +\textbf{c}(\vecu)\nabla\vecu = 0, \quad \textbf{x}\in\mathbb{R}^d, \, t\geqslant 0,
\end{equation}
\end{linenomath*}
where
\begin{linenomath*}
\begin{equation*}
    \vecu := 
    \begin{pmatrix}
    \alpha_i\\
    \alpha_i\rho_i\\
    \alpha_i\rho_i\textbf{v}_i\\
    \alpha_i\rho_iE_i
    \end{pmatrix}, \quad
    \textbf{f}(\vecu) := 
    \begin{pmatrix}
    0\\
    \alpha_i\rho_i\textbf{v}_i^\top\\
    \alpha_i(\rho_i\textbf{v}_i\textbf{v}^\top_i+\mathrm{p}_i\textbf{I})\\
    \alpha_i(\rho_iE_i+\mathrm{p}_i)\textbf{v}_i^\top
    \end{pmatrix},\quad
    \textbf{c}(\vecu)\nabla\vecu :=
    \begin{pmatrix}
    \mathrm{\textbf{v}_I}^\top\\
    0\\
    -\pI\textbf{I}\\
    -\pI\mathrm{\textbf{v}_I}^\top
    \end{pmatrix}\nabla\alpha_i, \quad i = 1,2,
\end{equation*}
\end{linenomath*}
with $\textbf{v}_i=(u_i,v_i,w _i)^\top$ the velocity vector of the $i$th phase, $\mathrm{p}_i=\mathrm{p}_i(\rho_i,e_i)$ given by (\cref{Eqn: EOS}) and $e_i=E_i-\tfrac{1}{2}{\bf v}_i\cdot{\bf v}_i$ the specific internal energy.

The DGSEM scheme (\cref{eqn: modified DG semi-discrete}) can be extended to (\cref{Eqn: 2D-BNM}). The derivation of the scheme for Cartesian meshes is introduced in \ref{Appendix: DGSEM in 2D}, while the numerical fluxes for the above model are presented in \ref{Appendix: EC ES flux 2D}. Unless stated otherwise, the time step is computed with the CFL condition in \ref{Appendix: Positivity in multiD} and was seen to maintain positivity of the solution though it does not guaranty positivity of the partial internal energies.

Numerical experiments in two-space dimensions are given in the remainder of this section including tests on high-order accuracy and entropy conservation of the scheme, together with the simulation of a shock-bubble interaction problem.

\subsection{Advection of density and void fraction waves}
We here reproduce the test on accuracy from  \cref{ssec: density wave} and consider the pure advection of oblique void fraction and density waves in a uniform flow in a unit square with periodic boundary conditions. The initial condition reads
\begin{linenomath*}
\begin{equation}\label{eqn: 2D_IC_ho_accuracy}
 \alpha_{1,0}({\bf x}) = \frac{1}{2} + \frac{1}{4} \sin \big(4\pi (x+y)\big),\;
 \rho_{i,0}({\bf x})   = 1 + \frac{1}{2} \sin \big(2\pi (x+y)\big),\;
 u_{i,0}({\bf x})      = 1,\;
 v_{i,0}({\bf x})      = 1,\;
 \press_{i,0}({\bf x}) = 1, \; i=1,2.
\end{equation}
\end{linenomath*}
The EOS parameters in (\cref{Eqn: EOS}) are $\gamma_i = 1.4$ and $\press_{\infty_i} = 10$ for $i=1,2$. The obtained results are presented in \cref{table: high-order_accuracy_2D}. It is again observed that the expected $p+1$ order of convergence is achieved.

\begin{table}[ht]
    \centering
    \begin{tabulary}{1.0\textwidth}{ c|l|*{6}{c} }\hline
        $p$ & $h$ & $\norm{e_h}_{L^1(\Omega_h)}$ & $\mathcal{O}_1$ & $\norm{e_h}_{L^2(\Omega_h)}$ & $\mathcal{O}_2$ & $\norm{e_h}_{L^\infty(\Omega_h)}$ & $\mathcal{O}_\infty$\\ \hline
         \multirow{4}{*}{1}
         & 1/4  & 3.26E-01 & -    & 4.19E-01 & -    & 6.09E-01 & -\\
         & 1/8  & 2.92E-01 & 0.16 & 3.36E-01 & 0.32 & 5.13E-01 & 0.25\\
         & 1/16 & 8.40E-02 & 1.80 & 9.41E-02 & 1.84 & 1.57E-01 & 1.71\\
         & 1/32 & 1.42E-02 & 2.56 & 1.71E-02 & 2.46 & 3.39E-02 & 2.21\\
         & 1/64 & 4.14E-03 & 1.78 & 4.91E-03 & 1.80 & 1.03E-02 & 1.72\\ \hline
        \multirow{4}{*}{2}
         & 1/4  & 3.96E-02 & -    & 4.67E-02 & -    & 8.33E-02 & -\\
         & 1/8  & 8.75E-03 & 2.18 & 1.11E-02 & 2.07 & 4.03E-02 & 1.05\\
         & 1/16 & 1.41E-03 & 2.63 & 1.63E-03 & 2.78 & 4.57E-03 & 3.14\\
         & 1/32 & 1.19E-04 & 3.57 & 1.58E-04 & 3.36 & 4.10E-04 & 3.48\\
         & 1/64 & 1.23E-05 & 3.28 & 1.73E-05 & 3.19 & 6.26E-05 & 2.71\\ \hline
        \multirow{4}{*}{3}
         & 1/4  & 1.64E-02 & -    & 1.88E-02 & -    & 3.11E-02 & -\\
         & 1/8  & 2.20E-03 & 2.90 & 2.70E-03 & 2.80 & 6.96E-03 & 2.16\\
         & 1/16 & 7.10E-05 & 4.95 & 8.90E-05 & 4.92 & 2.56E-04 & 4.60\\
         & 1/32 & 1.35E-06 & 5.71 & 1.79E-06 & 5.64 & 9.68E-06 & 4.89\\
         & 1/64 & 6.50E-08 & 4.38 & 8.64E-08 & 4.37 & 5.21E-07 & 4.22\\ \hline
    \end{tabulary}
    \caption{Test for high-order accuracy with initial condition (\cref{eqn: 2D_IC_ho_accuracy}): different norms of the errors on $\tfrac{1}{2}(\rho_1+\rho_2)$ under grid and polynomial degree refinements and associated orders of convergence at final time $T_{max} = 5$.}
    \label{table: high-order_accuracy_2D}
\end{table}

\subsection{Entropy conservation}
We also check entropy conservation by using the same procedure as in \cref{sssec: entropy conservation} on the unit square with periodic boundary conditions. The initial condition is EC in \cref{table: RP_IC} with zero transverse velocity, $v_i=0$ for $i=1,2$, and we keep the same EOS parameters. The global entropy budget, similar to (\cref{eqn: entropy_budget}), is displayed in \cref{table: EC_results_2D} when refining the time step. Again the conservation of entropy by the space discretization is observed.

\begin{table}[ht]
\centering
 \begin{tabular}{l |c |c} 
 \hline
 time step & $\mathcal{E}_\Omega(t)$ & $\mathcal{O}$ \\ [0.5ex] 
 \hline
 $\Delta t$    & 7.49E-04 & - \\ 
 $\Delta t/2$  & 1.07E-04 & 2.81\\
 $\Delta t/4$  & 1.37E-05 & 2.97 \\
 $\Delta t/8$  & 1.72E-06 & 2.99 \\
 $\Delta t/16$ & 2.15E-07 & 3.00 \\ 
 $\Delta t/32$ & 2.67E-08 & 3.01 \\
 $\Delta t/64$ & 3.19E-09 & 3.06 \\
 \hline
\end{tabular}
\caption{Global entropy budget (\cref{eqn: entropy_budget}) in two space dimensions and the corresponding order of convergence
$\mathcal{O}$.}
\label{table: EC_results_2D}
\end{table}

\subsection{Shock-bubble interaction}
This numerical test involves the interaction between a shock wave and a material discontinuity. The test was introduced by Haas and Sturtevant \cite{haas1987interaction} to experimentally study the interaction of a shock wave with a single discrete gaseous inhomogeneity. Later it was adopted as a numerical benchmark to validate the robustness and accuracy of various numerical schemes for compressible two-phase flows, see \cite{quirk1996dynamics,saurel2003multiphase,giordano2006richtmyer,kawai2011high,hu2006conservative,johnsen2006implementation,terashima2009front,renac2020multicomp} and references therein.  

The computational domain $\Omega_h = [0,6.5]\times[0,1.78]$ is discretized using a Cartesian mesh with $1300 \times 356$ elements. The initial condition involves a bubble of unit diameter containing a mixture of $95\%$ of helium by volume ($\alpha_1=0.95$) and $5\%$ of air, to exclude resonance effects (\cref{eqn: resonance}), in a domain filled with  $5\%$ of air. The center of the bubble is located at ${\bf x}=(3.5,0.89)$. A left moving shock is initially placed at the rightmost edge of the bubble, $x_0=4$, and then moves to the left and interacts with the bubble. The initial condition is provided in \cref{table: SBI_IC}.

\begin{table}[ht]
\centering
    \begin{tabulary}{1.0\textwidth}{ c|*{5}{c} }\hline
                   & $\alpha_1$ & $\rho_i$ & $u_i$     & $v_i$ & $\mathrm{p}_i$\\\hline
    Pre-shock air ($i=2$)  & 0.05 & 1.3764 & -0.3336 & 0.0   & 1.1213\\\hline
    Helium bubble ($i=1$)  & 0.95 & 0.1819 & 0.0     & 0.0   & 0.7143\\\hline
    Post-shock air ($i=2$) & 0.05 & 1.0    & 0.0     & 0.0   & 0.7143\\
     \hline
    \end{tabulary}
    \caption{Physical parameters for the initial condition of the shock-bubble interaction problem.}
    \label{table: SBI_IC}
\end{table}

The EOS parameters for helium and air are $\gamma_1 = 1.648$ and $\Cv_1 = 6.06$, and $\gamma_2 = 1.4$ and $\Cv_2 = 1.786$, respectively. The physical model does not involve viscous effects so to avoid oscillations of the interface we smoothen the initial condition around the material interface following \cite{kawai2011high,houim2011low,billet2008impact}. The numerical test is performed using periodic boundary conditions at the top and bottom boundaries, and non-reflective conditions on the left and right boundaries.

\Cref{result: sbi} illustrates the deformation of the He bubble as the shock passes through it. The plotted fields are those of the void fraction for phase 1, the total pressure and numerical Schlieren.  It is observed that the material interface and the shock are accurately captured without excessive smearing of the contact. Note however that, for the Baer-Nunziato model, the pressure field shows the presence of a secondary shock inside the bubble (see e.g. the Schlieren at $t=62 \mu$s). This secondary shock is due to the presence of air inside the bubble. Furthermore, as the shock leaves the bubble, vortices are generated on the bubble interface as a result of the Kevin-Helmoltz instability.

\Cref{Fig: SBI space-time diags.} shows the space-time diagram for three characteristic points on the interface of the bubble. We compare the results obtained with the DGSEM scheme to reference data from \cite{kawai2011high}. The deformation of the bubble shows complete agreement with the reference data and indicate that the smooth initial condition does not affect the global deformation of the bubble.

\begin{figure}[H]
 \center
  \subfloat[$t =32 \mu$s]{
  \includegraphics[width=.20\paperwidth,trim=0.3cm 0.3cm 0.3cm 0.3cm,clip=true]{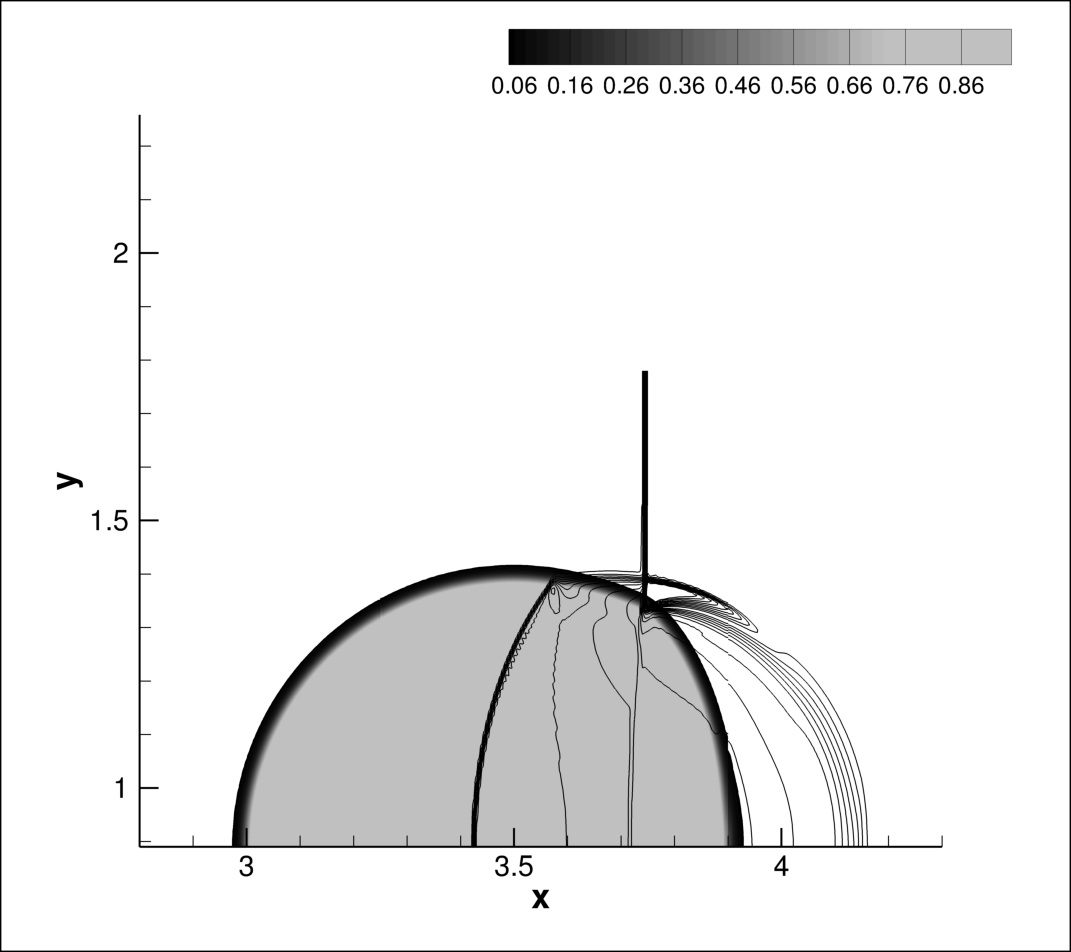}
  \includegraphics[width=.20\paperwidth,trim=0.3cm 0.3cm 0.3cm 0.3cm,clip=true]{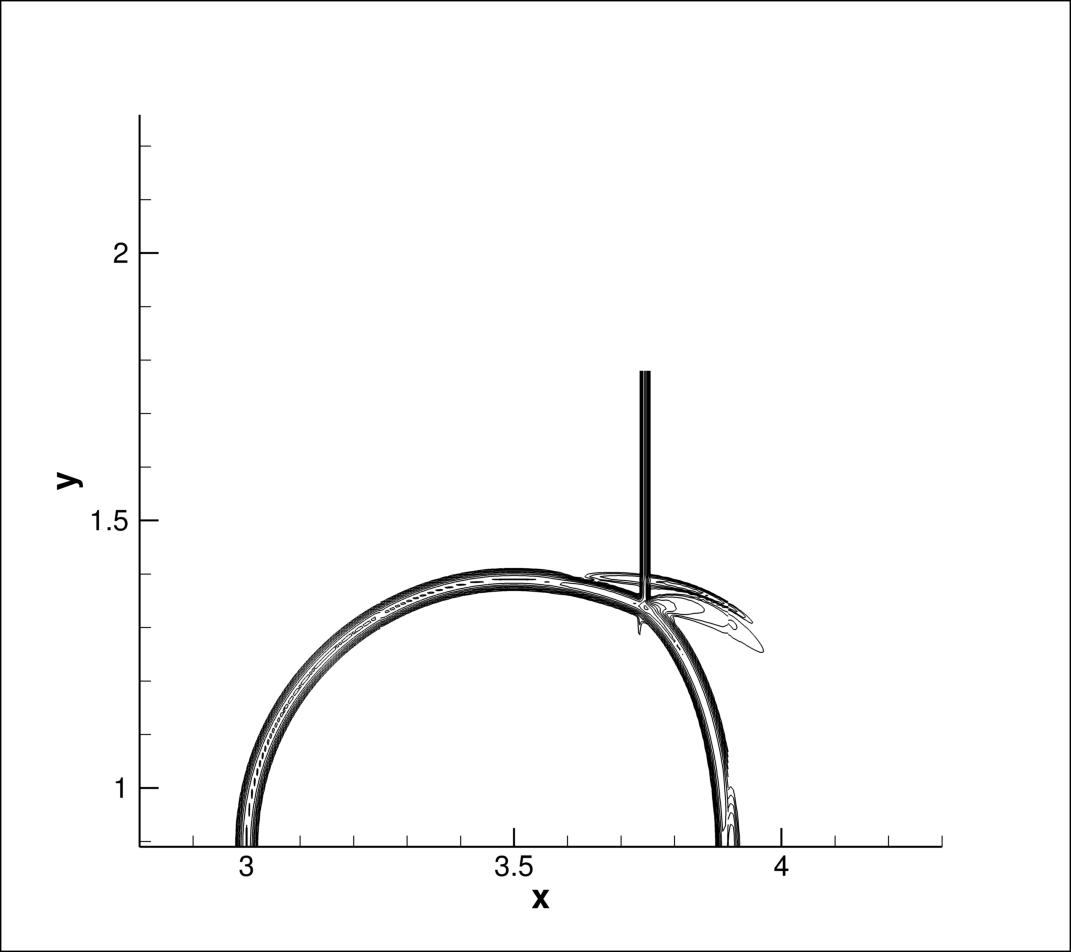}}
  \subfloat[$t=240 \mu$s]{
  \includegraphics[width=.20\paperwidth,trim=0.3cm 0.3cm 0.3cm 0.3cm,clip=true]{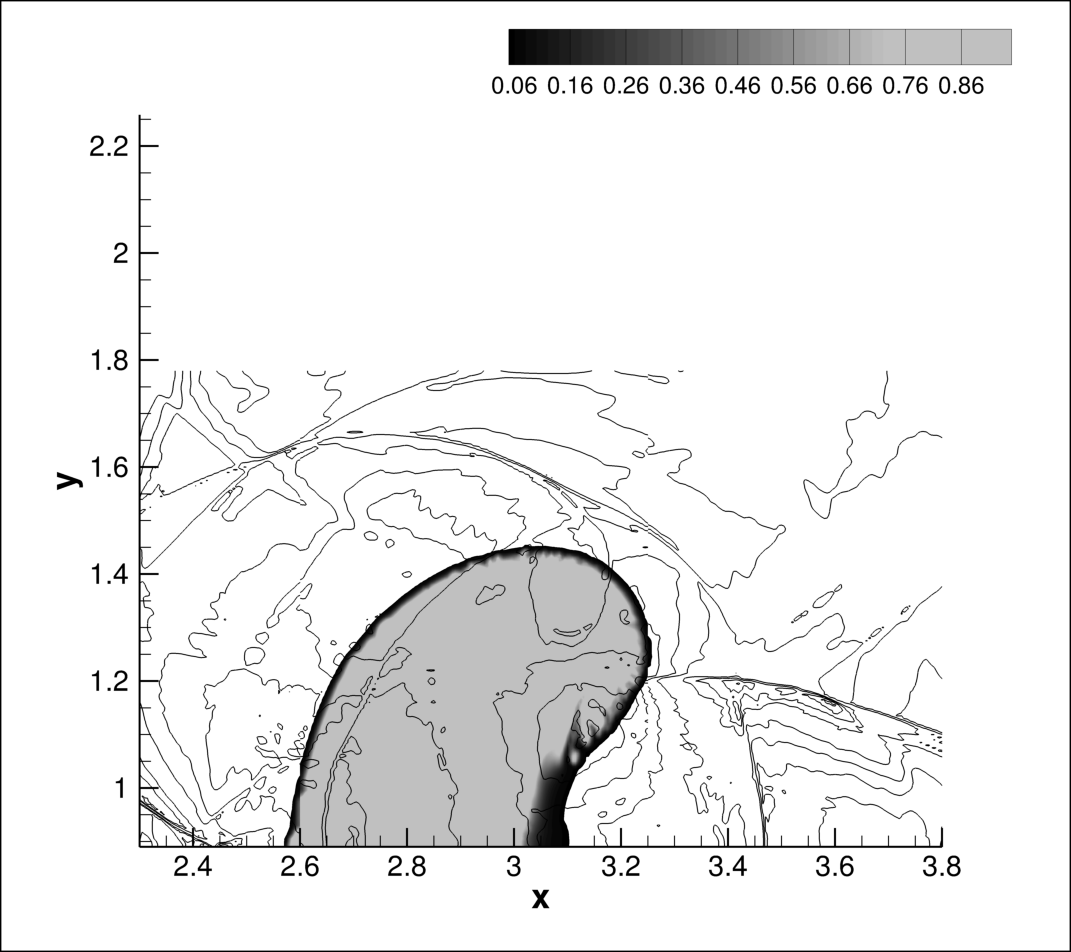}
  \includegraphics[width=.20\paperwidth,trim=0.3cm 0.3cm 0.3cm 0.3cm,clip=true]{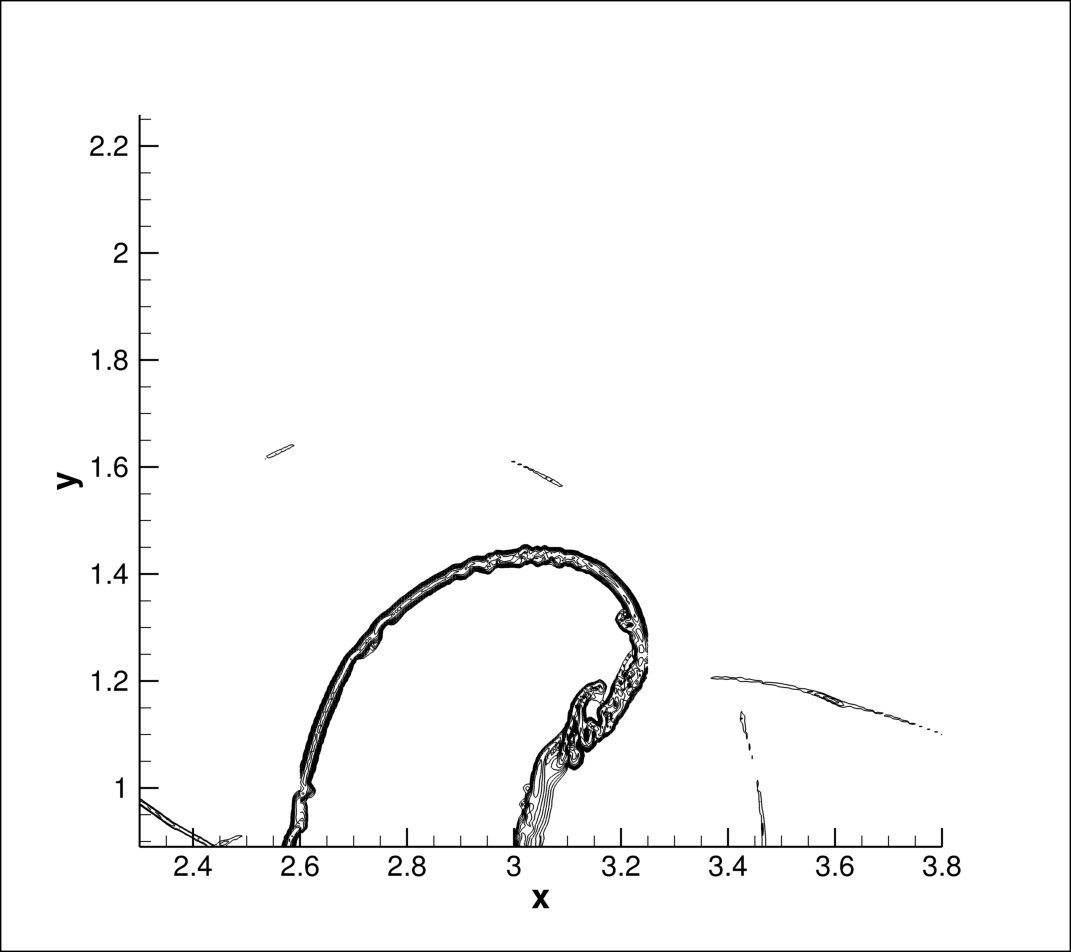}} \\
 \subfloat[$t=62 \mu$s]{
  \includegraphics[width=.20\paperwidth,trim=0.3cm 0.3cm 0.3cm 0.3cm,clip=true]{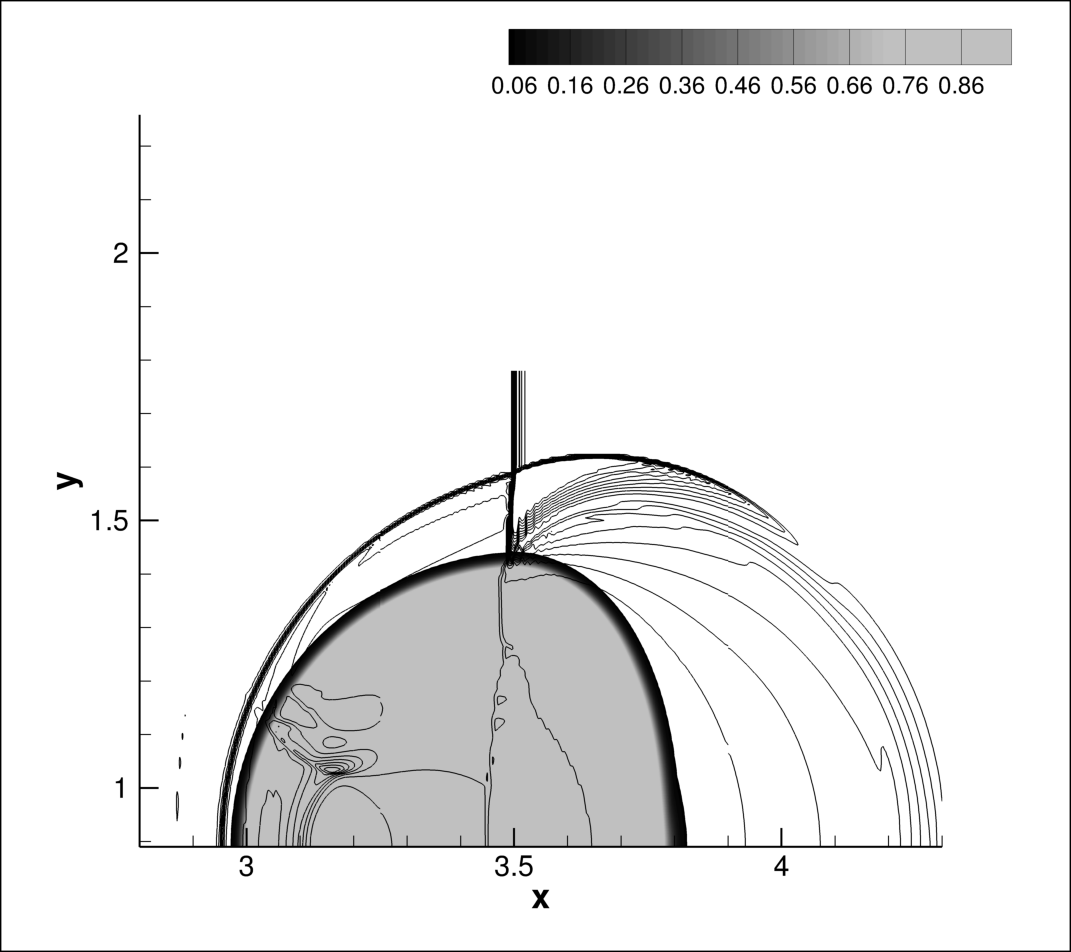}
  \includegraphics[width=.20\paperwidth,trim=0.3cm 0.3cm 0.3cm 0.3cm,clip=true]{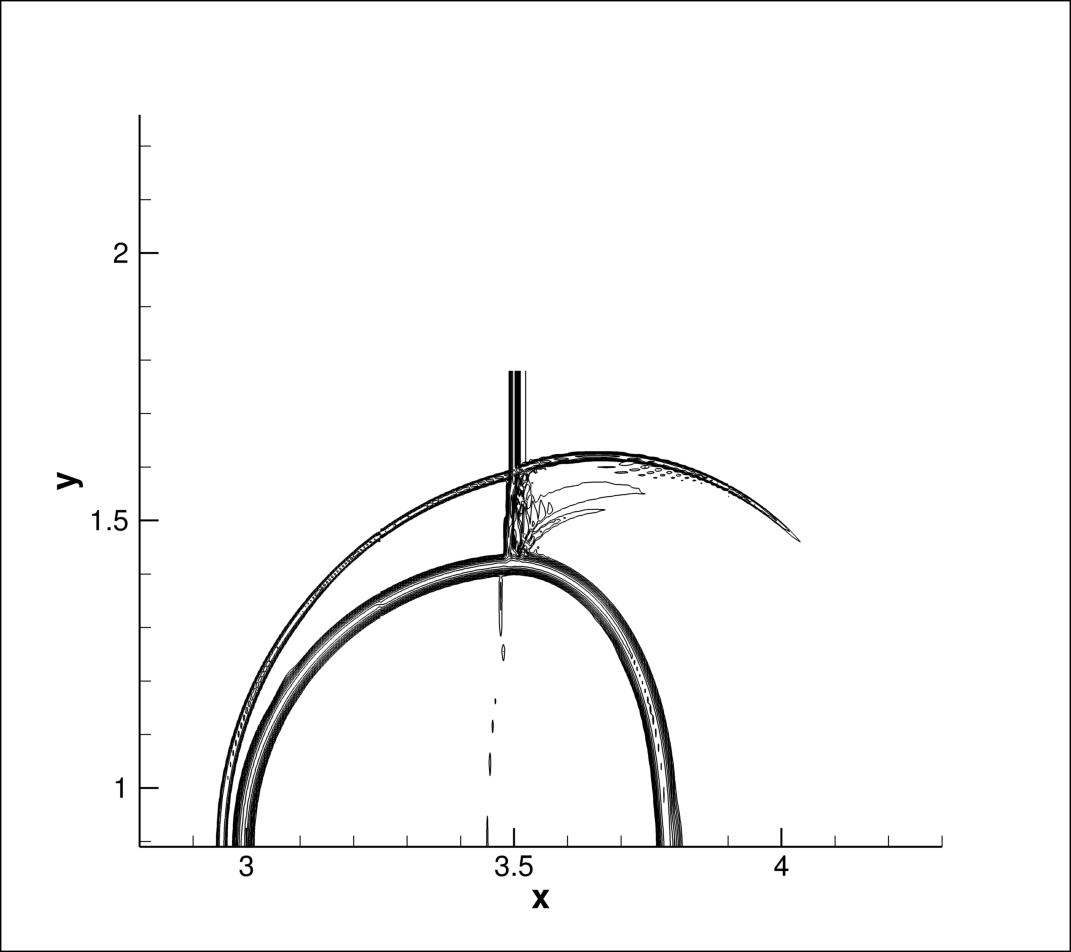}}
  \subfloat[$t=427 \mu$s]{
  \includegraphics[width=.20\paperwidth,trim=0.3cm 0.3cm 0.3cm 0.3cm,clip=true]{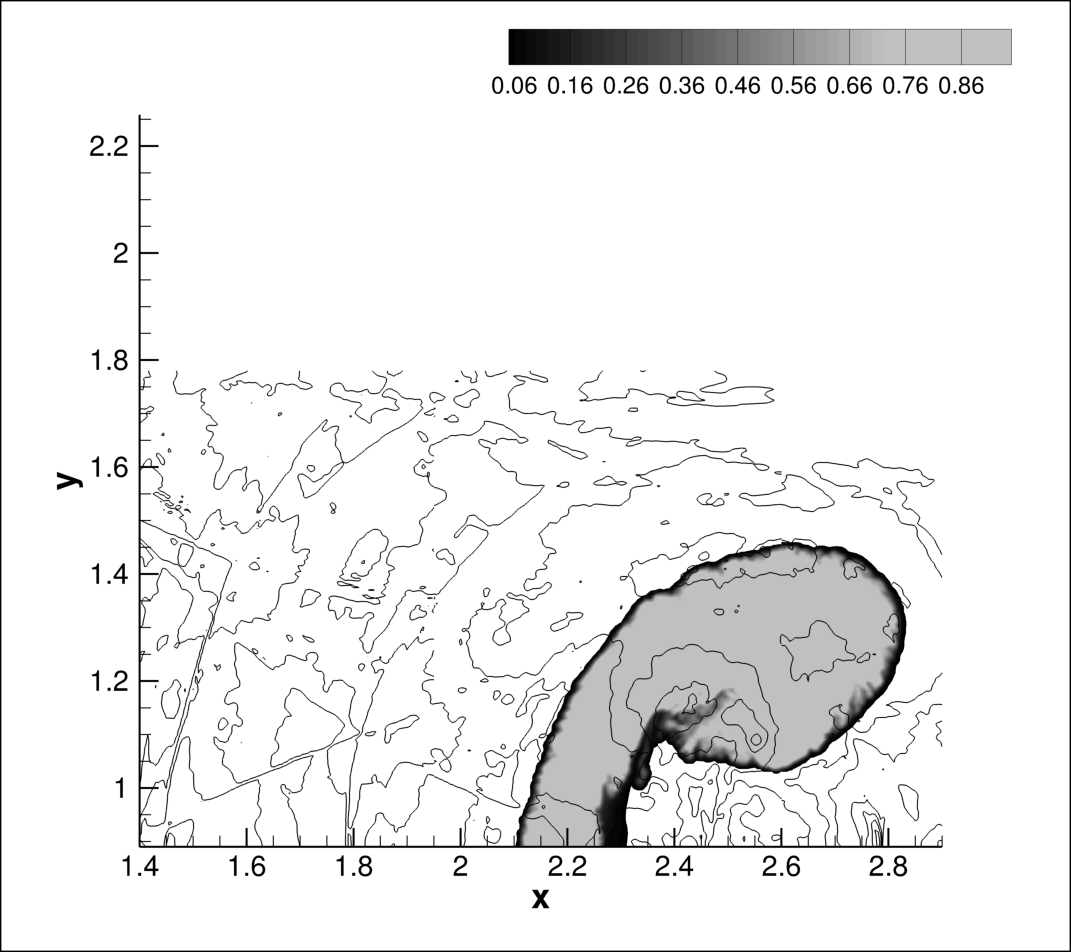}
  \includegraphics[width=.20\paperwidth,trim=0.3cm 0.3cm 0.3cm 0.3cm,clip=true]{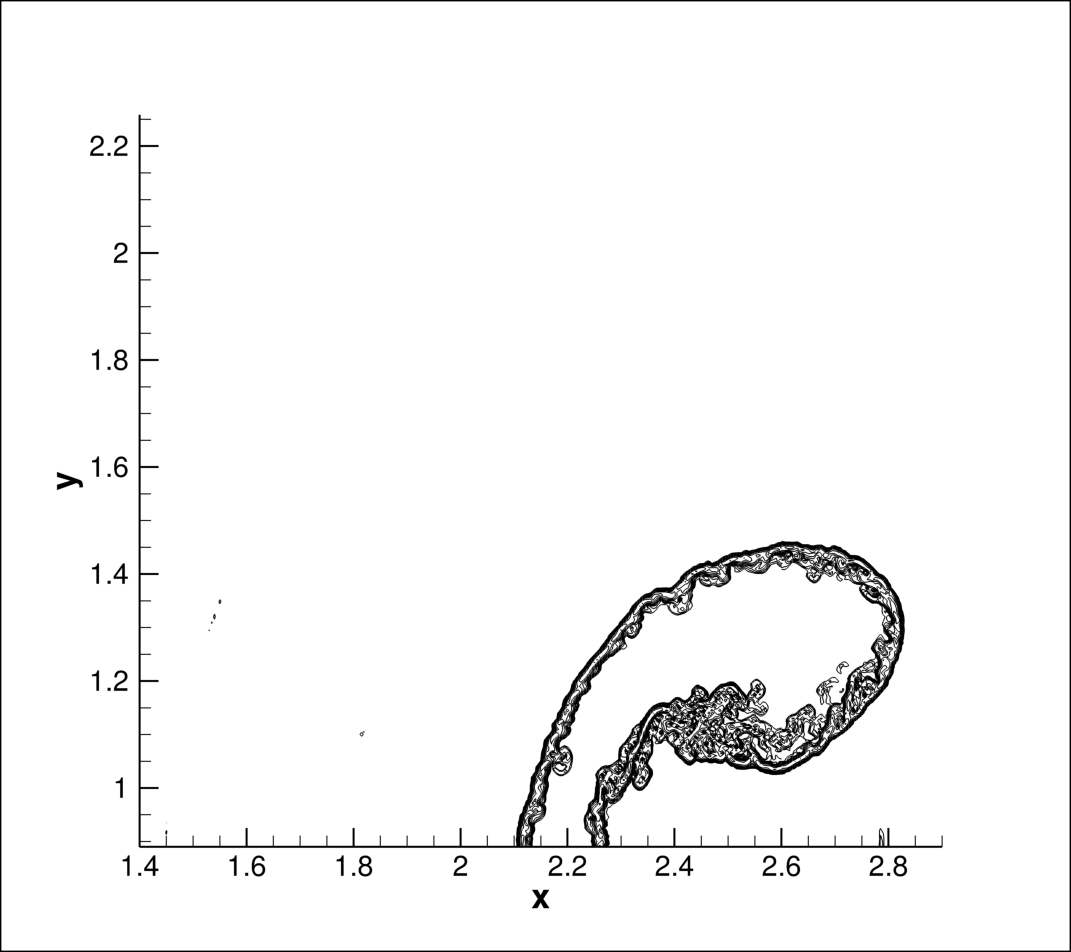}} \\
 \subfloat[$t=102 \mu$s]{
  \includegraphics[width=.20\paperwidth,trim=0.2cm 0.2cm 0.2cm 0.2cm,clip=true]{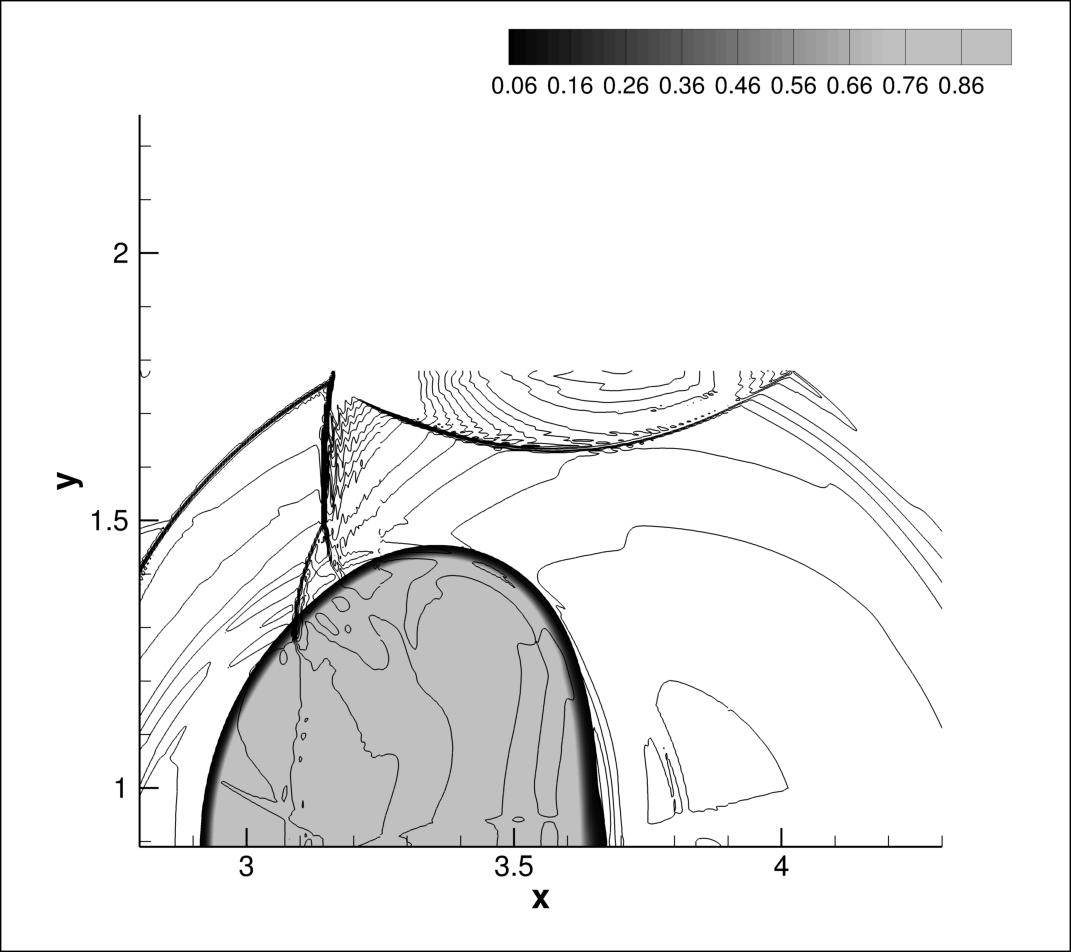}
  \includegraphics[width=.20\paperwidth,trim=0.2cm 0.2cm 0.2cm 0.2cm,clip=true]{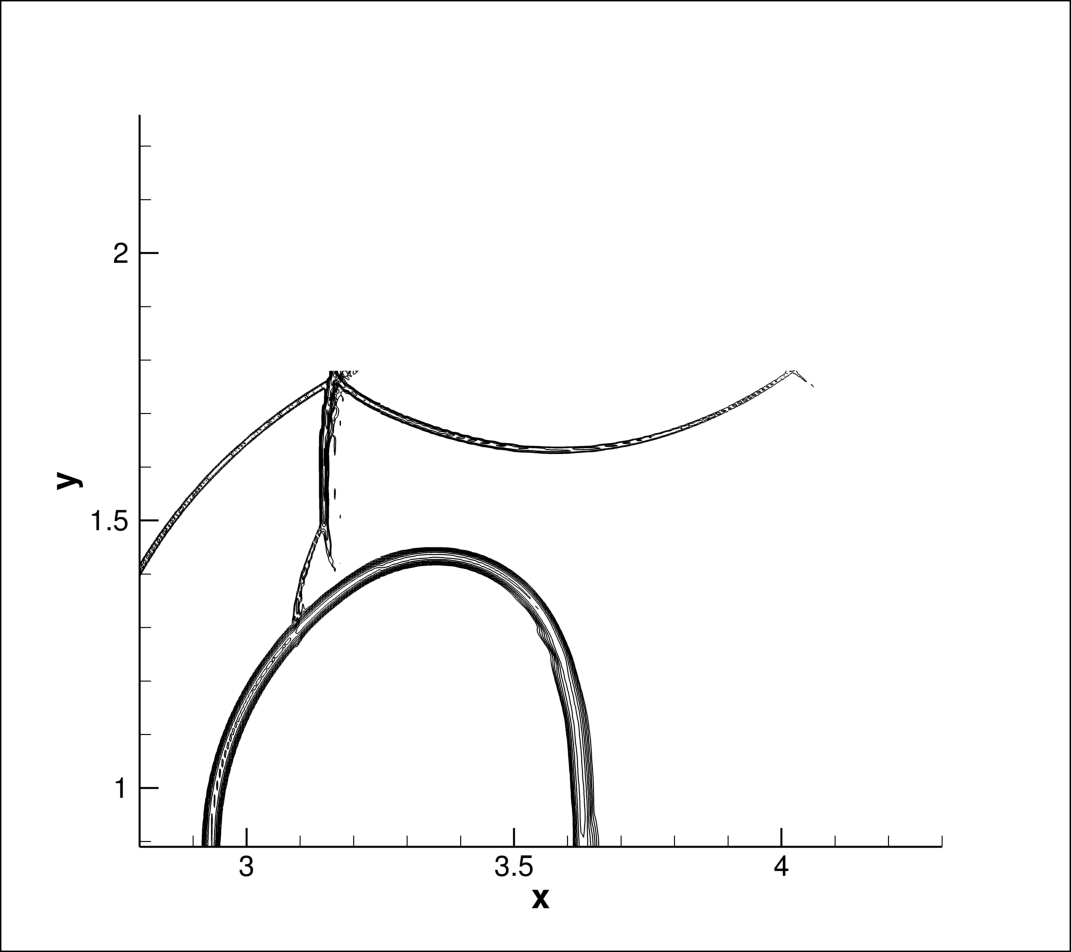}}
  \subfloat[$t=674 \mu$s]{
  \includegraphics[width=.20\paperwidth,trim=0.2cm 0.2cm 0.2cm 0.2cm,clip=true]{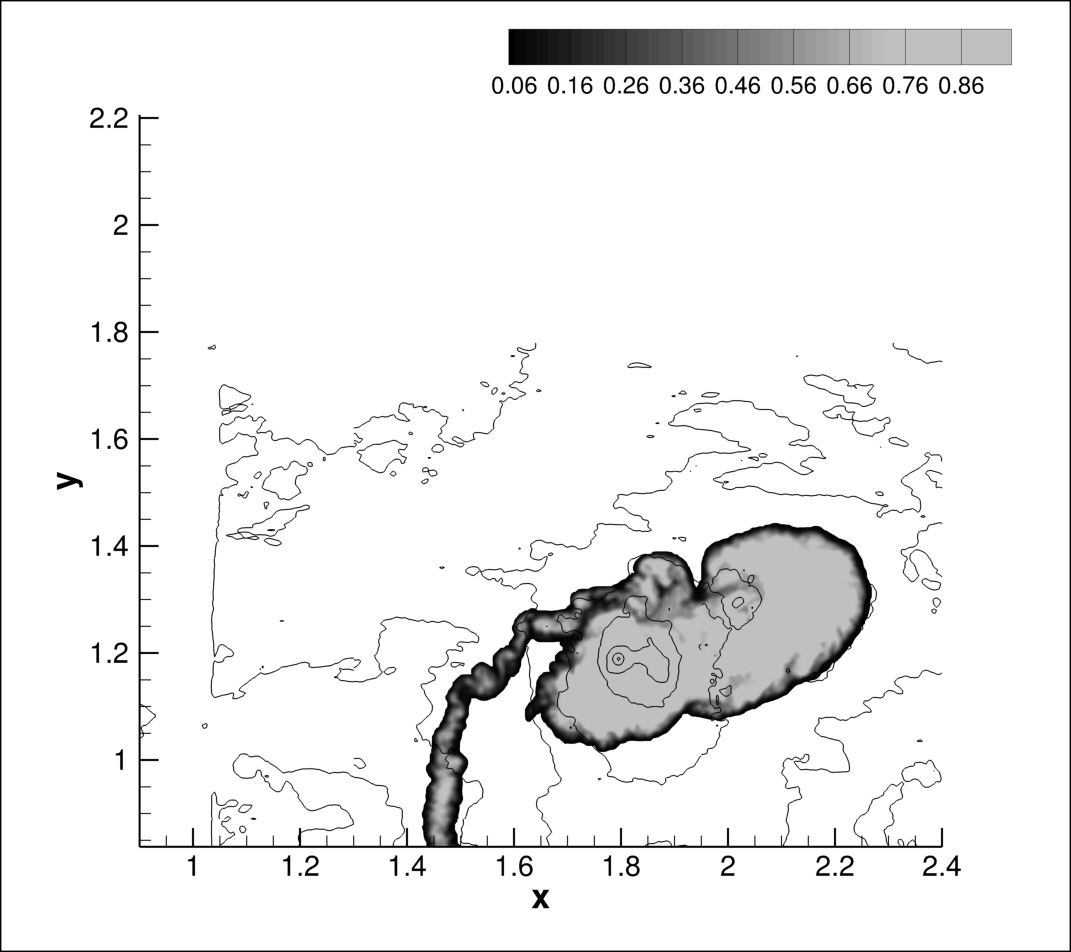}
  \includegraphics[width=.20\paperwidth,trim=0.2cm 0.2cm 0.2cm 0.2cm,clip=true]{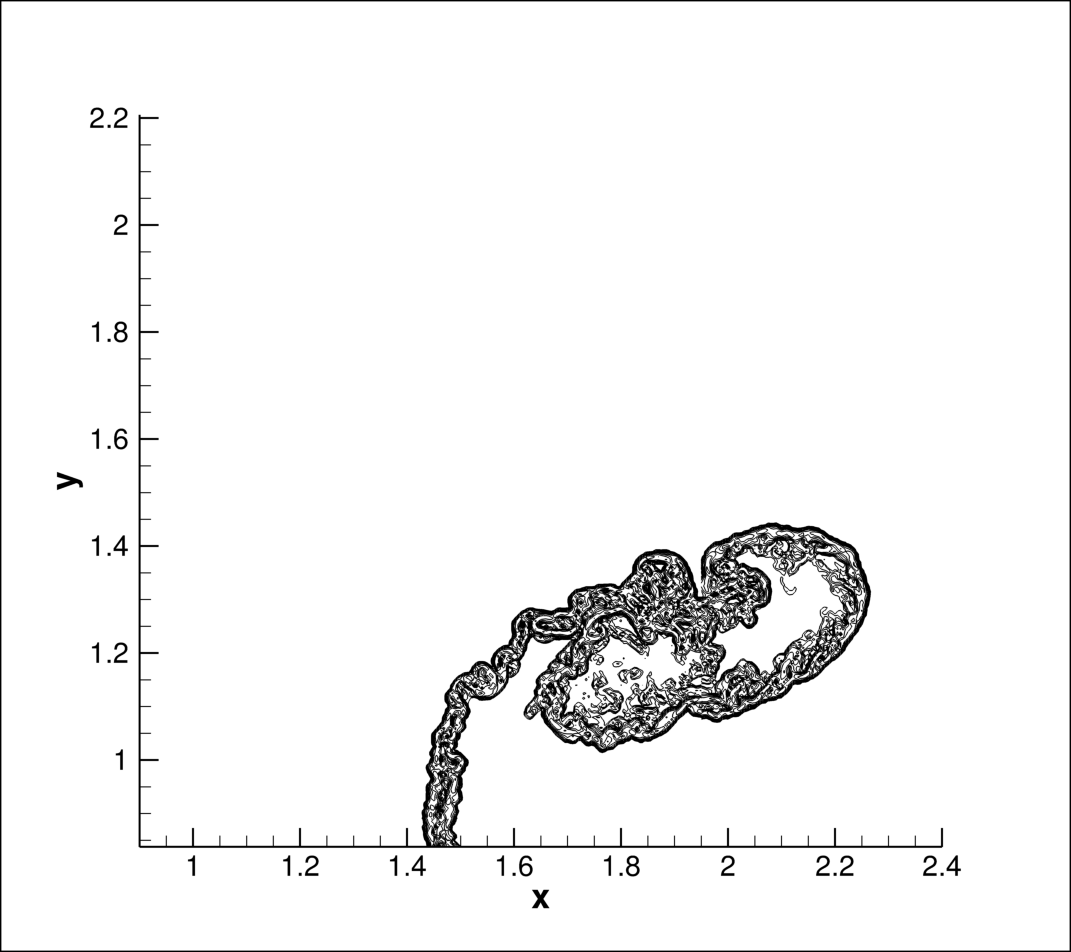}}
  \caption{The snapshots of the deformation of the He bubble due to the left traveling shock at various physical times. For each snapshot, the left plot displays contours of the void fraction $\alpha_1$ and of the total pressure $\press=\alpha_1\press_1+\alpha_2\press_2$, while the right plot shows the Schlieren $\phi = \exp(|\nabla \rho|/|\nabla \rho|_{\max})$, with $\rho=\alpha_1\rho_1+\alpha_2\rho_2$, obtained with a polynomial degree $p=3$.}
  \label{result: sbi}
\end{figure}

\begin{figure}[H]
 \center
  \includegraphics[height=.25\paperwidth]{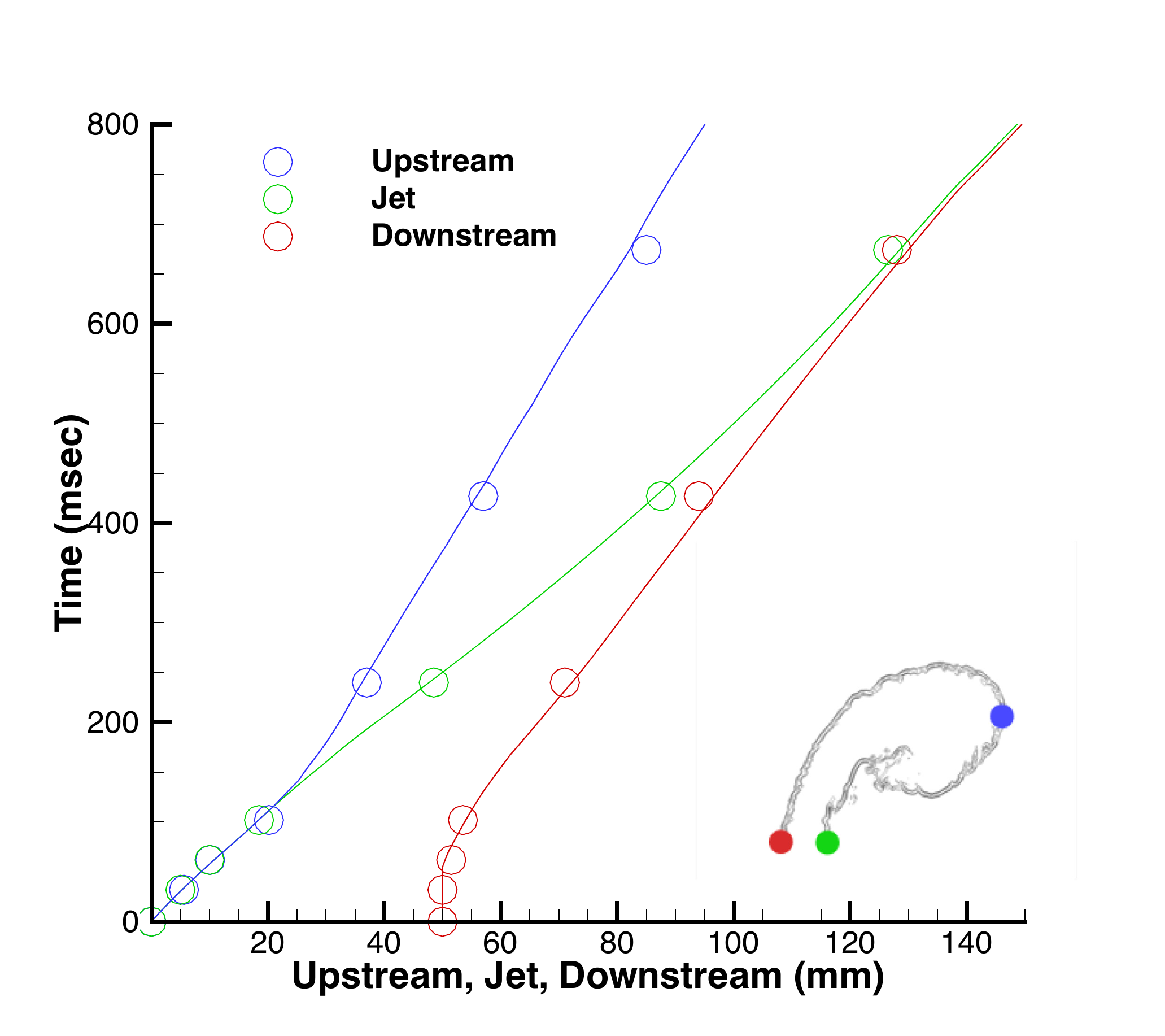}
  \caption{Space-time diagram for three characteristic points on the interface of the He bubble. The solid lines are the reference data from \cite{kawai2011high}, while the symbols are the results obtained with the present DGSEM scheme and a polynomial degree $p=3$.}
  \label{Fig: SBI space-time diags.}
\end{figure}

%
%
\section{Concluding remarks} \label{sec: Conclusion}
In this work, we derive a high-order entropy stable scheme for the Baer-Nunziato model \cite{baer1986two,saurel1999multiphase} for flows of two separated immiscible fluids in complete disequilibria with respect to the chemical, mechanical, thermal, and thermodynamic processes. Here we focus on the discretization of the convective part of the model and neglect the disequilibria source terms. The exchange of information at the interfaces of the fluids is governed through interface variables of pressure and velocity, for which we choose closure laws \cite{coquel2002closure,gallouet2004numerical} that allow the material interface to be associated to a LD field and an entropy inequality in conservative form to be derived from the model. The model is closed with stiffened gas EOS relevant for both gas and liquid phases.

The space discretization is performed by using the semi-discrete entropy stable DGSEM framework proposed in \cite{renac2019entropy}, which involves modifying the integration over cell elements by replacing the physical fluxes with two point entropy conservative fluxes in fluctuation form \cite{pares2006numerical,castro2013entropy}, while employing entropy stable fluctuation fluxes at the cell interfaces. This framework is here generalized to include both conservative and nonconservative terms to allow a conservative discretization of the former ones. The entropy conservative fluxes are derived by using the condition in \cite{castro2013entropy}, to which we add upwind type dissipation to obtain the entropy stable fluxes. The semi-discrete scheme is high-order accurate for smooth solutions, satisfies an entropy inequality, and is kinetic energy preserving.

We use a method of lines with an explicit time integration and propose conditions on the numerical parameters that guarantee the positivity of the cell-averaged partial densities and a maximum principle on the void fraction for the fully discrete scheme coupled with a first-order forward Euler discretization. High-order integration in time is performed using strong stability-preserving explicit Runge-kutta schemes \cite{shu1988efficient}. The positivity of the solution is then extended to nodal values using a posteriori limiters, inspired from \cite{zhang2010positivity,zhang2010maximum,perthame1996positivity}. 

The numerical tests involve specific test cases that validate the high-order accuracy and entropy conservation of the semi-discrete scheme in one and two space dimensions. Riemann problems are performed in one space dimension involving the development of strong shocks, contacts, near vacuum regions, and vanishing phases. The results obtained with a fourth-order scheme show that the present method captures the physically relevant solution. The intermediate states are well resolved, as well as the shocks and contacts and the computation is shown to be robust in situations close to either vacuum, or resonance. Furthermore, the application to the simulation of a shock-bubble interaction problem in two space dimensions confirm the accurate approximation of the shock and material interfaces.

\section*{Acknowledgement}

The authors would like to thank Prof. Soshi Kawai for sharing the reference data for the space-time diagrams in \cref{Fig: SBI space-time diags.}.

%
%
\appendix

\section{DGSEM in multiple space dimensions} \label{Appendix: DGSEM in 2D}

We here extend the DGSEM to multiple space dimensions and restrict ourselves to Cartesian meshes. For the sake of clarity we introduce the scheme in two space dimensions, $d=2$, on uniform grids without loss of generality.

The physical domain $\Omega$ is discretized with a Cartesian grid $\Omega_h$ with elements $\kappa_{i,j}=[x_{i-\frac{1}{2}},x_{i+\frac{1}{2}}]\times[y_{j-\frac{1}{2}},y_{j+\frac{1}{2}}]$ with $x_{i+\frac{1}{2}}=ih_x$, $y_{j+\frac{1}{2}}=jh_y$, where $h_x>0$ and $h_y>0$ are the space steps.The Cartesian coordinate system is denoted as $(0,{\bf e_x}, {\bf e_y})$. Each element $\kappa_{i,j}$ is defined through the mapping $\textbf{x}_{i,j}:I^2\ni(\xi,\eta)\mapsto{\bf x}=\textbf{x}_{i,j}(\xi,\eta)\in\kappa_{i,j}$ with $I^2=[-1,1]^2$. The function space $\mathcal{V}^p_h$ restricted onto an element $\kappa_{i,j}$ is spanned with functions defined as tensor products of one-dimensional Lagrange polynomials associated to the Gauss-Lobatto nodes (see \cref{Numerical soln}):
\begin{linenomath*}
\begin{equation*}
    \phi^{kl}_{i,j}\big(\textbf{x}_{i,j}(\xi,\eta)\big):= \ell_k(\xi)\ell_l(\eta), \quad 0 \leqslant k,l \leqslant p,
\end{equation*}
\end{linenomath*}
which satisfy the cardinality relation $\ell_k(\xi_{\tilde{k}})\ell_l(\eta_{\tilde{l}}) = \delta_{\tilde{k}k}\delta_{\tilde{l}l}$ for $0 \leqslant \tilde{k},k,\tilde{l},l \leqslant p$. The approximate solution is now represented as
\begin{linenomath*}
\begin{equation*}
    \vecu_h(\textbf{x},t):= \sum^p_{k,l=0} \phi^{kl}_{i,j}(\textbf{x}) \textbf{U}^{kl}_{i,j}(t) \quad \forall \textbf{x} \in \kappa_{i,j},\; t \geqslant 0.
\end{equation*}
\end{linenomath*}


The integrals over the physical elements and faces are approximated with Gauss-Lobatto quadratures:
\begin{linenomath*}
\begin{equation*}
        \int_{\kappa_{i,j}} f(\textbf{x})dV \approx \sum^p_{k,l=0}\omega_k\omega_l \frac{h_xh_y}{4}f(\textbf{x}_{i,j}^{kl}), \quad \int_{e} f(\textbf{x})dS \approx \sum^p_{k=0}\omega_k \frac{|e|}{2}f(\textbf{x}_e^{k}),
\end{equation*}
\end{linenomath*}
where $\omega_k$ and $\omega_k\omega_l$ are the Gaussian weights, and $|e|$ is the length of $e$. 

The semi-discrete DGSEM for the discretization of (\cref{Eqn: 2D-BNM}) then reads
\begin{linenomath*}
\begin{align*}
 \frac{h_xh_y}{4}\frac{d{\bf U}_{i,j}^{kl}}{dt} &+ \omega_l\frac{h_y}{2}\bigg(\sum^p_{m=0}\omega_kD_{km}\tD(\textbf{U}^{kl}_{i,j},\textbf{U}^{ml}_{i,j},{\bf e_x}) + \delta_{kp}\D^-(\textbf{U}^{pl}_{i,j},\textbf{U}^{0l}_{i+1,j},{\bf e_x}) +\delta_{k0}\D^+(\textbf{U}^{pl}_{i-1,j},\textbf{U}^{0l}_{ij},{\bf e_x})\bigg) \\
 &+ \omega_k\frac{h_x}{2}\bigg(\sum^p_{m=0}\omega_lD_{lm}\tD(\textbf{U}^{kl}_{i,j},\textbf{U}^{km}_{i,j},{\bf e_y}) + \delta_{lp}\D^-(\textbf{U}^{kp}_{i,j},\textbf{U}^{k0}_{i,j+1},{\bf e_y}) +\delta_{l0}\D^+(\textbf{U}^{kp}_{i,j-1},\textbf{U}^{k0}_{ij},{\bf e_y})\bigg) = 0,
\end{align*}
\end{linenomath*}
with
\begin{linenomath*}
\begin{equation*}
 \tD(\vecu^-,\vecu^+,{\bf n}) := \D^-_{ec}(\vecu^-,\vecu^+,{\bf n}) - \D^+_{ec}(\vecu^+,\vecu^-,{\bf n}),
\end{equation*}
\end{linenomath*}
and the numerical fluxes are defined in \ref{Appendix: EC ES flux 2D}.

%
\section{Entropy conservative and entropy stable fluxes in multiple space dimensions} \label{Appendix: EC ES flux 2D}
In multidimensional space, for solutions belonging to the phase space
\begin{linenomath*}
\begin{equation*}
 \Omega_{\BNM} = \big\{\vecu\in\mathbb{R}^{5+2d}:\, \rho_i>0,\, {\bf v}_i\in\mathbb{R}^d,\, e_i>0,\, 0<\alpha_i<1, \; i=1,2 \big\},
\end{equation*}
\end{linenomath*}
the entropy conservative fluxes (\cref{Eqn: fluctuation flux}) are defined as follows:
\begin{linenomath*}
\begin{equation*}
    \D^\mp_{ec}(\vecu^-,\vecu^+,{\bf n}) = \pm\textbf{h}(\vecu^-,\vecu^+,\textbf{n})\mp\textbf{f}(\vecu^\mp)\cdot\textbf{n}+\normalfont{\textbf{d}}^\mp(\vecu^-,\vecu^+,\textbf{n}),
\end{equation*}
\end{linenomath*}
for the system (\cref{Eqn: 2D-BNM}). They are assumed to be consistent $\D^\mp_{ec}(\vecu,\vecu,{\bf n}) = 0$ and are defined as follows:
\begin{linenomath*}
\begin{subequations}
\begin{equation*}
    \normalfont{\textbf{h}}(\vecu^-,\vecu^+,\textbf{n}) := 
    \begin{pmatrix*}
    0\\
    \displaystyle\xoverline{\alpha}_i\hat{\rho}_i\xoverline{\textbf{v}}_i\cdot \textbf{n}\\
    \displaystyle\xoverline{\alpha}_i\left(\hat{\rho}_i(\xoverline{\textbf{v}}_i\cdot\textbf{n})\xoverline{\textbf{v}}_i+\frac{\xoverline{\mathrm{p}_i\theta_i}}{\xoverline{\theta}_i}\n\right)\\
    \xoverline{\alpha}_i\left(\hat{\rho}_i\left(\displaystyle\frac{\Cv_i}{\hat{\theta}_i}+\frac{\textbf{v}^-_i\cdot\textbf{v}^+_i}{2}\right)+\frac{\xoverline{\press_i\theta_i}}{\xoverline{\theta_i}}+\press_{\infty,i}\right)\xoverline{\textbf{v}}_i\cdot\textbf{n}
    \end{pmatrix*}
    -\beta_s\frac{\lb \alpha_i\rb}{2}
    \begin{pmatrix*}
    1\\ 
    \hat{\rho}_i\\ 
    \hat{\rho}_i\xoverline{\textbf{v}}_i \\ 
    \displaystyle\hat{\rho}_i\left(\frac{\Cv_i}{\hat{\theta}_i}+\frac{\textbf{v}^-_i\cdot\textbf{v}^+_i}{2}\right) + \press_{\infty,i}
    \end{pmatrix*},
\end{equation*}
    \begin{equation*}
    \normalfont{\textbf{d}}^\pm(\vecu^-,\vecu^+,{\bf n}):=\frac{\lb\alpha_i\rb}{2}
    \begin{pmatrix}
    {\bf v}_{\mathrm{I}}^{\pm}\cdot\textbf{n}\\
    0\\ 
    -\pI^{\pm}\n\\ 
    -\pI^\pm{\bf v}_{\mathrm{I}}^{\pm}\cdot\textbf{n}
    \end{pmatrix},\quad i\in \{1,2\}.
\end{equation*}
\end{subequations}
\end{linenomath*}

The entropy stable fluxes read
\begin{linenomath*}
\begin{equation*}
    \normalfont{\textbf{D}}^\pm(\vecu^-,\vecu^+,\n)= \textbf{D}^\pm_{ec}(\vecu^-,\vecu^+,\n) \pm \textbf{D}_\nu(\vecu^-,\vecu^+,\n),
\end{equation*}
\end{linenomath*}
with
\begin{linenomath*}
\begin{equation*}
    \normalfont{\textbf{D}}_\nu(\vecu^-,\vecu^+,\n) = \frac{\epsilon_\nu}{2}\max\big(\rho_\textbf{A}(\vecu^-,\n),\rho_\textbf{A}(\vecu^+,\n)\big)
    \begin{pmatrix}
    0\\ 
    \lb\rho_i\rb \\ 
    \lb\rho_i \textbf{v}_i\rb\\  
    \displaystyle\left(\frac{\Cv_i}{\hat{\theta}_i}+\frac{\textbf{v}^-_i\cdot\textbf{v}^+_i}{2}\right)\lb\rho_i\rb+\xoverline{\rho_i}\lb E_i\rb
    \end{pmatrix}, \quad i\in \{1,2\},
\end{equation*}
\end{linenomath*}
where $\epsilon_\nu\geqslant0$ and $\rho_\textbf{A}(\vecu,\n)=\max_{i=1,2}(|{\bf v}_i\cdot{\bf n}|+c_i)$.

\section{Condition for positivity of the cell-averaged solution in multiple space dimensions}\label{Appendix: Positivity in multiD}
The condition for positivity of the solution is based on the extension of \cref{thm: positivity}. We introduce $\lambda_x=\tfrac{\Delta t}{h_x}$ and $\lambda_y=\tfrac{\Delta t}{h_y}$ with $\Delta t>0$ the time step. Let $\rho^{0\leqslant k,l\leqslant p,n}_{i,j} >0, 1>\alpha^{0\leqslant k,l\leqslant p,n}_{i,j}>0$, then the cell-averaged partial densities and void fraction are positive, at time $t^{(n+1)}$, under the following CFL condition: 

\begin{linenomath*}
\begin{equation}\label{eq:2D_CFL_cond}
    \begin{aligned}
        (\lambda_x+\lambda_y) \max_{\kappa\in\Omega_h} \max_{{\bf u}={\bf u}_1,{\bf u}_2} \max_{0\leqslant m\leqslant p} \Bigg( &\max_{0\leqslant k\leqslant p} \frac{1}{\omega_k} \Bigg(\sum_{l=0}^p \omega_lD_{lk} \uI_{i,j}^{lm} + \delta_{kp}\frac{\beta_{s_{i+1/2}}^m-\uI^{pm}_{i,j}}{2} + \delta_{k0}\frac{\beta_{s_{i-1/2}}^m+\uI^{0m}_{i,j}}{2}\Bigg),\\
				&\max_{0\leqslant l\leqslant p} \frac{1}{\omega_l} \Bigg(\sum_{k=0}^p \omega_kD_{kl} \vI_{i,j}^{mk} + \delta_{lp}\frac{\beta_{s_{j+1/2}}^m-\vI^{mp}_{i,j}}{2} + \delta_{l0}\frac{\beta_{s_{j-1/2}}^m+\vI^{m0}_{i,j}}{2}\Bigg),\\
				&\frac{1}{\omega_0} \Big(\frac{(\beta_{s_{i-1/2}}^m-\xoverline{u}_{i-1/2}^m)\hat{\rho}_{i-1/2}^m}{2\rho^{0m}_{i,j}} + \frac{\epsilon_{\nu_{i-1/2}}^m}{\alpha^{0m}_{i,j}}\Big), \frac{1}{\omega_p} \Big(\frac{(\beta_{s_{i+1/2}}^m+\xoverline{u}_{i+1/2}^m)\hat{\rho}_{i+1/2}^m}{2\rho^{pm}_{i,j}} + \frac{\epsilon_{\nu_{i+1/2}}^m}{\alpha^{pm}_{i,j}}\Big), \\
				&\frac{1}{\omega_0} \Big(\frac{(\beta_{s_{j-1/2}}^m-\xoverline{v}_{j-1/2}^m)\hat{\rho}_{j-1/2}^m}{2\rho^{m0}_{i,j}} + \frac{\epsilon_{\nu_{j-1/2}}^m}{\alpha^{m0}_{i,j}}\Big), \frac{1}{\omega_p} \Big(\frac{(\beta_{s_{j+1/2}}^m+\xoverline{v}_{j+1/2}^m)\hat{\rho}_{j+1/2}^m}{2\rho^{mp}_{i,j}} + \frac{\epsilon_{\nu_{j+1/2}}^m}{\alpha^{mp}_{i,j}}\Big)\Bigg)^{(n)} < \frac{1}{2},
    \end{aligned}
\end{equation}
\end{linenomath*}
where 
\begin{linenomath*}
\begin{align*}
    \beta_{s_{i+1/2}}^m &= \max_{i_p=1,2}\big(|u^{pm,n}_{i_p,i,j}|, |u^{0m,n}_{i_p,i+1,j}|\big), \quad \beta_{s_{j+1/2}}^m = \max_{i_p=1,2}\big(|v^{mp,n}_{i_p,i,j}|,|v^{m0,n}_{i_p,i,j+1}|\big), \quad 0\leqslant m \leqslant p, \\
		\xoverline{u}_{i+1/2}^m &= \tfrac{u^{pm,n}_{i,j}+u^{0m,n}_{i+1,j}}{2}, \quad \xoverline{v}_{j+1/2}^m = \tfrac{v^{mp,n}_{i,j}+v^{m0,n}_{i+1,j}}{2}, \quad 0\leqslant m \leqslant p, \\
		\hat{\rho}_{i+1/2}^m &= \tfrac{\rho^{0m,n}_{i+1,j}-\rho^{pm,n}_{i,j}}{\ln\rho^{0m,n}_{i+1,j}-\ln\rho^{pm,n}_{i,j}}, \quad  \hat{\rho}_{j+1/2}^m = \tfrac{\rho^{m0,n}_{i,j+1}-\rho^{mp,n}_{i,j}}{\ln\rho^{m0,n}_{i,j+1}-\ln\rho^{mp,n}_{i,j}},\quad 0\leqslant m \leqslant p,
\end{align*}
\end{linenomath*}
where $u$, $v$, and $\rho$ refer either to phase ${\bf u}_1$, or to ${\bf u}_2$ in (\cref{eq:2D_CFL_cond}).

\bibliographystyle{siam}
\bibliography{main}

\begin{thebibliography}{10}

\bibitem{abgrall1996prevent}
{\sc R.~Abgrall}, {\em How to prevent pressure oscillations in multicomponent
  flow calculations: a quasi conservative approach}, J. Comput. Phys., 125
  (1996), pp.~150--160.

\bibitem{abgrall2010comment}
{\sc R.~Abgrall and S.~Karni}, {\em A comment on the computation of
  non-conservative products}, J. Comput. Phys., 229 (2010), pp.~2759--2763.

\bibitem{ambroso2012godunov}
{\sc A.~Ambroso, C.~Chalons, and P.-A. Raviart}, {\em A {G}odunov-type method
  for the seven-equation model of compressible two-phase flow}, Comput. Fluids,
  54 (2012), pp.~67--91.

\bibitem{andrianov2004riemann}
{\sc N.~Andrianov and G.~Warnecke}, {\em The {R}iemann problem for the
  {B}aer--{N}unziato two-phase flow model}, J. Comput. Phys., 195 (2004),
  pp.~434--464.

\bibitem{baer1986two}
{\sc M.~Baer and J.~Nunziato}, {\em A two-phase mixture theory for the
  deflagration-to-detonation transition {(DDT)} in reactive granular
  materials}, Int. J. Multiphase Flow, 12 (1986), pp.~861--889.

\bibitem{berthon2012many}
{\sc C.~Berthon, F.~Coquel, and P.~G. LeFloch}, {\em Why many theories of shock
  waves are necessary: kinetic relations for non-conservative systems}, Proc.
  R. Soc. Edin. A, 142 (2012), pp.~1--37.

\bibitem{billet2008impact}
{\sc G.~Billet, V.~Giovangigli, and G.~De~Gassowski}, {\em Impact of volume
  viscosity on a shock--hydrogen-bubble interaction}, Combust. Theory Model.,
  12 (2008), pp.~221--248.

\bibitem{bohm2018entropy}
{\sc M.~Bohm, A.~Winters, G.~Gassner, D.~Derigs, F.~Hindenlang, and J.~Saur},
  {\em An entropy stable nodal discontinuous {G}alerkin method for the
  resistive {MHD} equations. {P}art {I}: Theory and numerical verification}, J.
  Comput. Phys.,  (2018).

\bibitem{castro2017well}
{\sc M.~J. Castro, T.~M. de~Luna, and C.~Par{\'e}s}, {\em Well-balanced schemes
  and path-conservative numerical methods}, in Handbook of Numer. Anal.,
  vol.~18, Elsevier, 2017, pp.~131--175.

\bibitem{castro2013entropy}
{\sc M.~J. Castro, U.~S. Fjordholm, S.~Mishra, and C.~Par{\'e}s}, {\em Entropy
  conservative and entropy stable schemes for nonconservative hyperbolic
  systems}, SIAM J. Numer. Anal., 51 (2013), pp.~1371--1391.

\bibitem{castro2006high}
{\sc M.~J. Castro, J.~Gallardo, and C.~Par{\'e}s}, {\em High order finite
  volume schemes based on reconstruction of states for solving hyperbolic
  systems with nonconservative products. applications to shallow-water
  systems}, Math. {C}omput., 75 (2006), pp.~1103--1134.

\bibitem{castro2008many}
{\sc M.~J. Castro, P.~G. LeFloch, M.~L. Mu{\~n}oz-Ruiz, and C.~Par{\'e}s}, {\em
  Why many theories of shock waves are necessary: Convergence error in formally
  path-consistent schemes}, J. Comput. Phys., 227 (2008), pp.~8107--8129.

\bibitem{castro2012central}
{\sc M.~J. Castro, C.~Par{\'e}s, G.~Puppo, and G.~Russo}, {\em Central schemes
  for nonconservative hyperbolic systems}, SIAM J. Sci. Comput., 34 (2012),
  pp.~B523--B558.

\bibitem{CHALONS2017592}
{\sc C.~Chalons and F.~Coquel}, {\em A new comment on the computation of
  non-conservative products using {R}oe-type path conservative schemes}, J.
  Comput. Phys., 335 (2017), pp.~592 -- 604.

\bibitem{Chandrashekar2013kep_es}
{\sc P.~Chandrashekar}, {\em Kinetic energy preserving and entropy stable
  finite volume schemes for compressible {E}uler and {N}avier-{S}tokes
  equations}, Commun. Comput. Phys., 14 (2013), pp.~1252--1286.

\bibitem{chen2017entropy}
{\sc T.~Chen and C.-W. Shu}, {\em Entropy stable high order discontinuous
  galerkin methods with suitable quadrature rules for hyperbolic conservation
  laws}, J. Comput. Phys., 345 (2017), pp.~427--461.

\bibitem{coquel2002closure}
{\sc F.~Coquel, T.~Gallou{\"e}t, J.-M. H{\'e}rard, and N.~Seguin}, {\em Closure
  laws for a two-fluid two-pressure model}, C. R. Acad. Sci. Paris, 334 (2002),
  pp.~927--932.

\bibitem{coquel2017positive}
{\sc F.~Coquel, J.-M. H{\'e}rard, and K.~Saleh}, {\em A positive and
  entropy-satisfying finite volume scheme for the {B}aer--{N}unziato model}, J.
  Comput. Phys., 330 (2017), pp.~401--435.

\bibitem{dal1995definition}
{\sc G.~Dal~Maso, P.~Le~Floch, and F.~Murat}, {\em Definition and weak
  stability of nonconservative products}, J. Math. Pures Appl., 74 (1995),
  pp.~483--548.

\bibitem{drew2006theory}
{\sc D.~A. Drew and S.~L. Passman}, {\em Theory of multicomponent fluids},
  vol.~135, Springer Science \& Business Media, 2006.

\bibitem{dumbser2013high}
{\sc M.~Dumbser and W.~Boscheri}, {\em High-order unstructured {L}agrangian
  one-step {WENO} finite volume schemes for non-conservative hyperbolic
  systems: applications to compressible multi-phase flows}, Comput. Fluids, 86
  (2013), pp.~405--432.

\bibitem{dumbser2009ader}
{\sc M.~Dumbser, M.~Castro, C.~Par{\'e}s, and E.~F. Toro}, {\em {ADER} schemes
  on unstructured meshes for nonconservative hyperbolic systems: Applications
  to geophysical flows}, Comput. Fluids, 38 (2009), pp.~1731--1748.

\bibitem{dumbser2011simple}
{\sc M.~Dumbser and E.~F. Toro}, {\em A simple extension of the {O}sher
  {R}iemann solver to non-conservative hyperbolic systems}, SIAM J. Sci.
  Comput., 48 (2011), pp.~70--88.

\bibitem{fisher2013high}
{\sc T.~C. Fisher and M.~H. Carpenter}, {\em High-order entropy stable finite
  difference schemes for nonlinear conservation laws: Finite domains}, J.
  Comput. Phys., 252 (2013), pp.~518--557.

\bibitem{franquet2012runge}
{\sc E.~Franquet and V.~Perrier}, {\em {R}unge--{K}utta discontinuous
  {G}alerkin method for the approximation of {B}aer and {N}unziato type
  multiphase models}, J. Comput. Phys., 231 (2012), pp.~4096--4141.

\bibitem{fraysse2016upwind}
{\sc F.~Fraysse, C.~Redondo, G.~Rubio, and E.~Valero}, {\em Upwind methods for
  the {B}aer--{N}unziato equations and higher-order reconstruction using
  artificial viscosity}, J. Comput. Phys., 326 (2016), pp.~805--827.

\bibitem{gallouet2004numerical}
{\sc T.~Gallou{\"e}t, J.-M. H{\'e}rard, and N.~Seguin}, {\em Numerical modeling
  of two-phase flows using the two-fluid two-pressure approach}, Math. Models
  Methods Appl. Sci., 14 (2004), pp.~663--700.

\bibitem{gassner2013SBP-SAT}
{\sc G.~J. Gassner}, {\em A skew-symmetric discontinuous {G}alerkin spectral
  element discretization and its relation to {SBP}-{SAT} finite difference
  methods}, SIAM J. Sci. Comput., 35 (2013), pp.~A1233--A1253.

\bibitem{gassner2014kinetic}
{\sc G.~J. Gassner}, {\em A kinetic energy preserving nodal discontinuous
  {G}alerkin spectral element method}, Int. J. Numer. Methods Fluids, 76
  (2014), pp.~28--50.

\bibitem{gassner2016split}
{\sc G.~J. Gassner, A.~R. Winters, and D.~A. Kopriva}, {\em Split form nodal
  discontinuous {G}alerkin schemes with summation-by-parts property for the
  compressible {E}uler equations}, J. Comput. Phys., 327 (2016), pp.~39--66.

\bibitem{giordano2006richtmyer}
{\sc J.~Giordano and Y.~Burtschell}, {\em {R}ichtmyer-{M}eshkov instability
  induced by shock-bubble interaction: Numerical and analytical studies with
  experimental validation}, Phys. Fluids, 18 (2006), p.~036102.

\bibitem{gottlieb2001strong}
{\sc S.~Gottlieb, C.-W. Shu, and E.~Tadmor}, {\em Strong stability-preserving
  high-order time discretization methods}, SIAM review, 43 (2001), pp.~89--112.

\bibitem{haas1987interaction}
{\sc J.-F. Haas and B.~Sturtevant}, {\em Interaction of weak shock waves with
  cylindrical and spherical gas inhomogeneities}, J. Fluid Mech., 181 (1987),
  pp.~41--76.

\bibitem{harten1983upstream}
{\sc A.~Harten, P.~D. Lax, and B.~V. Leer}, {\em On upstream differencing and
  {G}odunov-type schemes for hyperbolic conservation laws}, SIAM review, 25
  (1983), pp.~35--61.

\bibitem{hiltebrand2014entropy}
{\sc A.~Hiltebrand and S.~Mishra}, {\em Entropy stable shock capturing
  space--time discontinuous {G}alerkin schemes for systems of conservation
  laws}, Numer. Math., 126 (2014), pp.~103--151.

\bibitem{hiltebrand2018entropy}
{\sc A.~Hiltebrand, S.~Mishra, and C.~Par{\'e}s}, {\em Entropy-stable
  space--time {DG} schemes for non-conservative hyperbolic systems}, ESAIM:
  M2AN, 52 (2018), pp.~995--1022.

\bibitem{hou2007solutions}
{\sc S.~Hou and X.-D. Liu}, {\em Solutions of multi-dimensional hyperbolic
  systems of conservation laws by square entropy condition satisfying
  discontinuous {G}alerkin method}, J. Sci. Comput., 31 (2007), pp.~127--151.

\bibitem{houim2011low}
{\sc R.~W. Houim and K.~K. Kuo}, {\em A low-dissipation and time-accurate
  method for compressible multi-component flow with variable specific heat
  ratios}, J. Comput. Phys., 230 (2011), pp.~8527--8553.

\bibitem{hu2006conservative}
{\sc X.~Y. Hu, B.~Khoo, N.~A. Adams, and F.~Huang}, {\em A conservative
  interface method for compressible flows}, J. Comput. Phys., 219 (2006),
  pp.~553--578.

\bibitem{ismail2009affordable}
{\sc F.~Ismail and P.~L. Roe}, {\em Affordable, entropy-consistent {E}uler flux
  functions ii: Entropy production at shocks}, J. Comput. Phys., 228 (2009),
  pp.~5410--5436.

\bibitem{jameson2008formulation}
{\sc A.~Jameson}, {\em Formulation of kinetic energy preserving conservative
  schemes for gas dynamics and direct numerical simulation of one-dimensional
  viscous compressible flow in a shock tube using entropy and kinetic energy
  preserving schemes}, J. Sci. Comput., 34 (2008), pp.~188--208.

\bibitem{jiang1994cell}
{\sc G.~S. Jiang and C.-W. Shu}, {\em On a cell entropy inequality for
  discontinuous {G}alerkin methods}, Math. {C}omput., 62 (1994), pp.~531--538.

\bibitem{johnsen2006implementation}
{\sc E.~Johnsen and T.~Colonius}, {\em Implementation of {WENO} schemes in
  compressible multicomponent flow problems}, J. Comput. Phys., 219 (2006),
  pp.~715--732.

\bibitem{kawai2011high}
{\sc S.~Kawai and H.~Terashima}, {\em A high-resolution scheme for compressible
  multicomponent flows with shock waves}, Int. J. Numer. Methods. Fluids, 66
  (2011), pp.~1207--1225.

\bibitem{kopriva2010quadrature}
{\sc D.~A. Kopriva and G.~Gassner}, {\em On the quadrature and weak form
  choices in collocation type discontinuous {G}alerkin spectral element
  methods}, J. Sci. Comput., 44 (2010), pp.~136--155.

\bibitem{kuya2018kinetic}
{\sc Y.~Kuya, K.~Totani, and S.~Kawai}, {\em Kinetic energy and entropy
  preserving schemes for compressible flows by split convective forms}, J.
  Comput. Phys., 375 (2018), pp.~823--853.

\bibitem{lax1960systems}
{\sc P.~Lax and B.~Wendroff}, {\em Systems of conservation laws}, Comm. Pure
  Appl. Math., 13 (1960), pp.~217--237.

\bibitem{lefloch2014numerical}
{\sc P.~G. LeFloch and S.~Mishra}, {\em Numerical methods with controlled
  dissipation for small-scale dependent shocks}, Acta Numerica, 23 (2014),
  pp.~743--816.

\bibitem{liu2018entropy}
{\sc Y.~Liu, C.-W. Shu, and M.~Zhang}, {\em Entropy stable high order
  discontinuous galerkin methods for ideal compressible {MHD} on structured
  meshes}, J. Comput. Phys., 354 (2018), pp.~163--178.

\bibitem{pares2006numerical}
{\sc C.~Par{\'e}s}, {\em Numerical methods for nonconservative hyperbolic
  systems: a theoretical framework.}, SIAM J. Numer. Anal., 44 (2006),
  pp.~300--321.

\bibitem{perthame1996positivity}
{\sc B.~Perthame and C.-W. Shu}, {\em On positivity preserving finite volume
  schemes for euler equations}, Numer. Math., 73 (1996), pp.~119--130.

\bibitem{quirk1996dynamics}
{\sc J.~J. Quirk and S.~Karni}, {\em On the dynamics of a shock--bubble
  interaction}, J. Fluid Mech., 318 (1996), pp.~129--163.

\bibitem{renac2019entropy}
{\sc F.~Renac}, {\em Entropy stable dgsem for nonlinear hyperbolic systems in
  nonconservative form with application to two-phase flows}, J. Comput. Phys.,
  382 (2019), pp.~1--26.

\bibitem{renac2020multicomp}
{\sc F.~Renac}, {\em Entropy stable, robust and high-order {DGSEM} for the
  compressible multicomponent {E}uler equations}, submitted,  (2020).

\bibitem{renac2015aghora}
{\sc F.~Renac, M.~de~la Llave~Plata, E.~Martin, J.~B. Chapelier, and
  V.~Couaillier}, {\em Aghora: A High-Order DG Solver for Turbulent Flow
  Simulations}, Springer International Publishing, Cham, 2015, pp.~315--335.

\bibitem{rhebergen2008discontinuous}
{\sc S.~Rhebergen, O.~Bokhove, and J.~J. van~der {V}egt}, {\em Discontinuous
  {G}alerkin finite element methods for hyperbolic nonconservative partial
  differential equations}, J. Comput. Phys., 227 (2008), pp.~1887--1922.

\bibitem{saurel1999multiphase}
{\sc R.~Saurel and R.~Abgrall}, {\em A multiphase {G}odunov method for
  compressible multifluid and multiphase flows}, J. Comput. Phys., 150 (1999),
  pp.~425--467.

\bibitem{saurel2003multiphase}
{\sc R.~Saurel, S.~Gavrilyuk, and F.~Renaud}, {\em A multiphase model with
  internal degrees of freedom: application to shock--bubble interaction}, J.
  Fluid Mech., 495 (2003), pp.~283--321.

\bibitem{shu1988efficient}
{\sc C.-W. Shu and S.~Osher}, {\em Efficient implementation of essentially
  non-oscillatory shock-capturing schemes}, J. Comput. Phys., 77 (1988),
  pp.~439--471.

\bibitem{sun2018entropy}
{\sc Z.~Sun, J.~A. Carrillo, and C.-W. Shu}, {\em An entropy stable high-order
  discontinuous {G}alerkin method for cross-diffusion gradient flow systems},
  arXiv preprint arXiv:1810.03221,  (2018).

\bibitem{tadmor1987numerical}
{\sc E.~Tadmor}, {\em The numerical viscosity of entropy stable schemes for
  systems of conservation laws. {I}}, Math. {C}omput., 49 (1987), pp.~91--103.

\bibitem{terashima2009front}
{\sc H.~Terashima and G.~Tryggvason}, {\em A front-tracking/ghost-fluid method
  for fluid interfaces in compressible flows}, J. Comput. Phys., 228 (2009),
  pp.~4012--4037.

\bibitem{tokareva2010hllc}
{\sc S.~Tokareva and E.~F. Toro}, {\em {HLLC}-type {R}iemann solver for the
  {B}aer--{N}unziato equations of compressible two-phase flow}, J. Comput.
  Phys., 229 (2010), pp.~3573--3604.

\bibitem{wintermeyer2017entropy}
{\sc N.~Wintermeyer, A.~R. Winters, G.~J. Gassner, and D.~A. Kopriva}, {\em An
  entropy stable nodal discontinuous galerkin method for the two dimensional
  shallow water equations on unstructured curvilinear meshes with discontinuous
  bathymetry}, J. Comput. Phys., 340 (2017), pp.~200--242.

\bibitem{winters2016affordable}
{\sc A.~R. Winters and G.~J. Gassner}, {\em Affordable, entropy conserving and
  entropy stable flux functions for the ideal {MHD} equations}, J. Comput.
  Phys., 304 (2016), pp.~72--108.

\bibitem{zhang2010positivity}
{\sc X.~Zhang and C.~Shu}, {\em On positivity-preserving high order
  discontinuous {G}alerkin schemes for compressible {E}uler equations on
  rectangular meshes}, J. Comput. Phys., 229 (2010), pp.~8918--8934.

\bibitem{zhang2010maximum}
{\sc X.~Zhang and C.-W. Shu}, {\em On maximum-principle-satisfying high order
  schemes for scalar conservation laws}, J. Comput. Phys., 229 (2010),
  pp.~3091--3120.

\end{thebibliography}

\end{document}